\documentclass{amsart}
\usepackage{amsmath}
\usepackage{amsfonts}
\usepackage{graphics}
\usepackage{epsfig}
\usepackage{amssymb}
\usepackage{amscd}
\usepackage[all]{xy}
\usepackage{latexsym}
\usepackage{graphicx}
\usepackage{multirow}
\usepackage{geometry}
\usepackage{color}
\usepackage[usenames,dvipsnames]{pstricks}
\usepackage{pst-grad} 
\usepackage{pst-plot} 
\usepackage{textcomp}

\newtheorem{thm}{Theorem}[section]

\newtheorem{cor}[thm]{Corollary}
\newtheorem{lem}[thm]{Lemma}
\newtheorem{prop}[thm]{Proposition}
\newtheorem{conj}[thm]{Conjecture}
\newtheorem{quest}[thm]{Question}

\theoremstyle{definition}
\newtheorem{Def}[thm]{Definition}
\newtheorem{rem}[thm]{Remark}

\newtheorem*{ack}{Acknowledgement}
\newtheorem{ex}[thm]{Example}

\numberwithin{equation}{subsection}
\numberwithin{figure}{section}

\def\Hom{{\text{\rm{Hom}}}}
\def\End{{\text{\rm{End}}}}
\def\trace{{\text{\rm{trace}}}}
\def\tr{{\text{\rm{tr}}}}
\def\rchi{{\hbox{\raise1.5pt\hbox{$\chi$}}}}
\def\Aut{{\text{\rm{Aut}}}}

\def\isom{\cong}

\def\tensor{\otimes}
\def\dsum{\oplus}

\def\deg{{\text{\rm{deg}}}}

\def\a{\alpha}
\def\b{\beta}
\def\o{\omega}
\def\O{\Omega}
\def\S{\Sigma}
\def\sig{\sigma}
\def\lam{\lambda}

\def\gam{\gamma}
\def\Gam{\Gamma}
\def\eps{\epsilon}
\def\nab{\nabla}

\def\Ext{{\text{\rm{Ext}}}}

\def\Dol{{\text{\rm{Dol}}}}
\def\deR{{\text{\rm{deR}}}}

\def\Mat{{\text{\rm{Mat}}}}

\def\Vect{{\rm{\bf{Vect}}}}

\def\two{{\bullet \!\!\frac{\hskip0.2in}{}\!\!\bullet}}
\def\three{{\bullet \!\!\frac{\hskip0.2in}{}\!\!\!\bullet
\!\!\!\frac{\hskip0.2in}{}\!\!\bullet}}

\newcommand{\bea}{\begin{eqnarray}}
\newcommand{\eea}{\end{eqnarray}}
\newcommand{\be}{\begin{equation}}
\newcommand{\ee}{\end{equation}}

\newcommand{\Mbar}{{\overline{\mathcal{M}}}}

\newcommand{\bP}{{\mathbb{P}}}
\newcommand{\bC}{{\mathbb{C}}}

\newcommand{\bH}{{\mathbb{H}}}
\newcommand{\bQ}{{\mathbb{Q}}}
\newcommand{\bR}{{\mathbb{R}}}

\newcommand{\bZ}{{\mathbb{Z}}}
\newcommand{\cA}{{\mathcal{A}}}

\newcommand{\cM}{{\mathcal{M}}}

\newcommand{\cC}{{\mathcal{C}}}
\newcommand{\CG}{{\mathcal{C}\mathcal{G}}}
\newcommand{\cD}{{\mathcal{D}}}
\newcommand{\cE}{{\mathcal{E}}}
\newcommand{\cF}{{\mathcal{F}}}

\newcommand{\cK}{{\mathcal{K}}}
\newcommand{\cL}{{\mathcal{L}}}
\newcommand{\cO}{{\mathcal{O}}}

\newcommand{\cS}{{\mathcal{S}}}

\newcommand{\cU}{{\mathcal{U}}}

\newcommand{\cW}{{\mathcal{W}}}

\newcommand{\la}{{\langle}}
\newcommand{\ra}{{\rangle}}
\newcommand{\half}{{\frac{1}{2}}}

\newcommand{\rar}{\rightarrow}
\newcommand{\lrar}{\longrightarrow}

\begin{document}
\large

\allowdisplaybreaks

\title[Invitation to 
2D TQFT and quantization]
{An invitation to 2D TQFT and
quantization of Hitchin spectral curves}

\author[O.\ Dumitrescu]{Olivia  Dumitrescu}
\address{
Olivia Dumitrescu:
Department of Mathematics\\
Central Michigan University\\
Mount Pleasant, MI 48859, U.S.A.\\
and Simion Stoilow Institute of Mathematics\\
Romanian Academy\\
21 Calea Grivitei Street\\
010702 Bucharest, Romania}
\email{dumit1om@cmich.edu}

\author[M.\ Mulase]{Motohico Mulase}
\address{Motohico Mulase:
Department of Mathematics\\
University of California\\
Davis, CA 95616--8633, U.S.A.\\
and
Kavli Institute for Physics and Mathematics of the 
Universe\\
The University of Tokyo\\
Kashiwa, Japan}
\email{mulase@math.ucdavis.edu}

\begin{abstract}
This article consists of two parts. In Part 1,
we present a  formulation of two-dimensional
topological quantum field theories in terms  of
a functor from 
a category of
Ribbon graphs to the endofuntor category of 
a monoidal category. The key point is  that the category of
ribbon graphs produces all  
Frobenius objects. 
Necessary backgrounds from Frobenius algebras, topological
quantum field theories, and cohomological field
theories are  reviewed. A  result on
Frobenius algebra
 twisted topological recursion is included at 
the end of Part 1.

In Part 2, we explain a 
geometric theory of
 quantum curves. The focus 
is placed on  the process of quantization as
a passage from families of 
Hitchin spectral curves to families of opers. 
To make the presentation simpler, we unfold the
story using $SL_2(\bC)$-opers and 
rank $2$ Higgs bundles
defined on 
a compact Riemann surface $C$ of genus greater than $1$. 
In this case, 
quantum curves,
opers, and  projective structures  
in $C$ all become  the same notion. 
Background materials on projective coordinate systems, 
Higgs bundles, opers, and non-Abelian Hodge
correspondence are explained. 
\end{abstract}

\subjclass[2010]{Primary: 14H15, 14N35, 81T45;
Secondary: 14F10, 14J26, 33C05, 33C10, 
33C15, 34M60, 53D37}

\keywords{Topological quantum field 
theory; quantum curves; opers;
Hitchin moduli spaces; Higgs bundles; 
Hitchin section;
quantization; topological recursion.}

\maketitle
\tableofcontents

\setcounter{section}{-1}

\section{Preface: An inspiration from the past}
A recent discovery of a cuneiform tablet
dated  
around 350 to 50 B.C.E.\ 
suggests that ancient Babylonians  must have used
geometry of
\emph{time-momentum space} 
to establish accurate calculations of Jupiter's orbit
\cite{O}. This impressive paper also contains
the picture of the tablet which describes the
Babylonian's method of \emph{integration}.

\begin{figure}[hbt]
\includegraphics[height=2in]{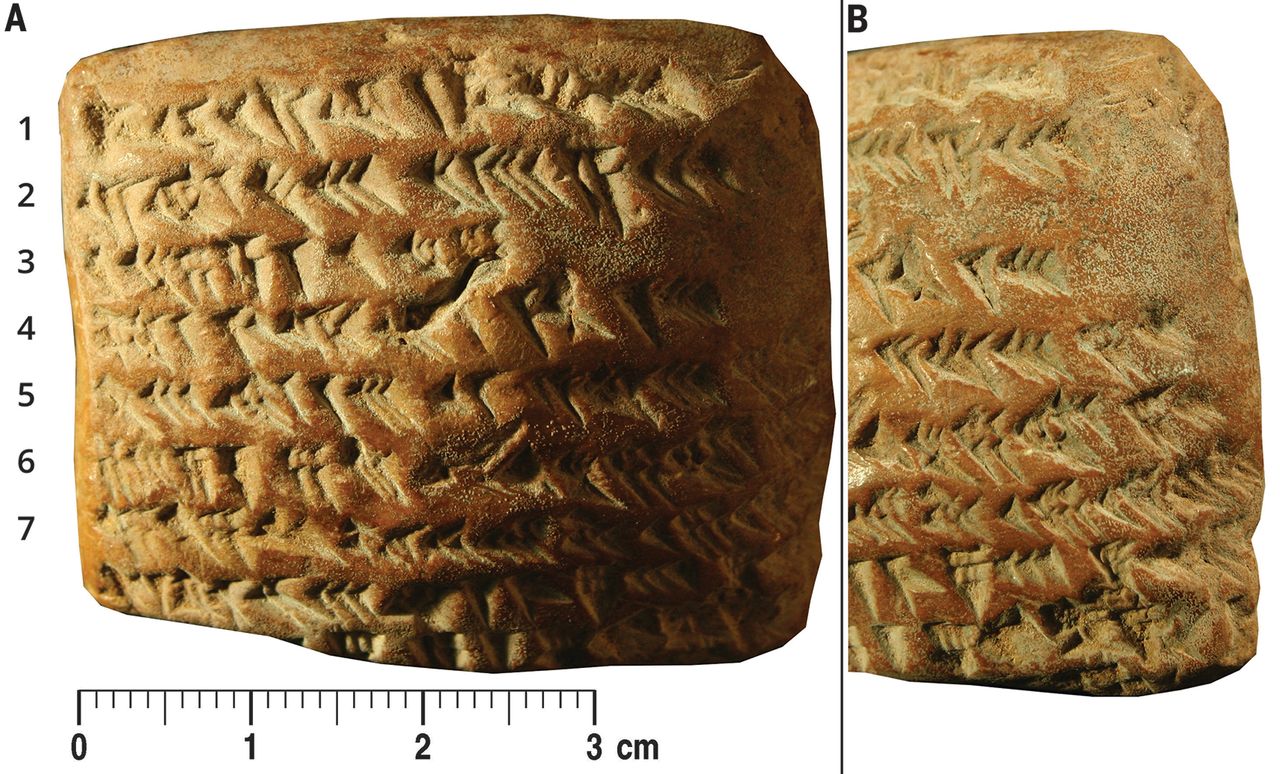}
\caption{The cuneiform tablet used in the analysis
of \cite{O}.}
\label{fig:cuneiform}
\end{figure}

The orbit of the Jupiter is a graph in  
coordinate space-time. 
By considering the graph of  momentum
of the Jupiter
in  time-momentum space, Babylonians 
\emph{visualized}  integral of the momentum by the 
area underneath the curve, 
and using a trapezoidal approximation, they
actually obtained an estimated  value of the integral.
This gives a prototype of Newton's Fundamental 
Theorem of Calculus. 
This idea of Babylonians relating  
geometry of  time-momentum space
with the analysis of actual orbit of the Jupiter
is striking, 
because it suggests their equal 
treatment of  coordinate space  the momentum 
space.  Although it is a stretch, 
we could imagine the very foundation of 
\emph{symplectic geometry} here.

Many mathematical cuneiform 
tablets recording numbers and algebraic calculations
have been our source of imagination.
The most famous is  \textbf{Plimpton 322}  
of around 1,800 B.C.E.\ (see Figure~\ref{fig:Plimton322}).
It lists 15 Pythagorean numbers in the increasing
order of hypotenuse angles from about 
45 degrees to 60 degrees \cite{Robson}.

\begin{figure}[hbt]
\includegraphics[height=2in]{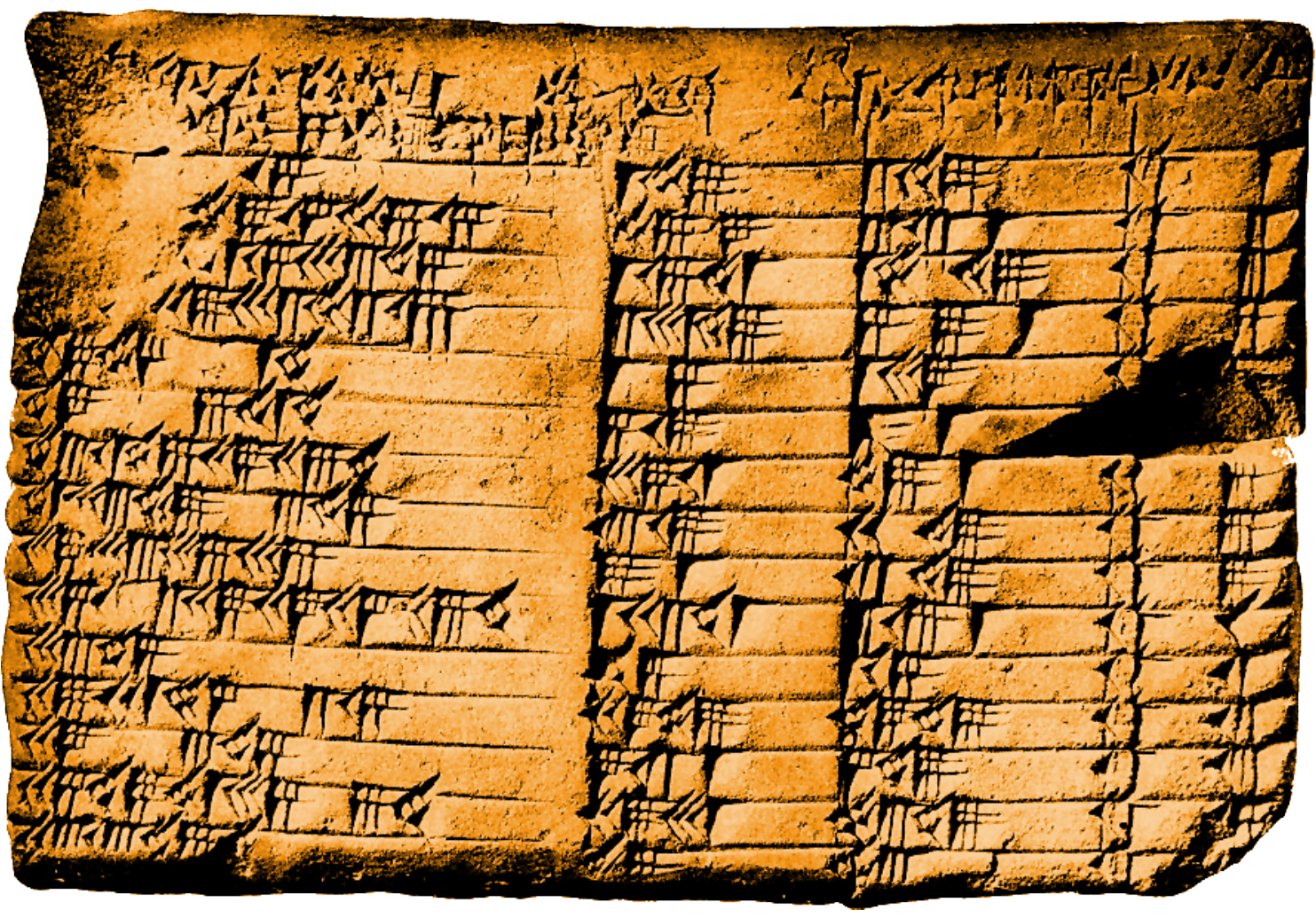}
\caption{Plimton 322}
\label{fig:Plimton322}
\end{figure}

Babylonians seem to have known
an algorithm  to calculate an approximate value 
of the
square root of any  number. 
For example, there is a cuneiform 
tablet that shows the sexagesimal expansion 
of $\sqrt{2}$. 
Although there have been
many speculations for  practical purposes of
Plimton 322 and  mechanisms to come up with
the listed numbers, our imagination
goes to the surprise of
the creator of the tablet. For the pair of numbers
$(y,z)$
 listed
in the second and the third columns, 
$z^2-y^2$ is always a perfect square. Therefore,  the 
square root algorithm terminates in a finite number of
steps
for these values, and gives an exact 
answer. They must have found \emph{a finiteness in the
forest of infinity}. This is a strong sentiment 
that resonates with our 
mind today.

The tablet of Figure~\ref{fig:cuneiform}
does not show any geometry. 
The author of \cite{O} convincingly 
argues that behind the list of 
these seemingly meaningless numbers on the
tablets, there is
a profound geometric investigation  of the
planetary motion. Our imagination is piqued by 
the discovery: the interplay between algebra,
geometry, and astronomy; and 
the equal treatment of  momentum space
and  coordinate space. 
What the feeling of the authors
 of the tablet  would have 
 been, when they were creating it? 
 It must have been  a sharp happiness of  
 discovery that  mathematics 
 correctly  predicts the very nature surrounding us,
 such as planetary motion.
 
 The accuracy of  calculations of Babylonians is
 another astonishment, in addition to 
 the lack of pictures. It is in sharp contrast to Euclid's \textbf{Elements}, which 
 was written around the same time in Greece. In \emph{Elements},
 we see 
many beautiful geometric
pictures, and  \emph{proofs}. The discovery of \emph{axiomatic method}
 culminated in it. This dichotomy,
  being able to \textbf{calculate} a quantity to a high
 precision, vs. 
  having a \textbf{proof} of the formula
 based on a \textbf{finiteness property}, is 
 our very motivation of writing these notes. 
 There are a lot of amazing formulas around
 us in the direction that we describe.
 At this moment, we do not have the final understanding
 of them, yet.

In the first part of these lectures, we explore 
two-dimensional topological quantum field theories
formulated  in terms of
\emph{cell graphs}. A cell graph of type $(g,n)$
is the $1$-skeleton 
of a cell-decomposition of a compact oriented 
topological surface of genus $g\ge 0$ with
$n$ labeled $0$-cells, or vertices.
Such graphs are called in many different names:
ribbon graphs,
dessins d'enfants, maps, embedded graphs, etc. 
We use the terminology ``cell graph'' indicating the different
nature of the  graphs on surfaces, 
which record degeneration 
of surfaces.
Although a cell graph is a $1$-dimensional object,
the notion of a \emph{face} is well defined as a $2$-cell
of the cell decomposition that the cell 
graph defines on a given surface.

The degree of a vertex of a cell graph is 
the number of incident half-edges.  
We call a $1$-cell an edge of the graph. 
As explained in 
our previous lecture notes \cite{OM4}, 
the number of cell graphs of type $(0,1)$  
with the unique vertex of degree $2m$ is
$\frac{1}{2m}C_m$, where 
$$C_m= \frac{1}{m+1}\binom{2m}{m}
$$
is the $m$-th \emph{Catalan number}. 
A generating function 
$$
z = z(x) = \sum_{m=0} ^\infty C_m\frac{1}{x^{2m+1}}
$$
of Catalan numbers satisfies an algebraic equation
\be\label{z}
x = z + \frac{1}{z}.
\ee
The story of \cite[Section 2]{OM4} tells us that 
enumeration 
problem of cell graphs of arbitrary type $(g,n)$
is solved by the \textbf{quantization}
of the algebraic curve \eqref{z}. 
The key formula for the quantization comes from
the Laplace transform of a combinatorial equation
 \cite[Catalan Recursion, Section 2.1]{OM4}.
There, the combinatorial formula is derived
by analyzing the effect of \textbf{edge-contraction 
operations} on cell graphs. 

The new story we wish to tell in these lectures
is that the \emph{category} of 
cell graphs carries the information of
\emph{all} two-dimensional  topological
quantum field theories. In  
Part 1 of this article, we will present an axiomatic setup 
of  2D TQFT.
The key idea is simple: exact same edge-contraction
operations characterize Frobenius algebras. 
When we contract an 
edge of a cell graph, two vertices collide. It 
represents \emph{multiplication}. 
When we shrink a loop, since it
is a  cycle on the topological surface,
the process breaks a vertex into
two vertices. This is the operation of
\emph{comultiplication}. The topological structure of 
a cell graph makes these operations  dual
to one another in the context of Frobenius algebras.

We start with defining Frobenius 
algebras. Topological quantum field theories
(TQFT) and cohomological field theories (CohFT)
are introduced in the following sections. 
A category of cell graphs is  defined. 
We explain 
that this category generates all Frobenius algebras.
We then make a connection between cell graphs
and  2D TQFT. The theory of \emph{topological recursion}
is a  leitmotiv
 in this article. It is another manifestation
 of a quantization procedure. 
Since there are many review articles  on
topological recursion, we touch the subject only
tangentially in this article. Our new result 
related to this topic is the
2D TQFT-twisted topological recursion, which is
 explained in the end of Part 1.

The cohomology ring of a closed oriented manifold
$X$
is a Frobenius algebra. When we consider an 
even dimensional manifold and only even degree 
cohomologies
\be\label{A}
A = H^{even}(X,\bQ),
\ee
then we have a commutative Frobenius 
algebra, which is equivalent to a 2D TQFT. 
The  \emph{Gromov-Witten
invariants} of $X$
of genus $0$
determine a quantum deformation of the ring $A$,
known as the \emph{big quantum cohomology}. 
This quantum structure then defines a 
holomorphic object $Y$, the \emph{mirror} of $X$.
Thus the holomorphic geometry of $Y$ captures 
quantum cohomology of $X$. 
Gromov-Witten invariants are generalized to 
arbitrary genera. Here arises a question:

\begin{quest}
What should be the holomorphic geometry on $Y$
that corresponds to higher-genus Gromov-Witten 
theory of $X$?
\end{quest}

If we consider the passage from genus $0$ to higher
genera Gromov-Witten theory
as \emph{quantization}, then 
on the side on $Y$, we need
a quantized holomorphic geometry. A good candidate
is the $\cD$-module theory. The simplest such theory
fitting to this context is the notion of \textbf{quantum curves}. 
The relation between quantum curves and B-model 
geometry is considered  in \cite{ADKMV, DHSV,DFM,
OM1,OM2,OM4,GHM,GS} and others. 
The key point for this idea to work is when the 
holomorphic geometry $Y$ corresponding to the
Gromov-Witten theory of $X$ is captured by 
an algebraic or analytic \emph{curve}. These curves often 
appear in other areas of mathematics as
\emph{spectral curves}, such as integrable systems,
random matrix theory, and Hitchin's theory of Higgs
bundles. A spectral curve naturally comes with  
a projection to a \emph{base curve}, making it a
covering of the base curve. 
 A quantum curve is a $\cD$-module
 on this base curve, quantizing the spectral curve. 

Since a CohFT based on the Frobenius algebra 
$H^{even}(X,\bQ)$
plays a role similar to Gromov-Witten
theory of $X$, CohFT is a quantization of 2D TQFT.
This process consists of two steps 
of quantizations. The first one is from the classical 
cohomology ring $A$ to its
quantum deformation, which is the mirror 
geometry. If the mirror is a curve, then 
the second step should  be parallel to 
the quantization of a spectral curve to a quantum curve.
We may be able to 
 understand the reconstruction \cite{DOSS} of
CohFT based on a semi-simple Frobenius algebra from the 
2D TQFT, that is obtained as the degree $0$ part of the
cohomology $H^0(\Mbar_{g,n},\bQ)$ of the 
moduli space of stable curves, analogous to 
the construction of quantum curves.

Part 2 deals with quantum curves. Since 
the authors have produced a long article \cite{OM4}
explaining the relation between quantum 
curves and topological recursion, 
the present article is focused on 
geometry of the process of
quantization of Hitchin spectral curves
 from a  perspective of \textbf{opers}. Instead of
 providing a  
 general theory of
 \cite{O-Paris, DFKMMN, OM5},
 we use $SL_2(\bC)$ to explain our ideas here. 
 We will show that the classical notion of 
 projective structures on a compact Riemann
 surface $C$ of genus $g\ge 2$ studied by 
 Gunning \cite{Gun}, $SL_2(\bC)$-opers, and
 quantum curves of Hitchin spectral curves, are indeed
 the exact same notion. 
  A conjecture
 of Gaiotto \cite{Gai} relating non-Abelian 
 Hodge correspondence and opers, together with its
 solution by \cite{DFKMMN},  will   be briefly 
  explained in the final section.

\begin{ack}
The authors are grateful to the organizers of the
\textit{Warsaw Advanced School on Topological Quantum Field Theory} for providing the opportunity to produce 
this article. They are in particular indebted to 
Piotr Su\l kowski for constant encouragement
and patience. 

The joint research of authors presented in 
this article is 
  carried out  while they  
  have been staying in the following institutions
  in the last four years:
the American Institute of Mathematics in 
California, 
the Banff International Research Station,
Institutul de
 Matematic\u{a} ``Simion Stoilow'' 
 al Academiei Rom\newtie{a}ne,
 Institut Henri Poincar\'e,
  Institute for Mathematical
Sciences at the National University of Singapore,
 Kobe University, 
 Leibniz Universit\"at Hannover, 
 Lorentz Center Leiden,
 Mathematisches Forschungsinstitut Oberwolfach,
 Max-Planck-Institut f\"ur Mathematik-Bonn,
 and
  Osaka City University Advanced
 Mathematical Institute.
 Their generous financial 
 support, hospitality, and stimulating research
 environments are greatly appreciated. 

The research of
 O.D.\ has been  supported by
 GRK 1463 \emph{Analysis,
Geometry, and String Theory} at the 
Leibniz Universit\"at 
 Hannover, Perimeter Institute for Theoretical
 Physics in  Waterloo,
 and  a grant from 
 MPIM-Bonn.
The research of M.M.\ has been  supported 
by IH\'ES, MPIM-Bonn, Simons Foundation, 
Hong Kong University of Science and Technology,
Universit\'e Pierre et Marie Curie, and
NSF grants DMS-1104734, DMS-1309298, 
DMS-1619760, DMS-1642515,
and NSF-RNMS: Geometric Structures And 
Representation Varieties (GEAR Network, 
DMS-1107452, 1107263, 1107367).
\end{ack}


\part{Topological Quantum Field Theory}

\section{Frobenius algebras}
\label{sect:Frobenius}

Throughout this article, we denote by $K$ a 
field of characteristic $0$. Most of the cases we
consider $K=\bQ$ or $K=\bC$. 
Let $A$ be a finite-dimensional, 
unital, and associative algebra 
defined over a field $K$. A \emph{bialgebra}
comes with an extra set of structures, including a
comultiplication. Examples 
of  bialgebras include
the group algebra $A=K[G]$ of a finite 
group $G$. 
Its algebra structure is canonically determined by
the multiplication rule of $G$. However, 
the group algebra has \emph{two distinct} co-algebra
structures. One of which makes $K[G]$ a Frobenius
algebra, and the other makes it a Hopf algebra. 
This difference is reflected in which topological invariants
we are dealing with. The Frobenius structure is 
useful for two-dimensional topology, and the Hopf
structure is natural for three-dimensional topological 
invariants. 

Let us start with defining Frobenius algebras. 

\begin{Def}[Frobenius algebras] 
\label{def:Frobenius}
A finite-dimensional, unital, and associative
algebra $A$ over a field $K$ is a \textbf{Frobenius
algebra} if there exists a linear map 
$\eps:A\lrar K$ called
\textbf{counit} such that for every
$u\in A$, $\eps(ux) = 0$, or
$\eps(xu) = 0$, for all $x\in A$ implies that 
$u=0$.
\end{Def}

\begin{prop}
If $A_1$ and $A_2$ are Frobenius algebras, then
$A_1\dsum A_2$ and $A_1\tensor A_2$ are 
also Frobenius algebras. 
Let us denote by $\cA$ the category of 
Frobenius algebras defined over $K$. 
Then $\cC = (\cA,\tensor,K)$ is a monoidal 
category. 
\end{prop}

\begin{proof}
Denote by $\eps_i:A_i\lrar K$ the 
counit of $A_i$, $i=1,2$. 
Then $(A_1\dsum A_2,\eps_1+\eps_2)$
is a Frobenius algebra. Similarly, 
$$
\eps_1\cdot \eps_2:A_1\tensor A_2
\owns u_1\tensor u_2 \longmapsto \eps_1(u_1)\eps_2(u_2)
\in K
$$
makes 
$A_1\tensor A_2$ a Frobenius algebra.
\end{proof}

\begin{ex} The one-dimensional vector 
space $A=K$, with the
identity map $\eps(u) = u$, is a simple commutative
Frobenius algebra. More generally, by taking
the direct sum $A=K^{\dsum n}$ and defining
$\eps(u_1,\dots,u_n) = u_1+\cdots +u_n\in K$, 
we construct a \emph{semi-simple} commutative
Frobenius algebra $K^{\dsum n}$. 
\end{ex}

\begin{ex}
The full matrix algebra $A=\Mat_n(K)$ consisting of
$n\times n$ matrices with $\eps(u) = \trace(u)$ 
is an example of a non-commutative simple 
Frobenius algebra for $n\ge 2$. 
\end{ex}

\begin{ex}
As mentioned above, the group algebra $A=K[G]$ of 
a finite group $G$ is a Frobenius algebra. Here, we define
$\eps:K[G]\lrar K$ by linearly extending the map
whose value for every $g\in G$ is given by
$$
\eps(g) = \begin{cases}
1 \qquad g=1\\
0 \qquad g\ne 1.
\end{cases}
$$
\end{ex}

\begin{rem}
If we define $\eps(g) = 1$ for all $g\in G$, then the 
group algebra becomes a \textbf{Hopf algebra}. 
\end{rem}

\begin{ex}
The cohomology ring $H^*(M,\bR)$ of an oriented compact 
differential manifold $M$ of $\dim_{\bR} M =n$
with the cup product and 
$$
\eps:H^*(M,\bR)\lrar H^n(M,\bR)=\bR
$$
is a Frobenius algebra.
\end{ex}

The above example makes us wonder if an analogue of
\emph{Poincar\'e pairing} exists in a general 
 Frobenius algebra. Indeed, the counterpart is  a 
\textbf{Frobenius bilinear form} 
defined by
\be
\label{eta}
\eta:A\tensor A\lrar K, \qquad 
\eta(u,v) = \eps (uv)
\ee
in terms of the counit $\eps:A\lrar K$.
The Frobenius form satisfies the Frobenius associativity 
\be
\label{Frobenius form}
\eta(uv,w) = \eta(u,vw), \qquad u,v,w\in A.
\ee
The condition for 
the counit $\eps$ of Definition~\ref{def:Frobenius}
exactly means that
the Frobenius form $\eta$ is non-degenerate. 
It therefore defines a canonical isomorphism 
\be
\label{lam}
\lam:A\overset{\sim}{\lrar}
 A^*, \qquad \la\lam(u),v\ra 
=\eta(u,v), \qquad u,v\in A.
\ee
The  isomorphism  introduces
a unique  
\textbf{comultiplication} 
$\delta:A\lrar A\tensor A$ by
  the
following commutative diagram:
\be
\label{delta}
\quad
\xymatrix{
A\ar[d]_\lam\ar[rr]^\delta && 
A\tensor A\ar[d]^{\lam\tensor\lam \;\;\;  .} 
\\
A^*\ar[rr]_{m^*}&& A^*\tensor A^*
}
\ee
Here, $m^*:A^*\lrar A^*\tensor A^*$ is the 
dual of the multiplication operation
$m:A\tensor A\lrar A$ in $A$. 
Since we are not assuming the multiplication to be
commutative, we define the natural pairing
$$
(A^*\tensor A^*)\tensor (A\tensor A) \lrar K
$$
 by observing the order of entities as written. 
For example, for $\a,\b\in A^*$ and $u,v\in A$,
we calculate 
\be
\label{tensor pairing}
\la\a\tensor \b,u\tensor v\ra = \la \a,\la \b,u\ra v\ra = \la \b,u\ra
\la\a,v\ra.
\ee
As the dual of the associative multiplication,
the comultiplication $\delta$ is \emph{coassociative}, i.e.,
it satisfies 
\be
\label{coassociative}
\xymatrix{
&A\tensor A\ar[dr]^{\delta\tensor 1}&\\
{\phantom{A\tensor}} A {\phantom{\tensor A}} \ar[ur]^{\delta}\ \ar[dr]_{\delta}& 
  & A\tensor A\tensor A .\\
&A\tensor A \ar[ur]_{1 \tensor \delta}
		}
\ee
Note that the identity element $\mathbf{1}\in A$ 
corresponds to $\eps\in A^*$ by $\lam$, i.e., 
$\lam(\mathbf{1})=\eps$. This is because
$$
\eta(\mathbf{1},u) = \eta(u,\mathbf{1}) = \eps (u).
$$
We now have the full set of data 
$(A,\mathbf{1},m,\eps,\delta)$
that defines a bialgebra. The algebra and coalgebra
structures satisfy a compatibility condition, as described
below.

\begin{prop}
\label{prop:m delta}
 The following diagram commutes:
\be
\label{m delta}
\xymatrix{
&A\tensor A\tensor  A\ar[dr]^{m\tensor 1}&\\
A\tensor A \ar[ur]^{1\tensor \delta}\ar[r]^{\;\;\; m} \ar[dr]_{\delta\tensor 1}& 
A \ar[r]^{\delta \;\;\;\;} & A\tensor A .\\
&A\tensor A\tensor A \ar[ur]_{1 \tensor m}
		}
\ee
\end{prop}

\begin{rem}
Alternatively, we can  define  a Frobenius algebra
as a bialgebra satisfying \eqref{m delta}, 
together with a non-degenerate Frobenius form
satisfying \eqref{Frobenius form}.
\end{rem}

\begin{proof}
Let  $\la e_1,e_2,\dots,e_r\ra$ be a $K$-basis for
$A$, where $r = \dim_K A$. 
The bilinear form $\eta$  
defines $r\times r$  matrices
\be
\label{eta-1}
\eta = [\eta_{ij}], 
\qquad \eta^{-1}=[\eta^{ij}],
\qquad
\eta_{ij}:=\eta(e_i,e_j).
\ee
It gives the canonical basis 
expansion of $v\in A$:
\be
\label{complete set}
v
= \sum_{a,b} \eta(v,e_a)\eta^{ab}e_b = 
\sum_{a,b} \eta(e_a,v)\eta^{ba}e_b.
\ee
From \eqref{delta}, 
we calculate 
\begin{align*}
\delta(v) &= 
\sum_{i,j,a,b}
 \eta(v,e_ie_j) (\eta^{jb}e_b)
 \tensor 
(\eta^{ia} e_a)
\\
&=
\sum_{i,j,a,b}
 \eta(ve_i,e_j) (\eta^{jb}e_b)
 \tensor 
(\eta^{ia} e_a)
\\
&=
\sum_{i,a}
 (ve_i) 
 \tensor 
(\eta^{ia} e_a)
\\
&=
(m\tensor 1) (v\tensor \delta(\mathbf{1})),
\end{align*}
 using the pairing convention \eqref{tensor pairing}
and
\be
\label{delata 1}
\delta(\mathbf{1}) = \sum_{a,b}\eta^{ab}e_a\tensor e_b.
\ee
We then have
\begin{align*}
(1\tensor m)\circ (\delta\tensor 1)(u\tensor v) &=
(1\tensor m)\left(\sum_{i,j,a,b}
 \eta(u,e_ie_j) (\eta^{jb}e_b)
 \tensor 
(\eta^{ia} e_a) \tensor v\right)
\\
&=
\sum_{i,j,a,b}
 \eta(ue_i,e_j) (\eta^{jb}e_b)
 \tensor 
(\eta^{ia} e_a v) 
\\
&=
\sum_{i,j,a,b,c,d}
 \eta(ue_i,e_j) (\eta^{jb}e_b)
 \tensor 
\eta^{ia} \eta(e_a v,e_c) \eta^{cd}e_d
\\
&=
\sum_{i,a,c,d}
 ue_i
 \tensor 
\eta^{ia} \eta(e_a, ve_c) \eta^{cd}e_d
\\
&=
\sum_{c,d}
 (uve_c) \tensor\eta^{cd}e_d
\\
&=
(m\tensor 1)(uv\tensor \delta(\mathbf{1}))
=\delta(uv).
\end{align*}
Similarly,
\begin{align*}
(m\tensor 1)\circ (1\tensor \delta)(u\tensor v) &=
(m\tensor 1)\left(u\tensor \sum_{i,j,a,b}
 \eta(v,e_ie_j) (\eta^{jb}e_b)
 \tensor 
(\eta^{ia} e_a) \right)
\\
&=
\sum_{i,j,a,b}
u \eta(ve_i,e_j) (\eta^{jb}e_b)
 \tensor 
(\eta^{ia} e_a ) 
\\
&=
\sum_{i,a}
uve_i
 \tensor 
\eta^{ia} e_a
\\
&=
(m\tensor 1)(uv\tensor \delta(\mathbf{1}))
=\delta(uv).
\end{align*}
This completes the proof of Proposition~\ref{prop:m delta}.
\end{proof}

\begin{rem}
We note that if $A$ is  commutative, 
then the quantities 
\be
\label{symmetric}
\eta\big(e_{i_1}\cdots
e_{i_{j}},e_{i_{j+1}}\cdots e_n\big)
=\epsilon(e_{i_1}\cdots e_{i_n}),
\qquad 1\le j<n,
\ee
are completely symmetric with respect to permutations
of  indices. 
\end{rem}

\begin{Def}[Euler element]
The \textbf{Euler element} of a Frobenius 
algebra $A$ is defined by
\be\label{Euler}
\mathbf{e}:= m\circ \delta(\mathbf{1}).
\ee
In terms of basis, the Euler element is given by
\be\label{Euler basis}
\mathbf{e} = \sum_{a,b}\eta^{ab} e_ae_b.
\ee
\end{Def}

The Euler element provides the genus expansion 
of 2D TQFT, allowing us to calculate 
higher genus correlation functions from the
genus $0$ part of the theory.

Another application of \eqref{complete set}
is the following formula that relates 
the multiplication and comultiplication:
\be\label{prod=coprod}
\left(\lam(u)\tensor 1\right)\delta(v)
= uv.
\ee
This is because
\begin{align*}
\left(\lam(u)\tensor 1\right)\delta(v)
&=
\left(\lam(u)\tensor 1\right)
\sum_{a,b}
(v\eta^{a b}
e_a)\tensor e_b
\\
&=
\sum_{a,b}
\eta(u,ve_a)\eta^{ab}
 e_b
\\
&=
\sum_{a,b}
\eta(uv,e_a)\eta^{ab}
 e_b
=uv.
\end{align*}


\section{TQFT}
\label{sect:TQFT}

In this section, we briefly review TQFT. 
Although there have been 
explosive mathematical
developments in higher dimensional topological 
quantum field
theories mixing different dimensions
during the last decade (see for example, \cite{CN, L}
and more recent work inspired by them),
we restrict our attention to the two-dimensional
speciality in these lectures. 
From now on, Frobenius algebras we consider are 
finite-dimensional,
unital, associative, and \emph{commutative}.
It has been established (see
\cite{Abrams,Dijkgraaf})
that  2D TQFT's are classified by these types of 
Frobenius algebras.

The axiomatic formulation of conformal and 
topological quantum field theories was discovered 
in the 1980s by Atiyah \cite{Atiyah} and
Segal \cite{Segal}. A $(d-1)$-dimensional
TQFT is a monoidal 
functor $Z$ from the monoidal  
category  of $(d-1)$-dimensional closed (i.e., 
compact without boundary)
oriented smooth manifolds with oriented $d$-dimensional
cobordism as morphisms, to 
the monoidal category of finite-dimensional vector spaces
defined over a field $K$. 
The monoidal structure in the category
of  $(d-1)$-dimensional 
smooth manifolds is defined by the operation 
of taking disjoint union, which is a symmetric operation.
 Disjoint union with the empty set is the
identity  of this operation.
Therefore, we define 
the monoidal category of $K$-vector spaces 
 by \emph{symmetric tensor products}, with the 
field $K$ serving as the identity operation of 
tensor product. The functor $Z$
satisfies the non-triviality condition
$$
Z(\emptyset) = K, 
$$
and maps a 
disjoint union of smooth $(d-1)$-dimensional
manifolds  to a symmetric tensor product of vector spaces.

Let us denote by $M$  an oriented closed smooth manifold
of dimension $d-1$, and by $M^{op}$ the same 
manifold with the opposite orientation. 
The functor $Z$ satisfies the duality condition
$$
Z(M^{op}) = Z(M)^*,
$$
where $Z(M)^*$ is the dual vector space of $Z(M)$. 
Suppose  we have an oriented 
bordered $d$-dimensional smooth manifold $N$. 
The boundary 
$\partial N$ is a smooth manifold
of dimension $d-1$, and the complement 
$N\setminus \partial N$ is an oriented smooth
manifold of dimension $d$. We give the  orientation 
induced from $N$ to its boundary $\partial N$. 
The TQFT functor $Z$ then gives an element 
$$
Z(N) \in Z(\partial N).
$$
If $N$ is closed, i.e., $\partial N = \emptyset$, then 
$$
Z(N) \in Z(\emptyset) = K
$$
is a number that represents a topological invariant of $N$.

Now consider a bordered oriented smooth manifold
$N_1$ with  boundary
$$
\partial N_1 = M_1^{op}\sqcup M_2,
$$
meaning that the two separate boundaries
carry different orientations. We choose that
 $M_2$ is given the induced orientation 
from $N_1$, and $M_1$ the opposite orientation. 
We interpret the situation
as $N_1$ giving an oriented  cobordism 
from $M_1$ to $M_2$.
In this case, the functor $Z$ defines an element
$Z(N_1) \in Z(M_1)^*\tensor Z(M_2)$, or equivalently,
a linear map
$$
Z(N_1):Z(M_1)\lrar Z(M_2).
$$
Suppose we have another smooth manifold $N_2$
with boundary
$$
\partial N_2 = M_2^{op}\sqcup M_3
$$
corresponding to a linear map
$$
Z(N_2):Z(M_2)\lrar Z(M_3).
$$
We can then smoothly glue
$N_1$ and $N_2$ along the common
boundary component $M_2$ forming a new
manifold 
$$
N = N_1 \cup_{M_2} N_2.
$$
Clearly 
$$
\partial N = M_1^{op}\sqcup M_3,
$$
and $N$ gives a cobordism from $M_1$ to $M_3$. 
The Atiyah-Segal sewing axiom \cite{Atiyah}
asserts that 
\be
\label{sewing}
\xymatrix{
Z(M_1) \ar[r]^{Z(N_1)} \ar[d]_{\parallel}&Z(M_2) 
\ar[r]^{Z(N_2)} &Z(M_3)\;  \ar[d] ^{\parallel}
 \\
Z(M_1)\ar[rr]^{Z(N)} && Z(M_3).
		}
\ee

 A $d$-dimensional TQFT $Z$ defines a topological 
 invariant $Z(N)$ for each closed $d$-manifold $N$. 
 The essence of TQFT is to break $N$ into 
 simpler pieces 
 by cutting
 along $(d-1)$-dimensional submanifolds, and 
represent the invariant $Z(N)$ from that of simpler pieces.  
The situation is special for the case of 2D TQFT.
First of all, we already know \emph{all}
topological invariants  of a closed surface. 
They are just functions of the genus $g$ of the surface. 
So there should be no reason for another theory to study them.
Since there is only one $1$-dimensional connected compact
smooth manifold, a 2D TQFT is based on a single 
vector space 
\be
\label{ZS1}
Z(S^1) = A,
\ee
and its tensor products. What else  could we gain?

The biggest surprise is that 
\textbf{a 2D TQFT is further quantized}
\cite{CK,DZ, Givental, KM}. This process starts 
with a Frobenius algebra $A$ of \eqref{A}. 
It is quantized to a quantum cohomology, and then 
further  
to Gromov-Witten invariants. 
Construction of
spectral curves, and then quantizing them,
have a parallelism to this story.  
This is the story we now explore
in  this  and the next sections,
but only to take a snapshot of the \emph{converse}
direction, i.e., from a quantum theory to a classical
theory.

Let us denote by $\S_{g,\bar{m},n}$ a 
connected, bordered, oriented,
smooth surface of genus $g$ with $m+n$ boundary circles. 
This is a surface obtained
by removing $m+n$ disjoint open discs from a 
compact
oriented two-dimensional  smooth manifold 
of genus $g$ without boundary. We assume that
the boundary of $\S_{g,\bar{m},{n}}$ itself is a 
smooth manifold, hence it is a disjoint union 
of $m+n$ circles.  
We give the induced  orientation 
from the surface to  $n$ boundary 
circles, and  the opposite orientation to the
$m$ boundary circles. The 
 surface $\S_{g,\bar{m},{n}}$   
gives a cobordism of  $m$ circles to $n$ circles. 
The orientation-preserving diffeomorphism class of
such a surface is determined by the genus $g$ and the
two number $m$ and $n$ of boundary circles with 
different orientations. Therefore, the oriented equivalence
class of cobordism is also determined by $(g,m,n)$. 
The TQFT functor $Z$ then assigns to 
each cobordism  $\S_{g,\bar{m},{n}}$
  a multilinear map
$$
\o_{g,\bar{m},{n}} \overset{\text{def}}{=}
Z(\S_{g,\bar{m},{n}}) :A^{\tensor m}\lrar
A^{\tensor n},
$$
which is completely determined by the label 
$(g,m,n)$. This is the special situation of the 
two-dimensionality of 
 TQFT, reflecting the simple topological 
classification of surfaces.

Suppose  we have another cobordism 
$\S_{h,\bar{k},m}$  of genus $h$
from 
$k$ circles to $m$ circles,
with  orientation on
the $m$ circles induced from the surface, and
 the opposite orientation on $k$ circles.
We can \emph{compose} two cobordisms,
sewing the $m$ circles of the first cobordism with 
the $m$ circles of the second.  Here, we 
notice that  on each
pair of circles, one from the first surface
and the other from the second, the orientations are 
the same, and hence we can put one on top of
 the other.
 Therefore, the orientations
of the two surfaces are consistent after sewing. This 
sewing process
generates
 a new surface $\S_{g+h+m-1,\bar{k},n}$
 of genus $g+h+m -1$. This is  because
 the $m$ pairs of circles sewed together create
 $m-1$ handles (see Figure~\ref{fig:sewing}).
The \textbf{sewing axiom} of Atiyah-Segal \cite{Atiyah}
then requires that the functor $Z$ associates the 
composition of linear maps to the composition
of cobordisms. 
In our situation,  the composition
of two maps generates a new map
$$
\o_{g,\bar{m},n}
\circ 
\o_{h,\bar{k},m}
=
\o_{g+h+m-1,\bar{k},n}
:A^{\tensor k}\lrar A^{\tensor n},
$$
corresponding to the sewing of cobordisms.

\begin{figure}[htb]
\includegraphics[width=2in]{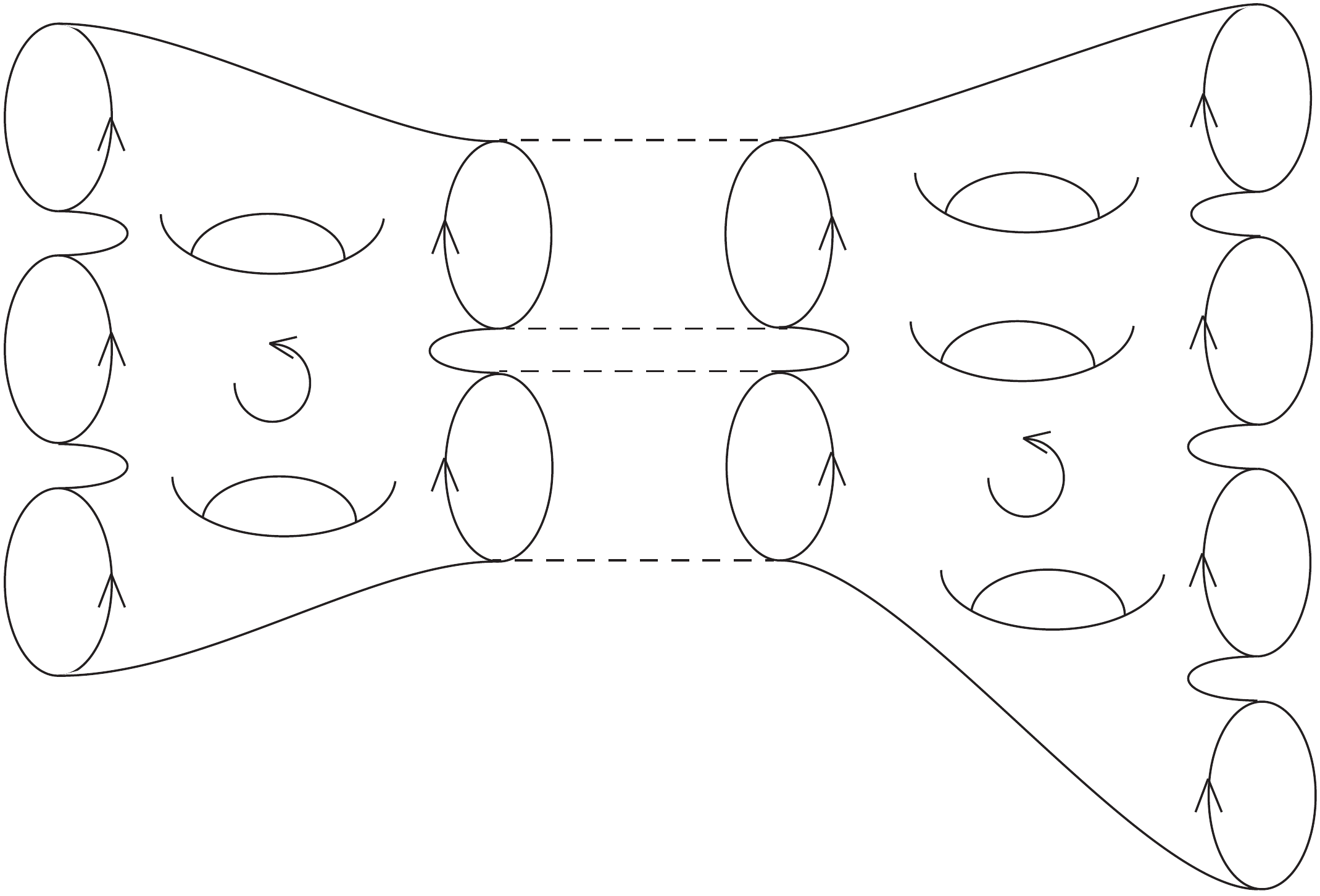}
\caption{A string of \textbf{String Interactions}.
The physics of a simple model for
 interacting strings
is captured by the sewing axiom of Atiyah-Segal.}
\label{fig:sewing}
\end{figure}

The sewing procedure can be generalized, allowing
partial sewing of the boundary circles of matching 
orientation. For example, if $j\le \ell$ and $j\le m$,
then gluing only $j$ pairs of circles, we have a composition
\be\label{partial sewing}
\o_{g,\bar{m},n}
\circ 
\o_{h,\bar{k},\ell}
=
\o_{g+h+j-1,\overline{k+m-j},n+\ell-j}
:A^{\tensor k+m-j}\lrar A^{\tensor n+\ell-j}.
\ee
Furthermore, since TQFT does not require cobordism 
to be given by a  \emph{connected} manifold, we can 
stack up disjoint union of  surfaces and 
apply partial sewing to 
create a variety of linear maps from $A^{\tensor n}$
to $A^{\tensor m}$.

It was Dijkgraaf \cite{Dijkgraaf}
who  noticed the equivalence between 2D TQFTs 
and commutative Frobenius algebras. We can see
the resemblance immediately. On the vector
space $A$, we have operations defined by
\be\begin{aligned}
\label{operations}
&\mathbf{1}=\o_{0,\bar{0},1}:K\lrar A,
&&m=\o_{0,\bar{2},1} :A^{\tensor 2}\lrar A,
\\
& \eps = \o_{0,\bar{1},0}:A\lrar K,
&&\delta = \o_{0,\bar{1},2}:A\lrar A^{\tensor 2},
\\
&\eta = \o_{0,\bar{2},0}:A^{\tensor 2}\lrar K.
\end{aligned}
\ee
A connected cobordism of 
$m$ incoming circles and $n$ outgoing circles 
is classified by the genus $g$ of the surface. It does
not depend on the history of how the cobordisms 
are glued together. 
Thus the diffeomorphism classes of surfaces after 
sewing cobordisms tell us relations
among the operations of \eqref{operations}.  For example,
$$
\o_{0,\bar{1},0}\circ \o_{0,\bar{2},1} =
\o_{0,\bar{2},0} \Longrightarrow \eps\circ m = \eta,
$$
which is  \eqref{eta}. Associativity of 
multiplication comes from uniqueness of the
topology of $\Sigma_{0,\bar{3},1}$ with one boundary
circle
carrying the induced orientation and three  the opposite
orientation. Changing the orientation of each boundary
component to its 
opposite provides coassociativity. In this way 
$A = Z(S^1)$ acquires a bialgebra structure.
To assure duality between algebra and coalgebra structures,
we need to impose another condition here:
non-degeneracy of $\eta=\o_{0,\bar{2},0}$. 
Then $A$ becomes
  a commutative Frobenius algebra. 
  
  Conversely, if we start with a finite-dimenionsal
  commutative
  Frobenius algebra $A$, then we first 
  construct $\o_{0,\bar{m},n}$ on the list of 
  \eqref{operations}. More general 
  maps $\o_{g,\bar{m},n}$ are constructed by
  partial sewing \eqref{partial sewing}. 
  By  construction, all these maps are associated with 
  cobordism of circles, and hence determine a 
  2D TQFT \cite{Abrams}.
  
  With the above  considerations, we give the following
  definition.
  
\begin{Def}[Two-dimensional
 Topological Quantum Field Theory]
 Let $(A,\eps,\lam)$ be a set of data 
 consisting of a 
 finite-dimensional vector space
 $A$ over $K$, a non-trivial linear map $\eps:A\lrar K$,
 and an isomorphism $\lam:A\overset{\sim}{\lrar}A^*$. 
 A \textbf{Two-dimensional Topological 
 Quantum Field Theory} is a system 
 $(A,\{\o_{g,\bar{m},n}\})$
consisting  of  
 linear maps 
 $$
 \o_{g,\bar{k},\ell}:A^{\tensor k}\lrar
A^{\tensor \ell}, \qquad 0\le g,k,\ell,
 $$
 satisfying the following axioms.
 \begin{itemize}
 \item \textbf{TQFT 1. Symmetry:} 
  \be\label{o-gkl}
 \o_{g,\bar{k},\ell}:A^{\tensor k}\lrar
A^{\tensor \ell}
 \ee 
 is symmetric with respect to the $S_k$-action on the 
 domain. 
 \item \textbf{TQFT 2. Non-triviality:}
  \be\label{non-triviality}
 \o_{0,\bar{1},{0}}=\eps: A\lrar K.
 \ee
 \item  \textbf{TQFT 3. Duality:}
  \be\label{duality}
  \xymatrix{
A^{\tensor m} \ar[d]_{\lam^{\tensor m}}
\ar[rr]^{\o_{g,\bar{m},n}} &&A^{\tensor n}  \ar[d] ^{\lam^{\tensor n} \;.}
 \\
(A^*)^{\tensor m}\ar[rr]^{(\o_{g,\bar{n},m})^*} && (A^*)^{\tensor n}
		}
 \ee
 \item \textbf{TQFT 4. Sewing:}
  \be\label{p-sewing}
 \o_{g,\bar{m},n}
\circ 
\o_{h,\bar{k},\ell}
=
\o_{g+h+j-1,\overline{k+m-j},n+\ell-j}
:A^{\tensor k+m-j}\lrar A^{\tensor n+\ell-j},
 \quad j\le m, \; j\le \ell.
\ee
 \end{itemize}
\end{Def}

\begin{rem}
As noted above, the vector space $A$ acquires 
the structure of a commutative Frobenius algebra
from these axioms. Conversely, if $A$ is a commutative
Frobenius algebra, then $\o_{0,\bar{m},n}$ of
\eqref{operations} can be extended to $\o_{g,\bar{m},n}$
that satisfy the TQFT axioms. 
\end{rem}


\section{Cohomological field theory}
\label{sect:CohFT}

Recall that for any oriented closed manifold $X$,
its cohomology ring $H^*(X,K)$ is a Frobenius 
algebra defined over $K$. If we restrict ourselves to
even dimensional manifold $X$ and consider only
even part of the cohomology,
then  $A=H^{even}(X,K)$ is a finite-dimensional
commutative Frobenius algebra. 
This Frobenius algebra naturally defines a
2D TQFT. The role of a TQFT in general dimension
is to represent a topological invariant of higher dimensional
manifolds. We are now seeing  that 
the \emph{classical} topology
of a manifold $X$ is a 2D TQFT.
Changing our point of view, we can ask
if we start with $X$, then how do we find the
corresponding 2D TQFT? Of course the Frobenius
algebra structure automatically determine 
the unique 2D TQFT as we have seen in 
Section~\ref{sect:TQFT}. But then the 
picture of sewing cobordisms is lost in this algebraic
formulation. What is the role of the cobordism
of circles in the context of understanding the
manifold $X$?

The amazing vision emerged in the early 1990s is that 
2D TQFT can be further \emph{quantized} into
\textbf{Gromov-Witten theory}, which produces 
\emph{quantum} topological invariants of $X$ that cannot
be reduced to classical invariants. 
This epoch making discovery was one of 
the decisive moments of 
 the fruitful interaction of string theory and
geometry, which is still continuing today.

String theory deals  with flying strings 
in a space-time manifold $X$. The trajectory of 
a moving string is a curved cylinder
 embedded in the
manifold $X$, like a duct pipe in the attic. 
 When  strings are considered as
\emph{quantum} objects, they can interact one another.
For example, Figure~\ref{fig:sewing} can 
be interpreted as  a \emph{string} of
 \textbf{string interactions}. First, three strings 
 collide in a complicated interaction to produce 
 two strings. The complexity of the interaction is
 classified as ``$g=2$'' in terms of the genus of
 the first surface. These two strings created in the
 first interaction then collide in an even more
 complicated (i.e., $g=3$) interaction to produce 
 four strings in the end. The trajectory of these interactions
 then forms a bordered surface of genus $6$ with
 $7$ boundary components embedded in the 
 space-time $X$.
 
 Models for string theory were originally introduced
 to understand how the space-time $X$ dictates 
 string interactions. In this process, the importance
 of Calabi-Yau spaces  was recognized,
 through  considerations on the consistency of 
 string theory as a physical theory. 
 By changing the point of view $180^\circ$, 
 researchers then noticed that a simple model of 
 string interactions could be used to obtain 
 a totally new class of
 topological invariants of $X$ itself. 
 These new invariants are not 
 necessarily captured by 
 the classical topology, such as (co)homology groups
 and homotopy groups. We use the terminology of
 \textbf{quantum topological invariants}
 for the invariants obtained by string theory 
 considerations. 
 
 Gromov-Witten theory is a mathematically rigorous
 interpretation of a model for a quantum string theory.
 A trajectory of interacting strings in $X$ is the image of 
 a map 
 \be\label{string map}
 f: \S_{g,\bar{m},n}\lrar X
 \ee
 from a bordered surface $\S_{g,\bar{m},n}$
 into $X$. Recall that the fundamental group
 $\pi_1(X,x_0)$ of $X$ is determined by the connectivity
 of the pointed loop space
 $$
 \cL(X,x_0) = \cC^\infty\big((S^1,*),(X,x_0)\big),
 $$
 which is the \emph{moduli space} of differentiable 
 maps from 
 $S^1$ to $X$ that send a reference
 point $*\in S^1$ to $x_0\in X$. In the same spirit,
 Gromov-Witten invariants are defined by looking at
 the \emph{classical}
 topological invariants of appropriate 
 moduli spaces of maps from $\S_{g,\bar{m},n}$
 to $X$.
 When we have the notion of \emph{size}
  of a string,  
 the trajectory $\S_{g,\bar{m},n}$ has a metric on it. 
 If the ambient manifold $X$ has a geometric structure,
 such as a symplectic structure, complex structure, 
 or a K\"ahler structure, then the map $f$ needs to 
 be compatible with the structures of the source and the
 target. For example, if $X$ is K\"ahler, 
 then we give a complex structure on $\S_{g,\bar{m},n}$
 that is compatible with the metric it has, and 
 require that $f$ is a holomorphic map.
 Technical difficulties arise  in defining
  moduli spaces of such maps. Even the moduli spaces are
 defined, they are often very different 
 objects from the usual
 differentiable manifolds. We need to extend 
 the notion of manifolds here. Thus identifying 
 reasonable classical topological invariants of 
 these spaces  also poses a difficult problem.
 
 The simplest scenario is the following: We take $X=pt$ 
 to be just a single point. Of course nobody wants to 
 know the topological structure of a point. We all know it!
 By exploring the Gromov-Witten theory of a point,
 we learn the structure of the theory itself. Since we 
 can map \emph{anything} to a point, a point may not be
 such a simple object, after all. 
 It is like considering a \emph{vacuum} in physics. 
 Again, anything can be thrown into a vacuum. 
 A quantum theory of a vacuum is inevitably a rich
 theory.
 
 When $X=pt$ is a point, we consider 
 all incoming and outgoing strings to be infinitesimally
 small. Thus the surfaces we are considering become
 closed, and boundary components
 are just several points identified on them.
 A metric on a closed surface naturally gives rise to 
 a unique complex structure, making it a 
 compact \emph{Riemann surface}. And every
 compact Riemann surface acquires a unique
 \emph{projective algebraic} structure, making it a projective
 algebraic curve. 
 Since we are allowing strings to become infinitesimally
 small during the interaction, the trajectory may 
 contain a process of an embedded $S^1$ in the surface
 shrinking to a point, and then becoming a finite 
 size circle again a moment later. 
 Such a process can be understood in complex algebraic 
 geometry as a \textbf{nodal singularity}
 of an algebraic curve. Locally, every nodal singularity
 of a curve is the same as the neighborhood of the 
 origin of a singular algebraic curve
 $$
 \{(x,y)\in \bC^2\;|\;xy = 0\}.
 $$
 We are talking about an abstract notion of 
 trajectories here, because nothing can move in 
 $X=pt$. In terms of the map idea \eqref{string map}, 
 we note that there is a \emph{unique} map $f$
 from any projective algebraic curve to a point, which
 satisfies any reasonable requirement of maps such as
 being a morphism of algebraic varieties. 
 Then the moduli space of all maps simply means
 the moduli space of the source. 
 
 A \textbf{stable curve} is a projective algebraic
 curve with only nodal singularities and 
 a finite number of smooth points marked on the curve
 such that it has only a finite number of algebraic
  automorphisms fixing each
  marked point. We recall that the holomorphic automorphism 
 group of $\bP^1$, 
 $$
 \Aut(\bP^1) = PSL_2(\bC),
 $$
 acts on $\bP^1$ triply transitively. 
 The number of automorphisms
 fixing  up to $2$ points is always infinity.
 Since $PSL_2(\bC)$ is compact and 
 three-dimensional, if we 
 choose three distinct points on $\bP^1$ and 
 require the automorphism to fix each of these points,
 then we have only finitely many choices. Actually 
 in our case, it is unique.  
 For an elliptic curve $E$, since it is an abelian group, 
 $E$ acts on $E$ transitively. To avoid these translations,
 we need to choose a point on $E$. We can naturally
 identify it as the identity element of the elliptic curve
as a group. Automorphisms of an elliptic curve fixing 
a point then form a finite group. 
If a  compact Riemann surface $C$ has genus
$g=g(C)\ge 2$, then it is known that the order
of the analytic automorphism group is bounded by
$$
|\Aut(C)|\le 84(g-1).
$$
This bound comes from hyperbolic geometry. 
It is easy to see the finiteness. First, we note that
 universal covering of $C$ is the upper half plane
$$
\bH = \{z\in \bC\;|\;Im(z)>0\},
$$
and $C$ is constructed by the quotient
$$
C \overset{\sim}{\lrar} \bH/\rho(\pi_1(C))
$$
through a faithful representation 
\be\label{rho}
\rho:\pi_1(C)\lrar PSL_2(\bR) = \Aut(\bH)
\ee
of the fundamental group of $C$ into the automorphism 
group of $\bH$. 
Every holomorphic automorphism 
of $C$  extends to an 
automorphism of $\bH$. 
We know that $C$ does not have
any non-trivial holomorphic vector field $v$. 
If it did, then $v$ would be a differentiable 
vector field with isolated zeros, and each zero
comes with positive index, because locally 
it is given by $z^n d/dz$. We learn from topology 
that the sum of the indices of isolated 
zeros of a vector field on $C$ is equal to $\chi(C)
= 2-2g<0$. It is a contradiction. Thus $\Aut(C)\subset 
PSL_2(C)$
is a discrete subgroup. 
Since $PSL_2(\bR)$ is compact, $|\Aut(C)|$ is   finite. 

An algebraic curve $C$
is stable if and only if (1)  every singularity is nodal,
and (2) 
every irreducible component $C'$ of $C$
has a finite number of 
automorphisms. The second condition
 means that the total number of
smooth marked points and  singular points of $C$
that are on $C'$  has to be $3$ or more if 
$g(C')=0$, and $1$ or more if $g(C') = 1$. 
There is no condition for an irreducible component of
genus two or more.

 For a pair $(g,n)$ of integers $g\ge 0$ and $n\ge 1$
in the stable range $2g-2+n>0$, we denote by
 $\Mbar_{g,n}$  the moduli space
of stable curves of genus $g$ and $n$ smooth
marked points. It is a complex orbifold of
dimension $3g-3+n$. Quantum topological 
invariants of a point $X=pt$ is then realized as 
classical topological invariants of $\Mbar_{g,n}$.
Alas to this day, we cannot identify the cohomology 
groups $H^i(\Mbar_{g,n},\bQ)$ for all values of 
$i, g, n$.
From the very definition of these moduli spaces,
we can construct many concrete cohomology classes
on each of $\Mbar_{g,n}$,
called \emph{tautological}
classes. What we do not know is 
indeed how much more we need to know to determine
all of $H^i(\Mbar_{g,n},\bQ)$.
The surprise of Witten's conjecture \cite{W1991},
proved in \cite{K1992},
is that we can actually explicitly write 
 intersection relations of certain
tautological classes 
for \emph{all} values of $g$ and $n$.

Among many different proofs available for
the Witten conjecture (see for example,
\cite{KL,K1992,Mir1,Mir2,MZ,OP}),
\cite{Mir2,MZ} deal with recursive relations 
among  surfaces much in the same spirit of TQFT. 
These relations are first noticed in \cite{DVV}. 
We will discuss these relations in connection to 
topological recursion later in these lectures.

Geometric relations among $\Mbar_{g,n}$'s 
for different values of $(g,n)$ have a simple meaning. 
Let us denote by $\cM_{g,n}$ the moduli space
of smooth $n$-pointed curves. Then the boundary
$$
\Mbar_{g,n}\setminus \cM_{g,n}
$$ 
consists of
points representing singular curves.  Simplest 
singular stable curve has one nodal singularity,
which can be described as 
collision of two smooth points. Analyzing how 
these singularities occur via degeneration, 
we come up with three types of natural morphisms 
 among the
moduli spaces $\Mbar_{g,n}$. They are
the forgetful morphisms
\be\label{pi}
\pi:\Mbar_{g,n+1}\lrar \Mbar_{g,n}
\ee
which simply erase one of the marked points on 
a stable curve,
and gluing morphisms 
\begin{align}
\label{gl1}
&gl_1:\Mbar_{g-1,n+2}\lrar \Mbar_{g,n}
\\
\label{gl2}
&gl_2:\Mbar_{g_1,n_1+1}\times
 \Mbar_{g_2,n_2+1}\lrar \Mbar_{g_1+g_2,n_1+n_2}
\end{align}
that construct boundary strata of 
$\Mbar_{g,n}$. 
Under a gluing morphism, we put two smooth points
of stable curves together to form a one nodal
singularity. The first one $gl_1$ 
glues two points on the same curve together,
and $gl_2$ one each on two curves.

The fiber of $\pi$ at a stable curve $(C,p_1,\dots,p_n)
\in \Mbar_{g,n}$ is the curve $C$ itself, because 
another marked point can be placed
anywhere on $C$. 
Thus $\pi$ is a \emph{universal family} of curves
parameterized by the base moduli space $\Mbar_{g,n}$.
When we place the extra marked point $p_{n+1}$
at $p_i$, $i=1,\dots,n$, then as the point of
$\Mbar_{g,n+1}$ on the fiber of $\pi$, the data 
represented is a singular curve obtained by
joining $C$ itself with a $\bP^1$ at the location
of $p_i\in C$, but the two marked points $p_i,p_{n+1}$
are
actually placed on the line $\bP^1$. Assigning 
this singular curve to $(C,p_1,\dots,p_n)$ defines
a section 
$$
\sigma_i:\Mbar_{g,n}\lrar \Mbar_{g,n+1},
$$
which is a right inverse of $\pi$. Geometrically, 
$\sigma_i$ sends $(C,p_1,\dots,p_n)$ to the point $p_i$
on the fiber $C=\pi^{-1}(C,p_1,\dots,p_n)$.
Obviously, 
$$
\pi\circ \sigma_i :\Mbar_{g,n}\lrar \Mbar_{g,n}
$$
is the identity map.

 Gromov-Witten theory
  for a K\"ahler manifold
 $X$  \cite{CK} concerns topological structure of 
 the moduli space 
 $$
 \Mbar_{g,n}(X,\b) =\big\{f:
 (C, p_1,\dots,p_n)\lrar X\;\big|\; [f(C)] = \b \big\}
 $$
 of \emph{stable} holomorphic maps $f$
 from a  nodal curve $C$ with $n$ smooth marked 
 points $p_1,\dots,p_n\in C$ to $X$
 such that the homology class of the image $f(C)$
 agrees with a prescribed homology class
 $\b\in H_2(X,\bZ)$. Here, stability of a 
 map $f$ is again defined by imposing the 
 finiteness of possible automorphisms. 
 Giving a definition 
 of this moduli space
 is beyond our scope of this article. 
 We refer to \cite{CK}. The moduli space, if defined,
 should come with  natural maps
 $$
 \begin{CD}
  \Mbar_{g,n}(X,\b) @>ev_i>>X\\
  @V\phi VV\\
  \Mbar_{g,n},
 \end{CD}
 $$
 where the forgetful map 
 $\phi$  assigns the \emph{stabilization}
 of  the source $(C,p_1,\dots,p_n)$ to
 the map $f$  by forgetting about the
 map itself, and 
 $$
 ev_i\big(f:(C,p_1,\dots,p_n)\lrar X\big) 
 = f(p_i)\in X, \qquad i=1, \dots,n,
 $$
 is the value of $f$ at the $i$-th marked point
 $p_i\in C$. Stabilization  means that evey
 irreducible component of $(C,p_1,\dots,p_n)$
 that is not stable is shrunk to a point.
If we indeed know the moduli
 space $\Mbar_{g,n}(X,\b)$ and that its 
 cohomology theory behaves as
 we expect, then we would have
 $$
 \begin{CD}
  H^*\big(\Mbar_{g,n}(X,\b),\bQ\big)@<ev_i^*<<
  H^*(X,\bQ)\\
  @V \phi_! VV\\
 H^*(\Mbar_{g,n},\bQ),
 \end{CD}
 $$
 where $\phi_!$ is the \emph{Gysin map}
 defined by integration along fiber. 
 If $\Mbar_{g,n}(X,\b)$ were a manifold, 
 and the
 map $\phi:\Mbar_{g,n}(X,\b)\lrar\Mbar_{g,n}$
 were a fiber bundle,
 then $\Mbar_{g,n}(X,\b)$ would have been  locally
 a direct product, 
 hence the Gysin map $\phi_!$ associated
 with  $\phi$
 would be defined by integrating 
 de Rham cohomology classes of $\Mbar_{g,n}(X,\b)$
 along  fiber of $\phi$. 
Choose any cohomology classes
 $v_1,\dots v_n\in H^*(X,\bQ)$   of $X$. We 
 then could have defined
 the  \textbf{Gromov-Witten invariants} by 
$$
 GW^{X,\b}_{g,n}(v_1,\dots,v_n)
 := \int_{\Mbar_{g,n}} \phi_!\big(ev^*_1(v_1)
 \cdots ev^*_n(v_n)\big).
 $$
 In general, however, construction of the Gysin map does not 
 work as we hope. This is due to the complicated 
 nature of the moduli space $\Mbar_{g,n}(X,\b)$,
 which often has components of \emph{unexpected}
 dimensions. The remedy is to define the
 \emph{virtual fundamental class}
 $[\Mbar_{g,n}(X,\b)]^{vir}$ of the expected
 dimension,
 and avoid the use of $\phi_!$
 by defining
   \be\label{GW}
 GW^{X,\b}_{g,n}(v_1,\dots,v_n)
 := \int_{[\Mbar_{g,n}(X,\b)]^{vir}} ev^*_1(v_1)
 \cdots ev^*_n(v_n).
 \ee
 See \cite{CK} for more detail.

Now recall that $A=H^{even}(X,\bQ)$ is a Frobenius algebra. 
Although Gromov-Witten theory goes through
a big black box $[\Mbar_{g,n}(X,\b)]^{vir}$, 
what we wish is a map
$$
\O_{g,n}:H^{even}(X,\bQ)^{\tensor n}
\lrar H^*(\Mbar_{g,n},\bQ)
$$
whose integral 
over $\Mbar_{g,n}$ 
gives the quantum invariants of $X$.
Then what are the properties that 
 Gromov-Witten invariants should satisfy? Can we
list the properties as axioms for the above map $\O_{g,n}$
so that we can characterize Gromov-Witten invariants?
 This was one of  the motivations of Kontsevich and 
 Manin  to introduce CohFT in \cite{KM}.

As we see below,
a 2D TQFT can be obtained as a special case of 
a CohFT. The amazing relation between 
TQFT and CohFT, i.e.,
 the \emph{reconstruction} of 
CohFT from its  restriction to TQFT
 due to Givental and Teleman
\cite{Givental, Teleman},
 plays a key role
in many new developments
 (see for example, \cite{AGP, DOSS, FLZ, 
MOPPZ}), some of which are deeply related
with topological recursion.

\begin{Def}[Cohomological Field Theory
\cite{KM}]
Let $A$ be a finite-dimensional, unital,
associative, and commutative Frobenius algebra with 
a basis $\{ e_1,\dots, e_r\}$.
A  \textbf{Cohomological 
Field Theory} is a system $(A,\{\O_{g,n}\})$
consisting of linear maps
\be\label{CohFT}
\Omega_{g,n}:A^{\tensor n}\lrar 
H^*(\Mbar_{g,n},K)
\ee
defined for $(g,n)$ in the stable range $2g-2+n>0$
and
 satisfying the following axioms:
\begin{align*}
&\textbf{CohFT 0:}\quad \Omega_{g,n} \text{ is 
$S_n$-invariant, and }
 \Omega_{0,3}(v_1,v_2,v_3) = 
 \eta(v_1v_2, v_3).
\\
&\textbf{CohFT 1:} \quad
\Omega_{g,n+1}(v_1,\dots,v_n,\mathbf{1}) =
 \pi^*\Omega_{g,n}(v_1,\dots,v_n).
 \\
&\textbf{CohFT 2:}\quad
gl_1 ^*\Omega_{g,n}(v_1,\dots,v_n)
=\sum_{a,b}\Omega_{g-1,n+2}
(v_1,\dots,v_n,e_a,e_b)\eta^{ab}.
\\
&\textbf{CohFT 3:}\quad
gl_2^*\Omega_{g_1+g_2,|I|+|J|}(v_I,v_J)
=\sum_{a,b}
\eta^{ab}\Omega_{g_1,|I|+1}(v_I,e_a)
\tensor
\Omega_{g_2,|J|+1}(v_J,e_b),
\end{align*}
where $I\sqcup J = \{1,\dots,n\}$
is a disjoint partition of the index set, and the tensor
product operation on the 
right-hand side   is performed via the K\"unneth 
formula of 
 cohomology rings 
 $$
 H^*(\Mbar_{g_1,n_1+1}\times
 \Mbar_{g_2,n_2+1},K) \isom
 H^*(\Mbar_{g_1,n_1+1},K)\tensor 
 H^*( \Mbar_{g_2,n_2+1},K).
 $$
\end{Def}

\begin{rem}
The condition $ \Omega_{0,3}(v_1,v_2,v_3) = 
 \eta(v_1v_2, v_3)$
 says that the product of the Frobenius algebra
 is determined by the $(0,3)$-value of the
 Gromov-Witten invariants. 
\end{rem}

\begin{prop}[2D TQFT is a CohFT]
Every 2D TQFT is a CohFT that takes values
in $H^0(\Mbar_{g,n},K)$. More precisely, 
let $(A, \o_{g,\bar{m},n})$ be a 2D TQFT. 
Then $\o_{g,n} = \o_{g,\bar{n},0}$ satisfies the 
CohFT axioms, by identifying $K=H^0(\Mbar_{g,n},K)$.
\end{prop}

\begin{proof}
Since $\Mbar_{g,n}$ is
connected, the three types of morphisms
\eqref{pi}, \eqref{gl1}, and \eqref{gl2} all produce
isomorphisms of degree $0$ cohomologies. Thus 
the  axioms CohFT 1--3 become
\begin{align}
\label{o-1}
&
\o_{g,n+1}(v_1,\dots,v_n,\mathbf{1}) =
 \o_{g,n}(v_1,\dots,v_n),
 \\
 \label{o-2}
&
\o_{g,n}(v_1,\dots,v_n)
=\sum_{a,b}\o_{g-1,n+2}
(v_1,\dots,v_n,e_a,e_b)\eta^{ab},
\\
\label{o-3}
&
\o_{g_1+g_2,|I|+|J|}(v_I,v_J)
=\sum_{a,b}
\eta^{ab}\o_{g_1,|I|+1}(v_I,e_a)
\cdot
\o_{g_2,|J|+1}(v_J,e_b).
\end{align}
We wish to show  that \eqref{o-1}-\eqref{o-3} are
consequences of the partial sewing
axiom  \eqref{p-sewing} of TQFT
under the identification $\o_{g,n} = \o_{g,\bar{n},0}$.

Since $\o_{0,\bar{0},1} = \mathbf{1}$, 
we have
$$
 \o_{g,\overline{n+1},0}\circ_{(n+1)}\o_{0,\bar{0},1}
=\o_{g,\bar{n},0}=\o_{g,n}:A^{\tensor n}\lrar K
$$
from \eqref{p-sewing}, which is \eqref{o-1}. Here, 
$\circ_{(n+1)}$ is the composition  taken 
at the $(n+1)$-th 
slot of the input variables of $\o_{g,\overline{n+1},0}$.
Next, we need to
 identify $\o_{0,\bar{0},2}:K\lrar A\tensor A$.
 Since 
 $\o_{0,\bar{0},2}=\o_{0,\bar{1},2}\circ\o_{0,\bar{0},1}$,
 we see that $\o_{0,\bar{0},2}(1) = \delta(\mathbf{1})
 =\sum_{a,b}\eta^{ab} e_a\tensor e_b$.
 Denoting by $\circ_{(n+1,n+2)}$ to indicate 
 composition taking at the last two slots of variables,
 we have \eqref{o-2}
 $$
 \o_{g-1,\overline{n+2},0}\circ_{(n+1,n+2)}
 \o_{0,\bar{0},2}=\o_{g,\bar{n},0}:A^{\tensor n}\lrar K.
 $$
 If we have two disjoint sets of variables $v_I$ and $v_J$,
 then we can apply composition
 of  $\o_{0,\bar{0},2}$ 
 simultaneously to two different maps. For example,
 we have
 $$
 \big(\o_{g_1,\overline{|I|+1},0}\tensor\o_{g_2,\overline{|J|+1},0}\big)\circ \o_{0,\bar{0},2} = 
 \o_{g_1+g_2,\overline{|I|+|J|},0}:
 A^{\tensor (|I|+|J|)}\lrar K,
 $$
which is \eqref{o-3}.

 Notice that a linear map
 $$
 \o_{g,m+n}:A^{\tensor m}\tensor A^{\tensor n}\lrar K
 $$ 
 is equivalent to $A^{\tensor m}\lrar (A^*)^{\tensor n}$. 
 Thus we can re-construct a map $\o_{g,\bar{m},n}$
 from $\o_{g,m+n}$
 by
 \be\label{extension}
  \xymatrix{
A^{\tensor m} \ar[d]_{\parallel}
\ar[rr]^{\o_{g,m+n}} &&(A^*)^{\tensor n}
\ar[rr]^{(\lam^{-1})^{\tensor n}}
&&A^{\tensor n}  \ar[d] ^{\parallel\; ,}
 \\
A^{\tensor m}\ar[rrrr]^{\o_{g,\bar{m},n}} 
&&&& 
A^{\tensor n}
		}
 \ee
 where $\lam:A\overset{\sim}{\lrar}A^*$ is the
 isomorphism of \eqref{lam}.
For the case of $g=0$, $m=2$, and $n=1$, 
\eqref{extension}
implies that
$$
\lam^{-1} \o_{0,3}(v_1,v_2,\;\cdot\;) = 
\o_{0,\bar{2},1}(v_1,v_2) = v_1v_2
$$
for $v_1, v_2\in A$. 
Or equivalently,  we have
\be
\label{TQFT 03}
\o_{0,3}(v_1,v_2,v_3) = \eta(v_1v_2,v_3) = \eps(v_1v_2v_3).
\ee
This completes the proof.
\end{proof}

Conversely,  the degree $0$ part of the cohomology
of a CohFT is a 2D TQFT.

\begin{prop}[Restriction of CohFT to the degree $0$ 
part of the cohomology ring]
Let $(A,\O_{g,n})$ be a CohFT
associated with a Frobenius algebra $A$. The
Frobenius algebra $A$ itself defines
a unique 2D TQFT $(A,\o_{g,\bar{m},n})$. 
Denote by 
$$
r:H^*(\Mbar_{g,n},K) \lrar H^0(\Mbar_{g,n},K)
=K
$$
  the restriction of the cohomology ring to its
degree $0$ component, and  define
\be\label{degree 0}
\o_{g,n} = r\circ \O_{g,n}:A^{\tensor n}\lrar K.
\ee
Then we have the equality of maps
\be\label{equivalence}
\o_{g,n} = \o_{g,\bar{n},0}:A^{\tensor n}\lrar K
\ee
for all $(g,n)$ with $2g-2+n>0$. In other words,
the degree $0$ restriction of a CohFT is the
2D TQFT determined by the Frobenius algebra $A$.
 \end{prop}

\begin{proof}
We already know that 
$(A,\o_{g,\bar{m},n})$ defines a CohFT with values
in $H^0(\Mbar_{g,n},K)$. We need to show that
this CohFT is exactly the degree $0$ restriction of 
the \emph{given} CohFT we start with. 

First we  extend $\o_{g,n}$ to the unstable
range by
\be\label{o-0102}
\o_{0,1} = \eps:A\lrar K , \qquad \o_{0,2} = \eta:A^{\tensor 2} 
\lrar K.
\ee
 We then  note that 
 from \eqref{operations} and \eqref{o-0102}, we see
 that
\eqref{equivalence} holds for $\o_{0,1}$ and $\o_{0,2}$. 
 The  general case of \eqref{equivalence}  follows 
 by induction on $3g-3+n$, provided that
 $\o_{0,3}$ is appropriately defined. This is because
 \eqref{o-2} and \eqref{o-3} are induction formula
 recursively generating $\o_{g,n}$ from those with
 smaller values of $g$ and $n$. Since
 \begin{itemize}
 \item Case of \eqref{o-2}: $3g-3+n = [3(g-1)-3 + (n+2)]+1$, 
 \item Case of \eqref{o-3}: $3(g_1+g_2) -3 + |I|+|J|
 = [3g_1-3+|I|+1] + [3g_2-3+|J|+1] +1,$
 \end{itemize}
 we see that the \emph{complexity} $3g-3+n$ is always
 reduced by $1$ in each of the recursive formulas. 
 
Finally,  we see that since
$\Mbar_{0,3}$ is just a point as we have noted
above in the discussion of $\Aut(\bP^1)$, we have
$\o_{0,3}=\O_{0,3}$. Hence 
  \be\label{o-03}
 \o_{0,3}(v_1,v_2,v_3)=\eta(v_1v_2,v_3),
 \ee
which  makes \eqref{equivalence} holds for all values of 
 $(g,n)$. This completes the proof.
\end{proof}

\begin{rem}
The above two propositions show that 
a 2D TQFT can be defined either just by a commutative
Frobenius algebra $A$, by a system of 
maps $\{\o_{g,\bar{m},n}\}$ satisfying the TQFT
axioms,
or the degree $0$ part of a CohFT. From now on,
we use $(A,\o_{g,n})$ to denote a 2D TQFT,
which is less cumbersome and easier to deal with.
\end{rem}

\begin{prop}
The genus $0$ values of a 2D TQFT is given by
\be\label{TQFT 0n}
\o_{0,n}(v_1,\dots,v_n) = 
\epsilon(v_1\cdots v_n).
\ee
\end{prop}

\begin{proof}
This is a direct consequence of CohFT 3 and
\eqref{complete set}.
\end{proof}

One of the original motivations of TQFT
\cite{Atiyah, Segal} is to identify the
\emph{topological invariant} $Z(N)$
of a closed 
manifold $N$. In our current setting, 
it is defined as 
\be\label{g-invariant}
Z(\Sigma_g):= \epsilon\big(
\lam^{-1}(\o_{g,1})\big)
\ee
for a closed oriented surface $\Sigma_g$ of genus
$g$. Here, $\o_{g,1}:A\lrar K$ is 
an element of
$A^*$, and $\lam:A\overset{\sim}{\lrar} A^*$
is the canonical isomorphism \eqref{lam}.

\begin{prop}
\label{prop:g-invariant}
The topological invariant $Z(\Sigma_g)$ of
\eqref{g-invariant} is given by
\be\label{g-invariant formula}
Z(\Sigma_g) = \epsilon(\mathbf{e}^g),
\ee
where $\mathbf{e}\in A$ is the  Euler element
of \eqref{Euler}.
\end{prop}

\begin{lem}
We have
\be\label{11=Euler}
\mathbf{e} := m\circ \delta(1) 
= \lam^{-1}(\o_{1,1}).
\ee
\end{lem}

\begin{proof}
This follows from
$$
\o_{1,1}(v) =
\sum_{a,b} \o_{0,3}(v,e_a,e_b)\eta^{ab}
=
 \sum_{a,b}\eta(v,e_ae_b)
\eta^{ab}=\eta(v,\mathbf{e})
$$
for every $v\in A$. 
\end{proof}

\begin{proof}[Proof of 
Proposition~\ref{prop:g-invariant}] Since the starting
 case  $g=1$ follows from the above Lemma,
we prove the formula by induction, which goes 
as follows:
\begin{align*}
\o_{g,1}(v) 
&= 
\sum_{a,b}
\o_{g-1,3}(v,e_a,e_b)\eta^{ab}
\\
&= 
\sum_{i,j,a,b}
\o_{0,4}(v,e_a,e_b,e_i)
\o_{g-1,1}(e_j)\eta^{ab}\eta^{ij}
\\
&=
\sum_{i,j,a,b}
\eta(ve_ae_b,e_i)
\o_{g-1,1}(e_j)\eta^{ab}\eta^{ij}
\\
&=
\sum_{i,j}
\eta(v\mathbf{e},e_i)
\o_{g-1,1}(e_j)\eta^{ij}
\\
&=
\o_{g-1,1}(v\mathbf{e})
\\
&=
\o_{1,1}(v\mathbf{e}^{g-1})
\\
&=
\eta(v\mathbf{e}^{g-1},\mathbf{e})
= \eta(v,\mathbf{e}^g).
\end{align*}
\end{proof}

A closed genus $g$ surface is obtained by
sewing $g$ genus $1$ pieces with one 
output boundaries to a genus $0$ surface with
$g$ input boundaries. Since the Euler element
is the output of the genus $1$ surface
with one boundary, we 
obtain the same result
$$
Z(\Sigma_g) = \o_{0,g}(\overset{g}{\overbrace{\mathbf{e},\dots,
\mathbf{e}}}).
$$
Finally we have the following:

\begin{thm}
The value of a 2D TQFT is given by
\be\label{TQFT gn}
\o_{g,n}(v_1,\dots,v_n)
=
\epsilon(v_1\cdots v_n \mathbf{e}^g).
\ee
\end{thm}

\begin{proof}
The argument is the same as the proof
of Proposition~\ref{prop:g-invariant}:
\begin{align*}
\o_{g,n}(v_1,\dots,v_n)
&=
\o_{1,n}(v_1\mathbf{e}^{g-1},v_2,\dots,
v_n)
\\
&=
\sum_{a,b}
\o_{0,n+2}(v_1\mathbf{e}^{g-1},v_2,\dots,
v_n, e_a,e_b)\eta^{ab}
\\
&=
\epsilon(v_1\cdots v_n \mathbf{e}^g).
\end{align*}
\end{proof}


\section{Category of cell graphs}
\label{sect:cell}

In the original formulation of 2D TQFT, 
the operations of multiplication and comultiplication
are associated with an oriented surface of
genus $0$ with three boundary circles. 
In this section, we introduce a category
of ribbon graphs, which carries the information of
all finite-dimensional Frobenius algebras. 
To avoid unnecessary confusion, we use
the terminology of \emph{cell graphs}
 in this article, instead of more common
ribbon graphs. Ribbon graphs naturally
appear for encoding complex structures of
a topological surface 
(see for example, \cite{K1992, MP1998}).
Our purpose of using ribbon graphs are
for degeneration of stable curves, and we label
vertices, instead of \emph{faces}, of a ribbon 
graph.

\begin{Def}[Cell graphs]
A connected \textbf{cell graph} 
of topological type $(g,n)$ is the
$1$-skeleton of a cell-decomposition of a
connected
closed oriented surface of genus $g$ with 
$n$ labeled $0$-cells. We call a $0$-cell a 
\emph{vertex}, a $1$-cell an \emph{edge}, 
and a $2$-cell a \emph{face}, of the cell graph.
We denote by $\Gam_{g,n}$ the set of 
connected cell graphs of type $(g,n)$.
Each edge consists of two \textbf{half-edges}
connected at the midpoint of the edge.
\end{Def}

\begin{rem}
\begin{itemize}
\item
The \emph{dual} of a cell graph is
a ribbon graph, or Grothendieck's 
dessin d'enfant. We note that we label vertices
of a cell graph, which corresponds to 
face labeling of a ribbon graph.
Ribbon graphs are also called by different names,
such as
embedded graphs and maps.

\item We identify two cell graphs if there is a 
homeomorphism of the surfaces that brings 
one cell-decomposition to the other, 
keeping the labeling of $0$-cells. The only 
possible automorphisms of a cell graph
come from cyclic rotations of half-edges
at each vertex. 
\end{itemize}
\end{rem}

\begin{Def}[Directed cell graph]
A \textbf{directed cell graph} is a cell graph
for which an arrow is assigned to each edge.
An arrow is the same as an ordering of the 
two half-edges forming an edge.
The set of directed cell graphs of type
$(g,n)$ is denoted by $\vec{\Gam}_{g,n}$.
\end{Def}

\begin{rem}
A directed cell graph is a \emph{quiver}. Since
our graph is drawn on an oriented surface, 
a directed cell graph carries more information than
its underlying quiver structure. The tail vertex
of an arrowed edge is called the \emph{source},
and the head of the arrow the \emph{target}, in the
quiver language.
\end{rem}

 To label $n$ vertices, we normally use the $n$-set
 $$
 [n]:= \{1,2,\dots,n\}.
 $$
 However,  it is often easier to use any totally ordered set of 
 $n$ elements for labeling. 
The main reason we label the vertices of a cell graph
is we wish to assign an element of a $K$-vector space 
$A$ 
to each vertex. In this article, we consider
the case that  a cell graph
$\gam\in \Gam_{g,n}$ defines a linear map
\be\label{gam as function}
\Gam_{g,n}\owns\gam : A^{\tensor n}\lrar K.
\ee
The set of 
values of these functions $\gam$ can be more general. 
We discuss some of the general cases in \cite{OM9}.

An effective tool in graph enumeration is
edge-contraction operations. Often edge contraction
leads to an inductive formula for counting problems
of graphs. 
The same edge-contraction
 operations acquire algebraic meaning in our consideration.

\begin{Def}[Edge-contraction operations]
\label{def:ECO}
There are two types of 
 \textbf{edge-contraction operations} 
 applied to cell graphs. 
 \begin{itemize}

\item \textbf{ECO 1}: Suppose there is a
directed edge $\vec{E}=\overset{\lrar}{p_ip_i}$
in a cell graph $\gam\in \Gam_{g,n}$,
connecting the tail vertex $p_i$ and the head
vertex $p_j$.
We  \emph{contract} $\vec{E}$ in $\gam$,
and put the two vertices $p_i$ and $p_j$ together.
We use $i$ for the label of this new vertex, and 
call it again $p_i$. 
Then we have a new cell graph
 $\gam'\in \Gam_{g,n-1}$ with one less vertices.
 In this process, the topology of the surface on
 which $\gam$ is drawn does not change. Thus
 genus $g$ of the graph stays the same.

\begin{figure}[htb]
\includegraphics[height=1in]{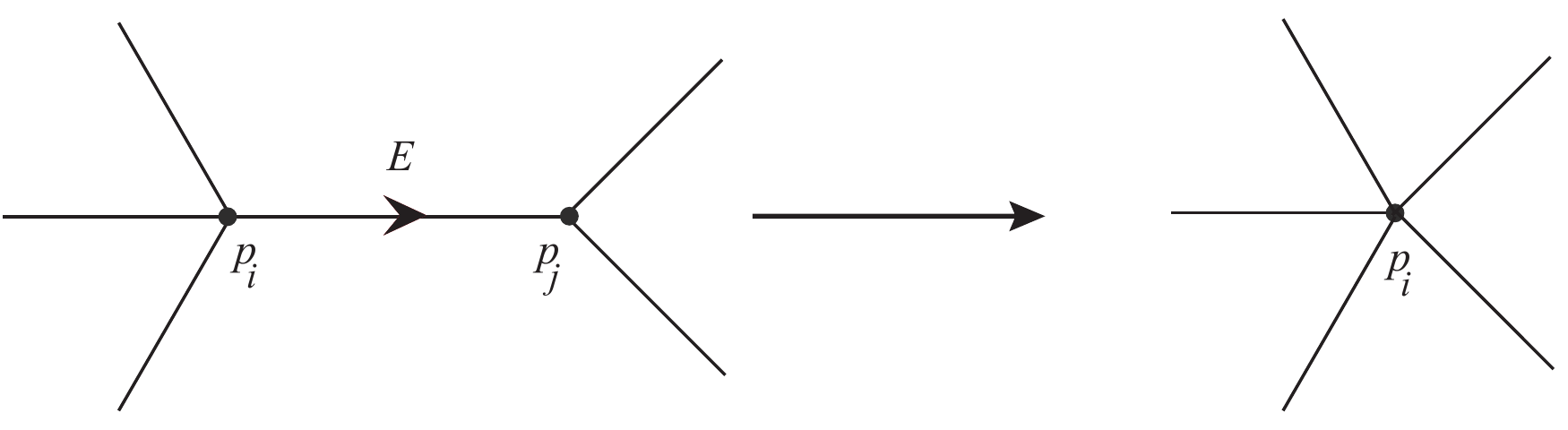}
\caption{Edge-contraction operation ECO 1.
The edge bounded by two vertices $p_i$ and $p_j$
is contracted to a single vertex $p_i$.
}
\label{fig:ECO 1}
\end{figure}

\item We use the notation $\vec{E}$ for the 
 edge-contraction operation 
 \be\label{ECO1}
 \vec{E}:\Gam_{g,n}\owns \gam\longmapsto \gam'\in 
 \Gam_{g,n-1}.
 \ee

\item \textbf{ECO 2}: Suppose there is a directed loop
$\vec{L}$ in $\gam\in\Gam_{g,n}$ at the $i$-th vertex $p_i$.
Since a loop in the $1$-skeleton of
a cell decomposition is 
a topological cycle on the surface, its contraction 
inevitably changes the topology of the surface. 
First we look at the half-edges incident to 
vertex $p_i$. Locally around $p_i$ on the 
surface, the directed loop $\vec{L}$ 
separates the neighborhood of
$p_i$
into two pieces. Accordingly, we put 
the incident half-edges into  
two groups. We then break the vertex $p_i$
into two vertices, $p_{i_1}$ and $p_{i_2}$, so that 
one group of half-edges are incident to $p_{i_1}$,
and the other group to $p_{i_2}$. 
The order of two vertices 
is determined by placing the loop $\vec{L}$
\emph{upward} near at vertex $p_i$. 
Then we name the new vertex on its left by $p_{i_1}$,
and on its right by $p_{i_2}$.

Let $\gam'$ denote the possibly 
disconnected graph obtained by 
contracting $\vec{L}$ and separating the vertex
to two distinct vertices labeled by $i_1$ and $i_2$.

\begin{figure}[htb]
\includegraphics[height=1.1in]{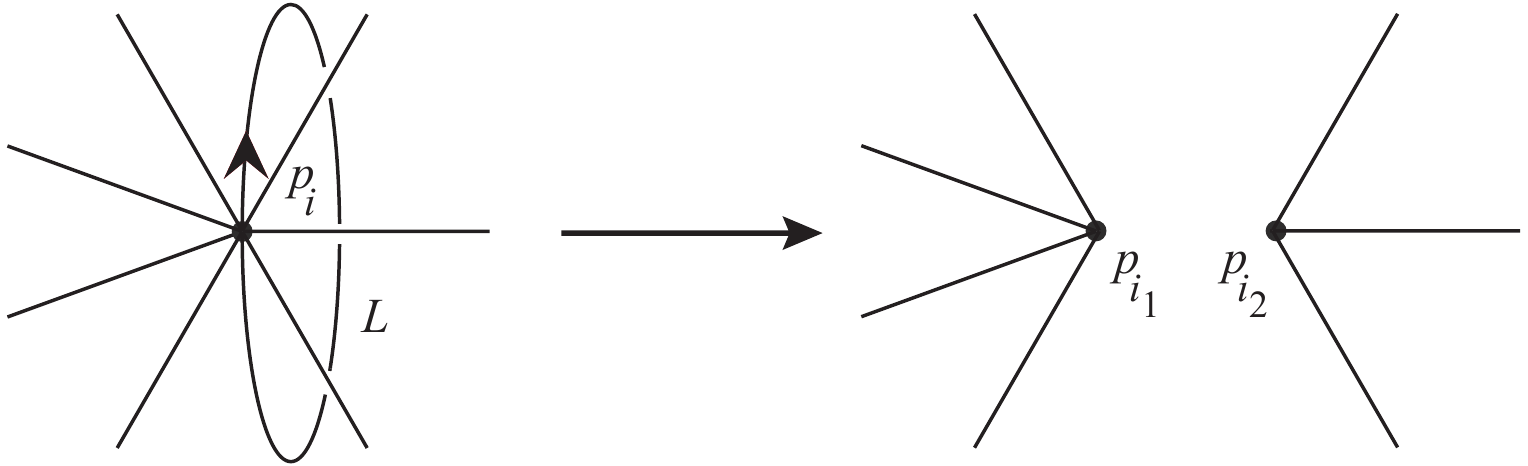}
\caption{Edge-contraction operation 
ECO 2. The contracted edge is a loop $\vec{L}$ of 
a cell graph. Place the loop 
so that it is upward near at $p_i$
to which 
$\vec{L}$ is attached. The vertex $p_i$  is 
then broken into two vertices,
$p_{i_1}$ on the left, and $p_{i_2}$ on 
the right. Half-edges 
incident to $p_i$ are separated into two groups,
belonging to two sides of the loop near $p_i$.
}
\label{fig:ECO2}
\end{figure}

\item
If $\gam'$ is connected, then it is in $\Gam_{g-1,n+1}$.
The loop $\vec{L}$  is a \textit{loop of handle}.
We use the same notation $\vec{L}$ to indicate
the edge-contraction operation
\be\label{ECO2-1}
\vec{L}:\Gam_{g,n}\owns \gam\longmapsto
\gam'\in \Gam_{g-1,n+1}.
\ee

\item
If $\gam'$ is disconnected, then write
$\gam'=(\gam_1,\gam_2)\in \Gam_{g_1,|I|+1}
\times \Gam_{g_2,|J|+1}$, where 
\be\label{disconnected}
\begin{cases}
g=g_1+g_2\\
I\sqcup J = \{1,\dots,\widehat{i},\dots,n\}
\end{cases}.
\ee
The edge-contraction operation is again denoted 
by
\be\label{ECO2-2}
\vec{L}:\Gam_{g,n}\owns \gam\longmapsto
(\gam_1,\gam_2)\in \Gam_{g_1,|I|+1}\times
\Gam_{g_2,|J|+1}.
\ee
In this case we  call $\vec{L}$ a \textit{separating loop}.
Here, vertices labeled by $I$ belong to the connected 
component of genus 
$g_1$, and those labeled by $J$ are on the other 
component of genus $g_2$. 
Let $(I_-,i,I_+)$ (reps. $(J_-,i,J_+)$) be the reordering of 
$I\sqcup \{i\}$ (resp. $J\sqcup \{i\}$)
in the increasing order. 
Although we give labeling $i_1,i_2$ to the
two vertices created by breaking $p_i$,
since they belong to distinct graphs, we can simply
use $i$ for the label of $p_{i_1}\in \gam_1$ and
the same $i$ for $p_{i_2}\in \gam_2$. 
The arrow of $\vec{L}$ translates
into the information of ordering among the 
two vertices $p_{i_1}$ and $p_{i_2}$. 
\end{itemize}
\end{Def}

\begin{rem}
Let us define $m(\gam)=2g-2+n$
for a graph $\gam\in \Gamma_{g,n}$. Then every
edge-contraction operation  reduces 
$m(\gam)$ exactly by $1$. 
Indeed, for ECO 1, we have 
$$
m(\gam') = 2g -2
+(n-1) = m(\gam)-1.
$$
The ECO 2 applied to a loop of handle produces
$$
m(\gam') = 2(g-1)-2+(n+1) = m(\gam)-1.
$$ 
For a separating loop, we have
$$
\begin{matrix}
&2g_1-2+|I|+1
\\
{+)}&{2g_2-2+|J|+1}
\\
&\overline{2g_1+2g_2-4+|I|+|J|+2}
&=\; \;2g-2+n-1.
\end{matrix}
$$
\end{rem}

The motivation for our
 introduction of directed cell graphs is that we need
them when we 
deal with non-commutative Frobenius algebras. 
The operation of taking 
disjoint union  is symmetric. Therefore,  
2D TQFT inevitably
 leads to a commutative Frobenius algebra. 
The advantage of our formalism using directed 
cell graphs is that we can deal with 
non-commutative Frobenius algebras and 
non-symmetric tensor products. 

For the purpose of presenting the idea of 
the category of cell graphs as simple as 
possible, we restrict ourselves
to undirected cell graphs in this article. 
Therefore, we will only recover commutative
Frobenius algebras and usual 2D TQFT. 
A more general theory will be given in \cite{OM9}.

We now introduce the category of cell graphs. 
The most unusual point we present here is that
a morphism between cell graphs is \emph{not}
a cell map. Recall that 
 a \textbf{cell map} 
$
f:\gam\lrar \gam'
$
 from a   cell graph
${\gam}$ to another cell graph ${\gam}'$ 
is a topological map between $1$-dimensional
cell complexes. Thus $f$ sends a vertex of $\gam$ 
to a vertex of $\gam'$,  and an edge of $\gam$ 
 to either an edge or a vertex of $\gam'$, keeping 
 the incidence relations. In particular, a cell map is 
 continuous with respect to the topological structure
 on cell graphs indued from the surface on which 
 they are drawn.

\begin{Def}[Category of cell graphs]
The category $\CG$ of cell graphs is defined as follows.
\begin{itemize}
\item The set of objects of $\CG$ is the set of all
 cell graphs:
\be\label{objects}
Ob(\CG) = \coprod_{g\ge 0,n>0}{\Gam}_{g,n}.
\ee

 \item 
A \textbf{morphism}
$$
f\in \Hom(\gam,\gam')
$$
 is  a composition of a finite 
 sequence of edge-contraction operations and
cell graph automorphisms.
 In particular,
$\Hom({\gam},{\gam}) 
=\Aut(\gam)$. 
If there is no way to bring ${\gam}$ to $\gam'$
by  consecutive applications of
edge-contraction operations and automorphisms, then
we define 
$\Hom({\gam},{\gam}') = \emptyset$,
even though there may be cell maps
between them. 
\end{itemize}
\end{Def}

\begin{rem}
The triple $(\CG,\sqcup, \emptyset)$ forms
a symmetric monoidal category.
\end{rem}

\begin{rem}
Automorphisms of a cell graph and ECOs of the first 
kind
are cell maps, but ECO 2 operations 
are not. When an ECO 2
is involved,  
a morphism between cell graphs does not
have to be a cell map. Even it
 may not be  a continuous map. 
\end{rem}

\begin{ex}
A few simple examples of morphisms 
are given below.   Note that vertices are all 
labeled, and automorphisms are required to 
keep labeling.
\begin{align}
\label{1 to 1}
&\Hom\left(
\bullet,\bullet\right)
=\{id\}.
\\
\label{1 to 2}
&\Hom\left(
\bullet,\two\right)
=\emptyset.
\\
\label{e1,e2}
&
\Hom(\bullet\!\! \frac{E_1}{\hskip0.2in}\!\!\!\bullet
\!\!\! \frac{E_2}{\hskip0.2in}\!\!\bullet, 
\two) = \{E_1,E_2\}.
\\
\label{e1e2}
&
\Hom(\bullet\!\! \frac{E_1}{\hskip0.2in}\!\!\!\bullet
\!\!\! \frac{E_2}{\hskip0.2in}\!\!\bullet, 
\bullet) = \{E_1E_2=E_2E_1\}.
\\
\label{e1 loop}
&
\Hom\left(
\underset{E_2}{\overset{E_1}{\bullet\!\!\!\bigcirc\!\!\!\bullet}},\bullet\!\!\bigcirc\right) 
=\{E_1,E_2=\sigma(E_1)\}.
\\
\label{e1e2 loop}
&
\Hom\left(
\underset{E_2}{\overset{E_1}{\bullet\!\!\!\bigcirc\!\!\!\bullet}},\bullet\;\;\bullet\right)
=\{E_1E_2=E_2E_1\}.
\end{align}
In \eqref{e1e2}, we note that $E_1E_2:=E_1\circ E_2$
is equal to 
$E_2E_1:=E_2\circ E_1$, because they both produce the 
same result
$\three\rar  \two \rar \bullet$.
The cell graph of the left of \eqref{e1 loop}
and \eqref{e1e2 loop} 
has an automorphism $\sigma$
that interchanges $E_1$
and $E_2$. Thus as an edge-contraction operation,
$E_2=E_1\circ \sigma = \sigma(E_1)$.
Note that there is a $2:1$ covering cell map for the 
case of
 \eqref{e1 loop} that sends both edges $E_1$ and $E_2$
 on $\underset{E_2}{\overset{E_1}{\bullet\!\!\!\bigcirc\!\!\!\bullet}}$
 to the single loop of $\bullet\!\!\bigcirc$, and
 the two vertices on the first graph
  to the single vertex on  the second. 
  Since it is not an edge-contraction,
  this cell map is not a morphism. The morphism 
  of \eqref{e1e2 loop} is not a cell map, since it is not
   continuous. 
\end{ex}

Let  $\Vect$ be
 the category of finite-dimensional
 $K$-vector spaces. 
 The triple 
 $$
 \cC= (\Vect,\tensor, K)
 $$
   forms a 
   monoidal category. 
 Again for simplicity, we are concerned only with
 symmetric tensor products in this article, 
 so we consider $\cC$ a \emph{symmetric}
 monoidal category. 
 A $K$-object in $\Vect$ is a pair
 $(V,\eps:V\lrar K)$ consisting of a vector 
 space $V$ and a linear map $\eps:V\lrar K$.
 We denote by $\Vect/K$ the category of
 $K$-objects in $\Vect$. It has the unique
 final object $(K,id:K\lrar K)$.
 Therefore, 
 $$
 \cC/K = \big(\Vect/K,\tensor,(K,id:K\lrar K)\big)
 $$
 is again a monoidal category.
 We denote by 
\be\label{endo}
\cF un(\cC/K,\cC/K)
\ee
the \textbf{endofunctor category}
of the monoidal category $\cC/K$, which
 consists of monoidal functors 
$
\a:\cC/K \lrar \cC/K
$ 
as its objects,
and their natural transformations 
$\tau$ as morphisms. 
Schematically, we have 
$$
\xymatrix{
V\ar[ddd]_h\ar[ddr]^f && \a(V) 
\ar[ddd]_{\a(h)}\ar[ddr]^{\a(f)} \ar[rr]^\tau 
&&\b(V)\ar[ddd]^{\b(h)}
\ar[ddr]^{\b(f)}
\\
\\
&K&&K\ar[rr]^{\tau\hskip0.5in}&&K.
\\
W \ar[ur]_g &&\a(W) \ar[ur]^{\a(g)} 
\ar[rr]^\tau&&\b(W)\ar[ur]_{\b(g)}
}
$$
Here, the triangle on the left shows two objects
$(V,f:V\lrar K)$ and $(W,g:W\lrar K)$ of
$\cC/K$, and a morphism $h$ between them.
The prism shape on the right represents
two monoidal endofunctors $\a$ and $\b$
that assigns
\begin{align*}
&V\longmapsto \a(V), \qquad
V\longmapsto \b(V)
\\
&W\longmapsto \a(W), \qquad
W\longmapsto \b(W),
\end{align*}
and a natural transformation $\tau: \a\lrar \b$
among them.
The final object of
$\cF un(\cC/K,\cC/K)$ is the functor 
\be\label{final phi}
\phi: (V,f:V\lrar K) \lrar (K,id_K:K\lrar K)
\ee
which assigns the final object of the codomain
$\cC/K$ to 
everything in the domain $\cC/K$. 
With respect to the tensor product and the above
functor \eqref{final phi} as its identity object,
the endofunctor category $\cF un(\cC/K,\cC/K)$
is  again a  monoidal category.

\begin{Def}[ECO functor, \cite{OM9}]
\label{def:ECO functor}
The ECO functor is 
a monoidal functor
\be\label{TQFT functor}
\o:\CG\lrar \cF un(\cC/K,\cC/K)
\ee
satisfying the following conditions.
\begin{itemize}
\item The graph $\bullet\in \Gam_{0,1}$ 
consisting of only one vertex and no edge
corresponds to 
the identity functor 
\be\label{o bullet}
\o(\bullet)=id:\cC/K\lrar \cC/K.
\ee
\item Each graph $\gam\in \Gam_{g,n}$ corresponds to 
a functor
\be\label{gam functor}
\o(\gam): (V,\eps:V\lrar K)\longmapsto 
(V^{\tensor n}, \o_V(\gam):V^{\tensor n}\lrar K).
\ee

\item Edge-contraction operations correspond
to natural transformations.  
\end{itemize}
\end{Def}

Let us recall the notion of Frobenius object. 

\begin{Def}[Frobenius object]
\label{def:Frobenius opject}
Let $(\cC,\tensor,K)$ be a symmetric monoidal category. 
A \textbf{Frobenius object} is an object $V\in Ob(\cC)$
together with 
 morphisms 
 $$
m:V\tensor V\lrar V,\quad
\mathbf{1}:K\lrar V, \quad
\delta:V\lrar V\tensor V,\quad
\epsilon:V\lrar K,
$$
satisfying the following conditions:
\begin{itemize}
\item $(V,m,\mathbf{1})$ is a monoid object in 
$\cC$.
\item $(V,\delta,\epsilon)$ is a comonoid 
object in $\cC$.
\end{itemize}
We also require the compatibility condition
\eqref{m delta}
among morphisms  $m$
and $\delta$:
$$
\xymatrix{
&V\tensor V\tensor  V\ar[dr]^{m\tensor id}&\\
V\tensor V \ar[ur]^{id \tensor \delta}\ar[r]^{\;\;\; m} \ar[dr]_{\delta\tensor id}& 
V \ar[r]^{\delta \;\;\;\;} & V\tensor V. \\
&V\tensor V\tensor V \ar[ur]_{id \tensor m}
		}
$$
\end{Def}

Since we are considering the monoidal category
of $K$-objects in $\Vect$, there is a priori no notion of
$\mathbf{1}$ in $V$. The existence of 
the morphism $\mathbf{1}:K\lrar V$ requires
a non-degeneracy 
condition. The following theorem is proved in 
\cite{OM9}.

\begin{thm}[Generation of Frobenius objects \cite{OM9}]
An object $(V,\eps:V\lrar K)$ of $\cC/K$ is a
Frobenius object if $\o_V(\two) : V\tensor V\lrar K$ 
defines a non-degenerate symmetric bilinear form 
on $V$. 
\end{thm}


\section{2D TQFT from cell graphs}
\label{sect:cell TQFT}

The result of Section~\ref{sect:CohFT}
 tells us that 
a 2D TQFT can be defined as a system $(A,\o_{g,n})$
of
linear maps 
\be\label{o-gn}
\o_{g,n}:A^{\tensor n}\lrar K
\ee
defined for all values of $g\ge 0$ and $n\ge 1$,
satisfying a set of conditions. 
The required conditions are the following: First,
$(A,\o_{g,n})$ is a CohFT for $2g-2+n>0$. 
In addition, we require that
\be\label{o-gn conditions}
\begin{aligned}
&\o_{0,1} = \eps: A\lrar K,
\\
&
\o_{0,2}=\eta:A\tensor A\lrar K.
\end{aligned}
\ee
In this section we give a different formulation of 
a 2D TQFT, 
based  on cell graphs and a different set of 
axioms. Our ultimate goal is to 
relate 2D TQFT, CohFT, mirror symmetry,
 topological recursion,
and quantum curves. Later in these lectures, we 
introduce quantum curves. Relations
between all these subjects
 will be discussed elsewhere \cite{OM9}.

\begin{thm}[Graph independence \cite{OM3}]
\label{thm:independence}
Let $(A,\eps:A\lrar K)$ be 
a Frobenius object under the ECO functor $\o$ of
Definition~\ref{def:ECO functor}. 
Then
every connected cell graph $\gam\in \Gam_{g,n}$
gives rise to the same map
\be\label{independence}
\o_A(\gam): A^{\tensor n}\owns 
v_1\tensor \cdots \tensor v_n \longmapsto
\epsilon(v_1\cdots v_n \mathbf{e}^g)\in  K,
\ee
where $\mathbf{e}$ is the Euler element 
of \eqref{Euler}.
\end{thm}

\begin{cor}[ECO implies TQFT]
\label{cor:ECA=TQFT}
Define $\o_{g,n}(v_1,\dots,v_n) = 
\o_A(\gam)(v_1,\dots,v_n)$ for every
$\gam\in \Gam_{g,n}$. Then 
$\{\o_{g,n}\}$ is a 2D TQFT. 
\end{cor}

\begin{proof}
Since the value of \eqref{independence} is the same
as
\eqref{TQFT gn}, it is a 2D TQFT. 
\end{proof}

The rest of the section is devoted to proving
Theorem~\ref{thm:independence}. 
We first give  three  examples of
graph independence.

\begin{lem} [Edge-removal lemma]
\label{lem:reduced}
Let $\gam\in \Gam_{g,n}$.
\begin{itemize}
\item \textbf{Case 1.} There is a disc-bounding loop $L$ in 
$\gam$. Let $\gam'\in \Gam_{g,n}$ be the graph
obtained by simply removing $L$ from $\gam$.
Note that we are not contracting $L$.

\item \textbf{Case 2.} The graph $\gam$
containts two edges $E_1$ and
$E_2$ between 
two distinct vertices  $p_i$ and  $p_j$
 that bound a disc. Let $\gam'\in \Gam_{g,n}$ 
be the graph obtained by 
removing $E_2$. Here again, we are just eliminating
$E_2$. 

\item \textbf{Case 3.} 
Two loops, $L_1$ and $L_2$, in $\gam$ are
attached to the $i$-th vertex $p_i$. If they are 
homotopic, then let $\gam'\in \Gam_{g,n}$
be the graph obtained by removing $L_2$ from 
$\gam$. 
\end{itemize}
In each of the above cases, we  have
\be\label{gam=gam'}
\o_A(\gam)(v_1,\dots,v_n) = \o_A(\gam')(v_1,\dots,v_n).
\ee
\end{lem}

\begin{figure}[htb]
\includegraphics[height=1in]{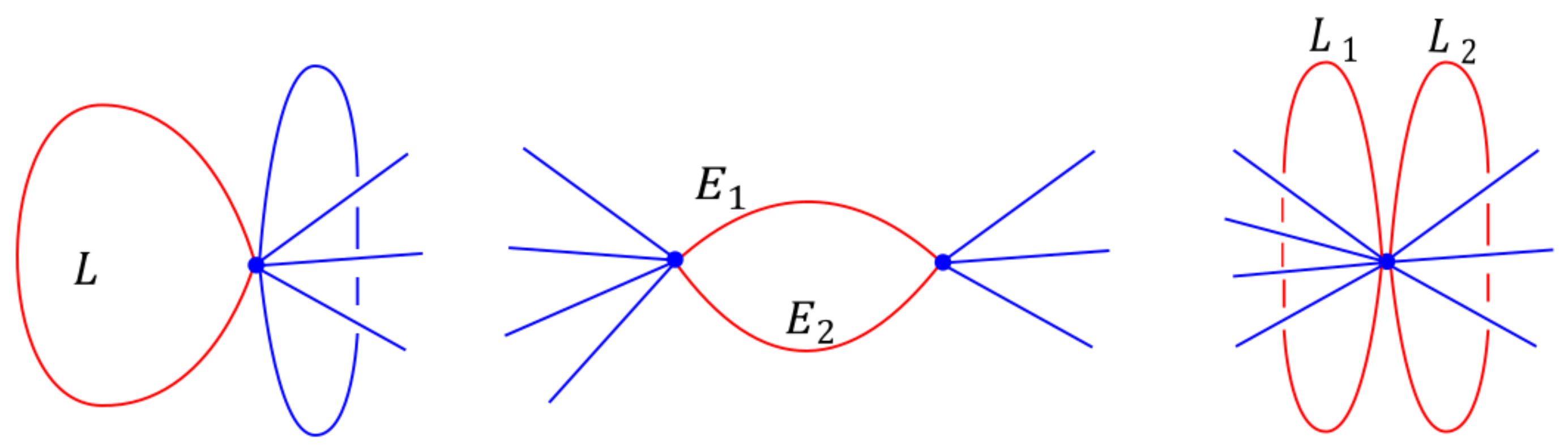}
\caption{Removal of disc-bounding edges. 
}
\label{fig:eliminate}
\end{figure}

\begin{proof}
Case 1.  Contracting a disc-bounding loop attached to 
$p_i$ creates 
$(\gam_0,\gam')\in \Gam_{0,1}\times \Gam_{g,n}$,
where $\gam_0$ consists of only one vertex and
no edges. The natural transformation 
corresponding to ECO 1 then gives
\begin{align*}
\o_A(\gam)(v_1,\dots,v_n) &= 
\sum_{a,b,k,\ell}\eta(v_i,e_ke_\ell)
\eta^{ka}\eta^{\ell b}\o_A(\gam_0)(e_a)
\o_A(\gam')(v_1,\dots,v_{i-1},e_b,v_{i+1}
\dots,v_n)
\\
&=
\sum_{a,b,k,\ell}\eta(v_i,e_ke_\ell)
\eta^{ka}\eta^{\ell b}\eta(\mathbf{1},e_a)
\o_A(\gam')(v_1,\dots,v_{i-1},e_b,v_{i+1}
\dots,v_n)
\\
&=
\sum_{b,k,\ell}\eta(v_i,\mathbf{1}\cdot e_\ell)
\eta^{\ell b}
\o_A(\gam')(v_1,\dots,v_{i-1},e_b,v_{i+1}
\dots,v_n)
\\
&=
\sum_{b,\ell}\eta(v_i,e_\ell)
\eta^{\ell b}
\o_A(\gam')(v_1,\dots,v_{i-1},e_b,v_{i+1}
\dots,v_n)
\\
&=\o_A(\gam')(v_1,\dots,v_{i-1},v_i,v_{i+1}
\dots,v_n).
\end{align*}

 Case 2. Contracting Edge $E_1$ makes $E_2$
 a disc-bounding 
loop at $p_i$. We can remove it by
Case 1. Note that the new vertex  is assigned
with $v_iv_j$. Restoring $E_1$  makes
the graph exactly the one obtained by removing
 $E_2$ from $\gam$. Thus \eqref{gam=gam'}
holds. 

\smallskip

Case 3. Contracting Loop $L_1$ makes $L_2$
a disc-bounding loop. Hence we can remove it
by Case 1. Then restoring $L_1$ creates a graph
obtained from $\gam$ by removing $L_2$.
Thus \eqref{gam=gam'}
holds. 
\end{proof}

\begin{rem} The three cases
treated above correspond to removing
degree $1$ and $2$ vertices from the dual ribbon
graph.
\end{rem}

\begin{Def}[Reduced graph]
A cell graph is \textbf{reduced} if it
 does not contain any disc-bounding
loops or disc-bounding bigons. In terms of
 dual ribbon graphs,   the dual
 of a reduced cell graph 
has no vertices of degree $1$ or $2$.
\end{Def}

We can see from Lemma~\ref{lem:reduced}, Case 1,
that
every graph $\gam\in\Gam_{0,1}$ gives the 
same map 
\be\label{01}
\o_A(\gam)(v) = \epsilon(v). 
\ee
Similarly, Cases 2 and 3 of 
Lemma~\ref{lem:reduced}  show 
that every graph $\gam\in \Gam_{0,2}$ defines $$
\o_A(\gam) (v_1,v_2) = \eta(v_1,v_2).
$$
This is because we can remove all edges and loops but
one that connects the two vertices.  Then 
by the natural transformation corresponding to   ECO 1,
the value of the assignment $\o_A(\gam)$ 
is $\epsilon(v_1v_2)=
\eta(v_1,v_2)$.

\begin{proof}[Proof of Theorem~\ref{thm:independence}]
We use the induction on $m=2g-2+n$. The base
case is $m=-1$, or $(g,n) = (0,1)$, for which
the theorem holds by \eqref{01}.
Assume that \eqref{independence}
holds for all $(g,n)$ with
$2g-2+n<m$. Now let $\gam
\in \Gam_{g,n}$ be a cell graph 
of type $(g,n)$  such that
$2n-2+n=m$. 
Choose an
arbitrary straight edge of $\gam$ that connects
two distinct vertices, say  $p_i$ and  $p_j$. 
Then the natural transformation of
 contracting this edge to $\gam'$ gives
$$
\o_A(\gam)(v_1,\dots,v_n)  = 
\o_A(\gam')(v_1,\dots,v_{i-1},v_iv_j,
v_{i+1} \dots, \widehat{v_j},\dots,v_n)
=\epsilon(v_1\dots v_n \mathbf{e}^g).
$$
If we have chosen an arbitrary
loop attached to  $p_i$, then its contraction by
ECO 2 gives two cases, depending  on whether
the loop is a loop of handle or a separating 
loop. For the former case, we have 
a graph $\gam'$, and by appealing to
\eqref{complete set} and
\eqref{Euler basis}, we obtain
\begin{align*}
\o_A(\gam)(v_1,\dots,v_n)  
&= 
\sum_{a,b,k,\ell}\eta(v_i,e_ke_\ell)\eta^{ka}
\eta^{\ell b}
\o_A(\gam')(v_1,\dots,v_{i-1},
e_a,e_b,v_{i+1},\dots,v_n)
\\
&=
\sum_{a,b,k,\ell}\eta(v_ie_k,e_\ell)\eta^{ka}
\eta^{\ell b}
\o_A(\gam')(v_1,\dots,v_{i-1},
e_a,e_b,v_{i+1},\dots,v_n)
\\
&=
\sum_{a,k}\eta^{ka}
\o_A(\gam')(v_1,\dots,v_{i-1},
e_a,v_ie_k,v_{i+1},\dots,v_n)
\\
&=
\sum_{a,k}\eta^{ka}
\epsilon(v_1\cdots v_n \mathbf{e}^{g-1}e_ae_b )
\\
&=
\epsilon(v_1\cdots v_n \mathbf{e}^g).
\end{align*}
For the case of a separating loop, 
ECO 2 makes $\gam\rar (\gam_1,\gam_2)$,
and
we have
\begin{align*}
\o_A(\gam)(v_1,\dots,v_n)  
&= 
\sum_{a,b,k,\ell}\eta(v_i,e_ke_\ell)\eta^{ka}
\eta^{\ell b}
\o_A(\gam_1)\big(v_{I_-},e_a,v_{I_+}\big)
\o_A(\gam_2)\big(v_{J_-},e_b,v_{J_+}\big)
\\
&=
\sum_{a,b,k,\ell}\eta(v_i,e_ke_\ell)\eta^{ka}
\eta^{\ell b}
\epsilon\left(e_a \prod_{c\in I}v_c
\mathbf{e}^{g_1}\right)
\epsilon\left(e_b \prod_{d\in J}v_d
\mathbf{e}^{g_2}\right)
\\
&=
\sum_{a,b,k,\ell}\eta(v_ie_k,e_\ell)\eta^{ka}
\eta^{\ell b}
\eta\left(\prod_{c\in I}v_c, e_a
\mathbf{e}^{g_1}\right)
\epsilon\left(e_b \prod_{d\in J}v_d
\mathbf{e}^{g_2}\right)
\\
&=
\sum_{a,k}\eta^{ka}
\eta\left(\prod_{c\in I}v_c\mathbf{e}^{g_1}
, e_a\right)
\epsilon\left(v_ie_k \prod_{d\in J}v_d
\mathbf{e}^{g_2}\right)
\\
&=
\epsilon\left(v_i 
\prod_{c\in I}v_c\mathbf{e}^{g_1}
\prod_{d\in J}v_d
\mathbf{e}^{g_2}\right)
\\
&=
\epsilon(v_1\cdots v_n\mathbf{e}^{g_1+g_2}).
\end{align*}
Therefore, no matter how we apply 
ECO 1 or ECO 2, we always obtain the same
result. This completes the proof.
\end{proof}


\section{TQFT-valued topological recursion}
\label{sect:TR}

There is a direct relation between a Frobenius algebra
and Gromov-Witten theory when $A$ is given by the
big quantum cohomology of a target space. Since 
these Frobenius algebras are usually infinite-dimensional
over the ground field, they do not correspond to a 
2D TQFT discussed in the previous sections. 
 But for the case that the target space is
$0$-dimensional, the TQFT indeed captures the whole
Gromov-Witten theory. 

In this section, we present a general framework. 
Fundamental examples are  $A=\bQ$,
which gives $\psi$-class intersection numbers on
$\Mbar_{g,n}$, and  the 
center of the group algebra of a finite group
$A=Z\bC[G]$, which produces
Gromov-Witten invariants of the classifying space 
$BG$. The first example is considered in
\cite{OM3,OM4}. The latter case will be discussed
elswhere \cite{OM9}.

We wish to solve a graph enumeration 
problem, where as a graph we consider
a cell graph, and 
each of its  vertex is \emph{colored}
by a  \emph{parameter} $v\in A$. 
We also impose functoriality under
the edge-contraction  of
Definition~\ref{def:ECO} for this coloring.
The ECOs reduce the complexity of coloring
considerably, because the functoriality  makes
the {coloring} process graph independent,
as we have shown in the last section.
Thus the answer is just the number of 
graphs times the value 
$\epsilon(v_1\cdots v_n \mathbf{e}^g)$ for 
each topological type $(g,n)$.
Here comes the difficulty: there are infinitely many
graphs for each topological type, since we allow
multiple edges and loops. The standard idea 
for such counting problem is to appeal to the
\emph{Laplace transform}, which is 
 introduced by Laplace 
 for this particular
context of counting an infinite number of 
objects. Let us denote by
\be\label{Gam(mu)}
\Gam_{g,n}(\mu_1,\dots,\mu_n)
\ee
the set of all cell graphs with labeled 
vertices of degrees $(\mu_1,\dots,\mu_n)$. 
Denoting by $c_d$  the number of $d$-cells in a
cell-decomposition of a surface
of genus $g$, $d=0,1,2$,  we have
$
n= c_0$, $\sum_{i=1} ^n \mu_i = 2 c_1,$ and
$2-2g = n-c_1+c_2.
$
Therefore, \eqref{Gam(mu)} is a finite set.

Counting processes become easier if we do not
have any object with non-trivial automorphism.
There are many ways to eliminate automorphism,
for example, the minimalistic way of 
imposing the least possible conditions, or an excessive
way to kill automorphisms but for the most
objects the conditions are redundant. Actually, it
is  known 
that most of the cell graphs counted in 
\eqref{Gam(mu)}
are without any non-trivial automorphisms.
Since any possible automorphism induces a
cyclic permutation of half-edges incident to 
each vertex, the easiest way to disallow
any automorphism is to assign an outgoing 
arrow to \emph{one} of these half-edges,
as in \cite{DMSS,MS} (but not as a quiver). An automorphism
should preserve the arrowed half-edges,
 in addition to labeled vertices.
We denote by
\be\label{arrowed graphs}
\widehat{\Gam}_{g,n}(\mu_1,\dots,\mu_n)
\ee
the set of \textbf{arrowed cell graphs} with
labeled vertices of degrees $(\mu_1,\dots,\mu_n)$,
and by
\be\label{Cgn}
C_{g,n}(\mu_1,\dots,\mu_n) := 
\left|\widehat{\Gam}_{g,n}(\mu_1,\dots,\mu_n)\right|
\ee
its cardinality. This number is
always a non-negative integer, and 
$
C_{0,1}(2m) = \frac{1}{m+1} \binom{2m}{m}
$
is the $m$-th Catalan number.

We use the notation $\o_A(\gam)=\o_{g,n}:A^{\tensor n}
\lrar K$ of Corollary~\ref{cor:ECA=TQFT}
for $\gam\in \widehat{\Gam}_{g,n}(\mu_1,\dots,\mu_n)$. 
We are interested in considering the 
$C_{g,n}(\mu_1,\dots,\mu_n)$-weighted
TQFT
\be\label{Cgngamgn}
C_{g,n}(\mu_1,\dots,\mu_n)\cdot \o_{g,n}
:A^{\tensor n}
\lrar K,
\ee
and applying the edge-contraction operations
to these maps.
We contract the edge of $\gam$  that carries the
outgoing arrow at the 
first vertex $p_1$, then place a new arrow to the
half-edge next to the original half-edge
with respect to the cyclic ordering 
induced by the orientation of the surface.
If the edge that carries the outgoing arrow 
at $p_1$ is a loop, then after  
splitting $p_1$, we place 
another arrow to the next half-edge 
at each of the two newly created vertices, again next to 
the original loop. From this process we obtain 
the following counting formula:

\begin{prop}
The 2D TQFT weighted by the number of arrowed 
cell graphs
satisfies the following
equation.
\be\label{counting formula}
\begin{aligned}
&C_{g,n}(\mu_1,\dots,\mu_n)\cdot \o_{g,n}
(v_1,\dots,v_n) 
\\
&=
\sum_{j=2} ^n \mu_j 
C_{g,n-1}(\mu_1+\mu_j-2,\mu_2,\dots,
\widehat{\mu_j},\dots,\mu_n)
\cdot \o_{g,n-1}
(v_1v_j,v_2,\dots,\widehat{v_j},\dots,v_n)
\\
&+
\sum_{\a+\b=\mu_1-2}
C_{g-1,n+1}(\a,\b,\mu_2,\dots,\mu_n)\cdot
\o_{g-1,n+2}\big(\delta(v_1),v_2,\dots,v_n\big)
\\
&+
\sum_{\a+\b=\mu_1-2}
\sum_{\substack{
g_1+g_2=g\\I\sqcup J=\{2,\dots,n\}}}
\sum_{a,b,k,\ell}
\eta(v_1,e_ke_\ell)\eta^{ka}\eta^{\ell b}
\\
&\qquad
\times
\left(C_{g_1,|I|+1}(\a,\mu_I)\cdot
\o_{g_1,|I|+1}(e_a,v_I)
\right)
\left(C_{g_2,|J|+1}(\b,\mu_J)\cdot
\o_{g_2,|J|+1}(e_b,v_J)
\right).
\end{aligned}
\ee
\end{prop}

This is exactly the same formula of 
\cite{DMSS, MS, WL}
multiplied by 
$$\o_{g,n}(v_1,\dots,v_n)
=\epsilon(v_1\dots v_n\mathbf{e}^g).
$$
Let us now consider a \textbf{Frobenius algebra
twisted topological 
recursion}. To simplify the notation, we adopt the
following way of writing:
\be\label{2 function notation}
(f\tensor h)\circ \delta
=\sum_{a,b,k,\ell}\eta(\bullet,e_ke_\ell)
\eta^{ka}\eta^{\ell b} f(e_a)h(e_b),
\ee
where $f,h:A\lrar K$ are linear functions on $A$.

First let us  review the original topological 
recursion of  \cite{EO2007} 
defined on a \textbf{spectral curve}
$\Sigma$, which is just a disjoint union of
$r$ copies of open discs.
Let $U$ be the unit disc centered at $0$ of the 
complex $z$-line. We choose $r$ sets of functions
$(x_\a,y_\a)$, $\a=1,\dots,r$, defined on $U$
with Taylor expansions
\be\label{xy}
x_\a(z) = z^2+\sum_{k=3}^\infty a_{\a,k}z^k, \qquad
y_\a(z) = z+\sum_{k=2}^\infty b_{\a,k}z^k,
\ee
and a  meromorphic $1$-form (Cauchy kernel)
\be\label{omega}
\omega_{\a}^{a-b} (z)= \frac{dz}
{z-a} - \frac{dz}{z-b} + \omega_\a
\ee
on $U$, where $a,b\in U$, and
$\omega_\a$ is a holomorphic 
$1$-form on $U$. Since each $x_\a:U\lrar \bC$ is
a $2:1$ map, we have an involution on $U$
that 
keeps the same $x_\a$-value: 
\be\label{sigma alpha}
\sigma_\a: U
\lrar U, \qquad x_\a\big(\sigma_\a(z)\big) = x_\a(z).
\ee
To avoid confusion, we label $r$ copies of $U$
by $U_1,\dots,U_r$, and consider the functions
$(x_\a,y_\a)$ to be defined on $U_\a$. 
The \textbf{topological recursion}
is the following recursive equation
\begin{multline}
\label{TR}
 W_{g,n}(z_1,\dots,z_n) 
=\frac{1}{2\pi i}\sum_{\a=1}^r
\oint_{\partial U_\a}\frac{\omega_\a^{\sigma_\a(z)-z}(z_1)}
{\left(y_\a\big(\sigma_\a(z)\big)-y_\a(z)\right)dx_\a(z)}
\\
\times
\left[
W_{g-1,n+1}\big(z,\sigma_\a(z),z_2,\dots,z_n\big)
+\sum_{\substack{
g_1+g_2=g\\
I\sqcup J=\{2,\dots,n\}}}^{\text{No}\;(0,1)}
W_{g_1,|I|+1}(z,z_I)W_{g_2,|J|+1}
\big(\sigma_\a(z),z_J\big)
\right]
\end{multline}
on symmetric  meromorphic $n$-differential
forms $W_{g,n}$ defined on the disjoint union
$U_1^n \sqcup  \cdots \sqcup U_r ^n$ for $2g-2+n>0$. 
Here, $z_I = (z_i)_{i\in I}$, and 
``No $(0,1)$'' in the summation
means the partition $g=g_1+g_2$
and the set partition $I\sqcup J=\{2,\dots,n\}$ do
not allow $g_1=0$ and $I=\emptyset$, or $g_2=0$
and $J=\emptyset$.
The integration is performed with respect 
to $z\in \partial U$. Note that 
the differential form in the big bracket $[\;\;\;]$
in \eqref{TR} is a symmetric quadratic differential
in the variable $z\in U_\a$. The expression
$1/dx_\a(z)$ is the contraction operator
with respect to the vector field
$\frac{\partial z}{\partial x_\a}
\frac{\partial}{\partial z}$ on $U_\a$. 
Thus the integrand of the recursion becomes
a meromorphic $1$-form on $U_\a$ in the
$z$-variable, for which the integration is 
performed. The multiplication
by $\omega_\a^{\sigma_\a(z)-z}(z_1)$ is
simply the symmetric tensor product with
a $1$-form proportional to $dz_1$. 
The $1$-form $W_{0,1}$ and the
$2$-form $W_{0,2}$ are defined separately:
\be
\label{W01}
W_{0,1}(z) :=\sum_{\a=1}^r y_\a(z)dx_\a(z)
\ee
is defined on the disjoint union $U_1\sqcup \cdots
\sqcup U_r$. If $z\in U_\a$, then 
$W_{0,1}(z) = y_\a(z)dx_\a(z)$.   
In the form of \eqref{TR}, however,
$W_{0,1}$ does not
appear anywhere.
Similarly, $W_{0,2}$ is defined by
\be\label{W02}
W_{0,2}(z_1,z_2):= d_{z_1}\omega_\a ^{z_1-b}
(z_2)
\ee
if $z_1,z_2\in U_\a$. Here, the constant $b\in U_\a$
does not play any role. This $2$-form explicitly
appears in the recursion part (the terms in the
big bracket $[\;\;\;]$) of the formula.

\begin{Def} A
 \textbf{
Frobenius algebra twisted topological recursion} for
$
{\cW}_{g,n}=
\cW_{g,n}(z_1,\dots,z_n)
:A^{\tensor n}\lrar K
$
is the following formula:
\be\label{TQFT-TR}
\begin{aligned}
{\cW}_{g,n}&(z_1,\dots,z_n;v_1,\dots,v_n)
\\
=&\frac{1}{2\pi i}\sum_{\a=1}^r
\oint_{\partial U_\a}
\sum_{a,b,k,\ell} 
\cK_\a\big(z,\sigma_\a(z),z_1;e_k,e_\ell,v_1\big)
\\
&\times 
\left[
{\cW}_{g-1,n+1}\big(
z,\sigma_\a(z),z_2,\dots,z_n;e_a,e_b,v_2,\dots,v_n)
\phantom{
\sum_{\substack{
g_1+g_2=g\\
I\sqcup J=\{2,\dots,n\}}}^{\text{No}\;(0,1)}}
\right.
\\
&+
\left.
\sum_{\substack{
g_1+g_2=g\\
I\sqcup J=\{2,\dots,n\}}}^{\text{No}\;(0,1)}
{\cW}_{g_1,|I|+1}(z,z_I;e_a,v_I)
 {\cW}_{g_2,|J|+1}\big(\sigma_\a(z),z_J;e_b,v_J\big)
\right].
\end{aligned}
\ee
Here 
\be\label{kernel}
\cK_\a\big(z,\sigma_\a(z),z_1;e_k,e_\ell,v_1\big)
=
\frac{\omega_\a^{\sigma_\a(z)-z}(z_1)
\eta(e_ke_\ell,v_1)\eta^{ka}\eta^{\ell b}}
{\left(y_\a\big(\sigma_\a(z)\big)-y_\a(z)\right)dx_\a(z)}
\ee
is the integration-summation kernel.
Symbolically we can write \eqref{TQFT-TR} as
\be\label{TQFT-TR-symbolic}
\begin{aligned}
{\cW}_{g,n}
&=\frac{1}{2\pi i}\sum_{\a=1}^r
\oint_{\partial U_\a}K_\a\big(z,\sigma_\a(z),z_1\big)
\\
&\times 
\left[
{\cW}_{g-1,n+1}\circ \delta
+ 
\sum_{\substack{
g_1+g_2=g\\
I\sqcup J=\{2,\dots,n\}}}^{\text{No}\;(0,1)}
\left({\cW}_{g_1,|I|+1}
\tensor {\cW}_{g_2,|J|+1}
\right)\circ \delta
\right]
\end{aligned}
\ee
with the usual integration kernel
$$
K_\a\big(z,\sigma_\a(z),z_1\big) := 
\frac{\omega_\a^{\sigma_\a(z)-z}(z_1)}
{\left(y_\a\big(\sigma_\a(z)\big)-y_\a(z)\right)
dx_\a(z)}.
$$
\end{Def}

\begin{thm}
\label{thm:TR implies TQFT}
The topological recursion \eqref{TQFT-TR}
uniquely determines the $(A^{\tensor n})^*$-valued
$n$-linear differential form $\cW_{g,n}$ from the 
initial data $\cW_{0,2}$. If the initial data is given
by
$$
\cW_{0,2}(z_1,z_2;v_1,v_2) = W_{0,2}(z_1,z_2)
\cdot \eta(v_1,v_2)
$$
for a $2$-form \eqref{W02}, 
then there exists a solution $\{W_{g,n}\}$
of the 
topological recursion \eqref{TR} and
a 2D TQFT $\{\o_{g,n}\}$ such that
\be\label{TR=TQFT}
\cW_{g,n}(z_1,\dots,z_n;v_1,\dots,v_n)
= W_{g,n}(z_1,\dots,z_n)\cdot 
\o_{g,n}(v_1,\dots,v_n).
\ee
\end{thm}
\begin{proof}
The proof is done by induction on $m=2g-2+n$ 
with the base case $m=0$. We assume that
\eqref{TR=TQFT} holds for all $(g,n)$
such that $2g-2+n <m$, and use the value
$\o_{g,n}=\epsilon(v_1\cdots v_n
\mathbf{e}^g)$ given by
\eqref{TQFT gn} for the values of $(g,n)$
in the range of induction hypothesis.
Then by \eqref{TQFT-TR} and 
functoriality under ECO 2, we conclude
that \eqref{TR=TQFT} also holds for all $(g,n)$
such that $2g-2+n =m$.
\end{proof}

\begin{rem}
In comparison to edge-contraction operations, we 
note that the multiplication case ECO~1
 does not seem to have a counterpart in an explicit
way. It is actually included in 
the terms involving $g_1=0, |I|=1$
and $g_2=0,|J|=1$  in the 
partition sum, and \eqref{prod=coprod} is
used to change the 
comultiplication to multiplication. 
More precisely, if $I = \{i\}$, then it gives 
a term 
\begin{align*}
&
\cW_{0,2}(z,z_i;e_a,v_i)
\cW_{g,n-1}\big(\sigma_\a(z),z_2,
\dots,\widehat{z_i},\dots,z_n;
e_b,v_2,\dots,\widehat{v_i},\dots,v_n\big)
\\
&=
W_{0,2}(z,z_i)W_{g,n-1}\big(\sigma_\a(z),z_2,
\dots,\widehat{z_i},\dots,z_n\big)
\\
&\qquad
\cdot \o_{0,2}(e_a,v_i)
 \o_{g,n-1}(
e_b,v_2,\dots,\widehat{v_i},\dots,v_n)
\end{align*}
in the partition sum, assuming the induction
hypothesis. 
We also note that
\begin{align*}
&\sum_{a,b,k,\ell}
\eta(e_ke_\ell,v_1)\eta^{ka}\eta^{\ell b}
\o_{0,2}(e_a,v_i)
 \o_{g,n-1}(
e_b,v_2,\dots,\widehat{v_i},\dots,v_n)
\\
&=
\sum_{b,\ell}
\eta^{\ell b}
\o_{0,2}(v_1e_\ell,v_i)
 \o_{g,n-1}(
e_b,v_2,\dots,\widehat{v_i},\dots,v_n)
\\
&=
\sum_{b,\ell}
\eta^{\ell b}
\eta(v_1v_i,e_\ell)
 \o_{g,n-1}(
e_b,v_2,\dots,\widehat{v_i},\dots,v_n)
\\
&=
\o_{g,n-1}(v_1v_i,v_1,\dots,\widehat{v_i},\dots,v_n).
\end{align*}
\end{rem}

By taking the Laplace transform
of \eqref{counting formula} using the method 
of \cite{DMSS, MS}, we obtain 
a solution 
$\cW_{g,n}=
W_{g,n}^D(z_1,\dots,z_n)\cdot \o_{g,n}$
to the topological recursion, 
where $W_{g,n}^D$ is given by
\cite[(4.14)]{DMSS} with respect
to a global spectral curve 
of \cite[Theorem~4.3]{DMSS}, and $\o_{g,n}$
is the TQFT corresponding to the Frobenius 
algebra $A$.

\bigskip

\part{Quantization of Higgs Bundles}

\section{Quantum curves}

The  cohomology ring $H^*(X,\bC)$
of a K\"ahler 
variety $X$ 
is a $\bZ/2\bZ$-graded 
$\bZ/2\bZ$-commutative Frobenius algebra over $\bC$. 
The genus $0$ Gromov-Witten
invariants of $X$ define the \emph{big quantum 
cohomology} of $X$, which is a quantum deformation 
of the cohomology ring. If the mirror symmetry
is established for $X$, then  the information of
big quantum cohomology of $X$ is supposed to be
encoded in 
holomorphic geometry of a mirror dual space $Y$.
Gromov-Witten invariants are generalized to all
values of $(g,n)$ with $2g-2+n>0, g\ge 0, n>0$. 
The question  is:

\begin{quest} What should be the
holomorphic geometry on $Y$ that captures 
 all genera Gromov-Witten invariants of $X$
through  mirror symmetry?
\end{quest}

Since the transition from  $g=0$ to all 
values of $g\ge 0$ is indeed a \emph{quantization}, 
the holomorphic object on $Y$ that should capture
higer genera Gromov-Witten invariants of $X$ 
is a \emph{quantum geometry} of $Y$. 
A na\"ive guess may be that it should be a $\cD$-module
that represents $Y$ as its classical limit. Hence
 the $\cD$-module is not
  defined on $Y$. Then where does it live?

The idea of \textbf{quantum curves}
concerns a rather restricted situation, when the 
mirror geometry $Y$ is captured by an algebraic, 
or an analytic, curve. Typical examples  are  the mirror of 
 toric Calabi-Yau orbifolds  of three dimensions. 
 Geometry of the mirror $Y$ of a toric
 Calabi-Yau $3$-fold is encoded in 
 a complex curve known as the \emph{mirror curve}. 
 Another situation is enumeration 
 problems of various 
Hurwitz-type coverings of $\bP^1$, 
and also many decorated
graphs on surfaces. For these examples, 
although there  are no ``space'' $X$, the mirror
geometry exists, and is indeed a curve. 
These mirror curves are special cases of 
more general notion of \textbf{spectral curves}. 
Besides mirror symmetry, spectral curves appear
in  theory of integrable systems, random matrix theory, 
topological recursion,
and Hitchin theory of Higgs bundles. A spectral curve $\S$
 has two
 common features. 
 The first one is that it
is a Lagrangian subvariety of a holomorphic 
symplectic surface.
The other 
  is the existence of
    a projection $\pi:\S\lrar C$ to another
     curve  $C$, called a \emph{base curve}.
The quantum curve  is a $\cD$-module 
on the base curve $C$ such that its
\emph{semi-classical limit} is the spectral curve 
realized in the cotangent bundle $\S\subset T^*C$.

From the analogy of 2D TQFT and CohFT, we
note that a spectral curve is already a
quantized object, since it corresponds to 
quantum cohomology. The even part
$H^{even}(X,\bC)$, which is a commutative
Frobenius algebra, is not the one that corresponds
to a spectral curve. In this sense, 
CohFT is a result of two quantizations: the first one
from classical cohomology to a big quantum 
cohomology through $n$-point Gromov-Witten
invariants of genus $0$; and the second quantization
is the passage from genus $0$ to all genera.

\begin{rem}
 We note that
quantum cohomology itself is a Frobenius algebra,
though it is not finite dimensional, because it
requires the introduction of Novikov ring.
The even degree part forms
a commutative Frobenius algebra, yet it does
not correspond to a 2D TQFT in the way 
we presented in Part 1. 
\end{rem}

Now let us turn to the topic of Part 2. 
The prototype of quantization is 
a Schr\"odinger equation. 
Consider a harmonic oscillator 
 of  mass $1$, energy $E$, and the
 spring constant  $1/4$
 in one dimension. 
 It has a geometric 
 description as an elliptical motion of a constant
 angular momentum in the phase space, or the cotangent
 bundle of the real axis. Here, the spectral curve is 
 an ellipse
 $$
 \frac{1}{4}x^2+y^2 = E
 $$
  in 
 a real symplectic plane. The quantization of this 
 spectral curve is the
 quantization of the harmonic oscillator, 
 which is a second order
 stationary  Schr\"odinger equation in one variable:
 \be
 \label{ho}
\left(- \hbar^2 \frac{d^2}{dx^2}+
\frac{1}{4}x^2 -E\right)
\psi(x, \hbar) = 0.
 \ee

The quantization we discuss in Part 2 is in complete
parallelism to quantization of harmonic oscillator. 
As a holomorphic symplectic surface, we use
 $(\bC^2,dx\wedge dy)$.
A plane quadric
\be
\label{hyperbola}
\frac{1}{4}x^2-y^2=1
\ee
is an example of a spectral curve. 
In complex coordinates, 
we identify $y=d/dx$, ignoring the imaginary unit.
An example of quantization of \eqref{hyperbola}
is a  Schr\"odinger equation
 \be
 \label{Weber}
 \left(\left(\hbar\frac{d}{dx}\right)^2+1-\half \hbar
 -\frac{1}{4}x^2\right)
 \psi(x, \hbar) = 0,
 \ee
 which is essentially the same as quantum harmonic
 oscillator equation \eqref{ho}, and 
 is known as the Hermite-Weber equation.
 Its solutions are all well studied.
  \begin{quest}
Why do we care this well-known classical 
differential equation?
\end{quest}

\noindent
A surprising answer 
\cite{OM2,OM4, DMSS,MS} to this
question is that 
we find the intersection numbers of 
$\Mbar_{g,n}$ through the 
asymptotic expansion!
First we apply a gauge transformation
\be
\label{Catalan}
\begin{aligned}
&e^{-\frac{1}{4\hbar}x^2}
 \left(\left(\hbar\frac{d}{dx}\right)^2+1-\half \hbar
 -\frac{1}{4}x^2\right)
 e^{\frac{1}{4\hbar}x^2} Z(x, \hbar) 
 \\
 &=
 \left(\hbar \frac{d^2}{dx^2}+
 \hbar x\frac{d}{dx} +1\right)Z(x,\hbar)=0,
\end{aligned}
\ee
where 
$Z(x,\hbar) = e^{-\frac{1}{4\hbar}x^2}\psi(x,\hbar)$.
Recall the integer valued 
function  $C_{g,n}(\mu_1,\dots,\mu_n)$
of \eqref{Cgn},
 and 
define their generating functions
by
\be
\label{F}
F_{g,n}(x_1,\dots,x_n)
:= \sum_{\mu_1,\dots,\mu_n>0}\frac{C_{g,n}(\mu_1,\dots,\mu_n)}
{\mu_1\cdots\mu_n}x_1^{-\mu_1}\cdots 
x_n^{-\mu_n}.
\ee
It is discovered in \cite{DMSS} that
the derivatives $W_{g,n}=d_1\cdots d_n F_{g,n}$ 
satisfy the topological recursion based on 
the spectral curve \eqref{z}, which is the
semi-classical limit of \eqref{Catalan}.
With an appropriate adjustment for
$(g,n) = (0,1)$ and $(0,2)$, we have the following
all-order WKB expansion formula
\cite{OM4, DMSS,MS}:
\be
\label{WKB}
Z(x,\hbar) 
= \exp\left( \sum_{g,n} \frac{1}{n!}\hbar^{2g-2+n}
F_{g,n}(x,\dots,x)\right).
\ee
We find  (\cite{DMSS})
 that if we change the coordinate from $x$ to $t$ by
\be\label{xt}
x = x(t) = \frac{t+1}{t-1}+\frac{t-1}{t+1},
\ee
then $F_{g,n}\big(x(t_1),x(t_2),\dots,x(t_n)\big)
$ is a Laurent polynomial for each $(g,n)$
with $2g-2+n>0$. The coordinate change \eqref{xt}
is identified in \cite{OM2} as a \textbf{normalization} 
 of the singular curve \eqref{z}
in the Hirzebruch surface $\bP\big(K_{\bP^1}\dsum
\cO_{\bP^1}\big)$ by a sequence of blow-ups.
The highest degree part of this Laurent polynomial
is a homogeneous polynomial of degree
$6g-6+3n$
\be
\label{topdegree}
F^{\text{highest}}_{g,n}(
t_1,\dots,t_n)
=
\frac{(-1)^n}{2^{2g-2+n}}
\sum_{\substack{d_1+\dots+d_n\\
=3g-3+n}}\la \tau_{d_1}\cdots \tau_{d_n}\ra_{g,n}
\prod_{i=1}^n \left(
|2d_i-1|!!\left(\frac{t_i}{2}\right)^{2d_i+1}
\right),
\ee
where the coefficients
\be
\label{intersection}
\la \tau_{d_1}\cdots \tau_{d_n}\ra_{g,n} =
\int_{\Mbar_{g,n}}
\psi_1 ^{d_1}\cdots \psi_{n}^{d_n}
\ee
are  cotangent class intersection 
numbers on the moduli space 
$\Mbar_{g,n}$.  
Topological recursion is a mechanism to 
calculate all $F_{g,n}(x_1,\dots,x_n)$ from
the single equation \eqref{Weber}, 
or equivalently, \eqref{Catalan}.
Thus the quantum curve \eqref{Weber}
has the information of all intersection numbers
\eqref{intersection}. These are the topics 
discussed in our previous lectures \cite{OM4}.

Although the following topic
 is not what we deal with in this article,
 for the moment let us consider a symplectic surface
$(\bC^*\times \bC^*,d\log x\wedge d\log y)$. 
As a spectral curve, we use the zero locus of  
the \textbf{A-polynomial} $A_K(x,y)$ of a knot defined
in  \cite{CCGLS}.
For a given knot $K\subset S^3$, 
the $SL_2(\bC)$-character variety
\be
\label{cv}
\Hom\big(\pi_1(S^3\setminus K),
SL_2(\bC)\big)\big/\!\!\big/
SL_2(\bC)
\ee
of the fundamental group of the knot complement
determines an algebraic curve in $(\bC^*)^2$
defined by $A_K(x,y)\in \bZ[x,y]$. 
 Here, $(\bC^*)^2$, or to be more precise,
$(\bC^*)^2/(\bZ/2\bZ)$, is the $SL_2(\bC)$-character
variety of the fundamental group of the 
torus $T^2 = S^1\times S^1$, which is the boundary
of the knot complement.

Now consider a function
$f(q,n)$ in $2$ variables
$(q,n)\in \bC^*\times \bZ_+$, and define operators
\be
\label{xy operation}
\begin{cases}
(\widehat{x} f)(q,n) := q^n f(q,n)\\
(\widehat{y}f)(q,n) := f(q, n+1),
\end{cases}
\ee
following \cite{FGL, Gar}.
These operators satisfy the commutation relation
$$
\widehat{x}\cdot\widehat{y} 
= q\widehat{y}\cdot \widehat{x}.
$$
The procedure of changing $x\longmapsto
\widehat{x}$ and $y\longmapsto
\widehat{y}$ is the \textbf{Weyl quantization}. 
Garoufalidis \cite{Gar} conjectures that there exists a
\emph{quantization} of the A-polynomial 
 such that
\be
\label{AJ}
\widehat{A}_K(\widehat{x},\widehat{y};q)
J_K(q,n) = 0,
\ee
where $J_K(q,n)$ is the colored Jones polynomial
of the knot $K$ indexed by the dimension 
$n$ of the irreducible representation 
of $SL_2(\bC)$. Here, the quantization means
that the operator $\widehat{A}_K(\widehat{x},\widehat{y};q)$ recovers the A-poynomial
by the restriction
$$
\widehat{A}_K({x},{y};1) = A_K(x,y).
$$
This relations is  
the semi-classical limit,
which provides the initial condition of the 
WKB analysis.

A geometric definition of a quantum curve
that arises as the quantization of a Hitchin spectral
curve  is  developed in
\cite{OM5}, based on the 
work of \cite{DFKMMN} that solves
a conjecture of Gaiotto \cite{Gai, GMN}.
The WKB analysis of the quantum curve
\cite{OM1,OM2,OM4,IMS,IS} is 
performed by applying the topological recursion of
\cite{EO2007}. 
The Hermite-Weber
 differential equation is an example of
this geometric theory.
Although there have been many speculations 
\cite{GS} of
the applicability of the topological recursion
to low-dimensional topology,
still there is no counterpart of the Hitchin
type geometric 
theory for the case of the
quantization of  A-polynomials. 
The appearance of the \emph{modularity} 
in this context 
\cite{DG2,KMar,Z}
is a tantalizing phenomenon, 
on which there has been a great advancement.

In the following sections, we unfold a different 
story of quantum curves. In geometry, 
there is a process parallel to the passage from 
a spectral curve \eqref{z} to a quantum 
curve \eqref{Catalan}. This process is 
a \emph{journey} from the moduli space of 
Hitchin spectral curves to the moduli space of
opers \cite{O-Paris, OM5}.
The quantization parameter, the Planck constant
$\hbar$ of \eqref{Catalan}, acquires a geometric
meaning in this process. We begin the story with
finding a coordinate independent 
description of global differential 
equations of order $2$ 
on a compact Riemann surface.


\section{Projective structures,  opers,
and  Higgs bundles}

In \cite{OM1}, the authors have
given a definition of partial differential
equation version of topological recursion
for Hitchin spectral curves. When the
spectral curve is a double sheeted covering
of the base curve, we have shown that
this PDE topological recursion 
produces a quantum curve of the Hitchin spectral 
curve through WKB analysis. 
The mechanism is explained in detail
in \cite{OM2,OM4}. 

WKB analysis is certainly one way to describe
quantization. Yet the passage from spectral curves
to their quantization is purely geometric. 
This point of view is adopted in \cite{OM5}, based on 
our work \cite{DFKMMN}
on a conjecture of Gaiotto \cite{Gai}.
Our statement is that the quantization process is 
 a biholomorphic map from the moduli
space of Hitchin spectral curves to the moduli space
of \textbf{opers} \cite{BD}. 
In this section, we introduce the notion of opers, 
and construct the biholomorphic map mentioned above.
In this passage, we give a geometric interpretation 
of the Planck \emph{constant} $\hbar$
as a deformation parameter of vector bundles and
connections. 
 For  simplicity of presentation, 
we restrict our attention to $SL_2(\bC)$-opers.
We refer to \cite{DFKMMN,OM5}
 for more general cases.

To deal with linear differential equation of 
order higher then or equal to $2$ globally 
on a compact Riemann surface $C$, we need a
\emph{projective coordinate system}. 
If $\o$ is a global holomorphic or meromorphic
$1$-form on $C$, then 
\be\label{d+o}
(d+\o)f=0
\ee
is a first order linear differential equation 
that makes sense 
globally  on $C$. This is because
$d+\o$ is a homomorphism from $\cO_C$ to 
$K_C$, allowing singularities if necessary. Here,
$K_C$ is the sheaf of holomorphic $1$-forms
on $C$. Of course existence of a non-trivial global
solution of \eqref{d+o} is a different matter, because
it is equivalent to $\o= -d\log f$. 

Suppose we have a second order differential equation
\be\label{2nd}
\left(\frac{d^2}{dz^2} -q(z)\right)f(z) = 0
\ee
locally on $C$ with a local coordinate $z$. 
Clearly $\frac{d^2}{dz^2}$ is not globally defined, 
in contrast to the exterior differentiation $d$ in 
\eqref{d+o}.
What is the requirement for \eqref{2nd}
to make sense coordinate free, then?
Let us find an answer by imposing
\be\label{z=w}
dz^2 \left[\left(
\frac{d}{dz}\right)^2-q(z)\right]f(z) = 0
\Longleftrightarrow 
dw^2 \left[\left(
\frac{d}{dw}\right)^2-q(w)\right]f(w) = 0.
\ee
We wish the
 shape of the equation to be analogous to \eqref{d+o}.
Since $\o$ is a $1$-form, we impose
that $q\in H^0(C,K_C^{\tensor 2})$ is a global
quadratic differential on $C$. Under a coordinate 
change $w = w(z)$, $q$
  satisfies $q(z)dz^2 = q(w)dw^2$. 
  
  How should we think about a solution $f$?
  For \eqref{2nd} to have a coordinate free meeting,
  we need to identify  the line bundle $L$
  on $C$ to which  $f$  belongs. Let us
  denote by $e^{-g(w(z))}$ the transition
  function of $L$ with 
  respect to the coordinate patch $w = w(z)$.
  A function $g(w) = g\big(w(z)\big)$
   will be determined later.
A solution $f$ satisfies the coordinate condition
  \be\label{f}
  e^{-g\big(w(z)\big)}f\big(w(z)\big) = f(z).
  \ee
  Then 
  \begin{align*}
0&=dz^2 \cdot e^{g\big(w(z)\big)}\left[\left(
\frac{d}{dz}\right)^2-q(z)^2\right]f(z)
\\
&=
dz^2 \cdot e^{g\big(w(z)\big)}\left[\left(
\frac{d}{dz}\right)^2-q(z)^2\right]
e^{-g\big(w(z)\big)}f\big(w(z)\big)
\\
&=
dz^2\cdot\left(e^{g\big(w(z)\big)}
\frac{d}{dz}
e^{-g\big(w(z)\big)}\right)^2
f\big(w(z)\big)
-dw^2 q(w)f(w)
\\
&=
dz^2\cdot \left(\frac{d}{dz}-g_w(w)w'
\right)^2 
f\big(w(z)\big)
-dw^2 q(w)f(w)
\\
&=
dz^2\cdot \left[\left(\frac{d}{dz}\right)^2-
 2g_w(w)w'\frac{d}{dz}
-\left(\left(g_w(w)w'\right)'-
 \left(g_w(w)w'\right)^2\right)\right]
f\big(w(z)\big)
-dw^2 q(w)f(w)
\\
&=
dz^2\cdot \left[\left(f_w(w)w'\right)'-2g_w(w)w'
f_w(w)w'-\left(\left(g_w(w)w'\right)'-
 \left(g_w(w)w'\right)^2\right)f(w)\right]
 \\
 &\hskip1in
-dw^2 q(w)f(w)
\\
&=
f_{ww}(w)\left(\frac{dw}{dz}\right)^2dz^2
+f_w(w)w''dz^2-2g_w(w)f_w(w)dw^2
\\
&\hskip1in
-\left(\left(g_w(w)w'\right)'-
 \left(g_w(w)w'\right)^2\right)f(w)dz^2
-dw^2 q(w)f(w)
\\
&=dw^2\cdot \left[\left(\frac{d}{dw}\right)^2-q(w)
\right]f(w)
+f_w(w)\left(w''-2g_w(w)(w')^2\right)dz^2
\\
&\hskip1in
-\left(\left(g_w(w)w'\right)'-
 \left(g_w(w)w'\right)^2\right)f(w)dz^2,
\end{align*}
where $w' = dw/dz$. Therefore, for \eqref{z=w}
to hold, we need
\begin{align}
\label{w equation}
w''-2g_w(w)(w')^2&=0
\\
\label{pre schwarz}
\left(g_w(w)w'\right)'-
 \left(g_w(w)w'\right)^2&=0.
 \end{align}
 From \eqref{w equation} we find
 \be\label{g}
 g_w(w)w' =\half\; \frac{w''}{w'}.
 \ee
 Substitution of \eqref{g} in   \eqref{pre schwarz}
 yields
 \be\label{schwarz}
 s_z(w):=\left(\frac{w''(z)}{w'(z)}\right)'-\frac{1}{2}\left(
\frac{w''(z)}{w'(z)}\right)^2 = 0.
 \ee
 We thus encounter the \textbf{Schwarzian derivative}
 $s_z(w)$.
 Since \eqref{g} is equivalent to 
 $$
 g(w)' = \half \big(\log(w')\big)',
 $$
 we obtain $g(w) = \half \log w'$. Here, the constant of 
 integration is $0$ because $g(z)=0$ when $w=z$. 
Then \eqref{f} becomes
$$
f(z) = e^{-\half \log w'} f(w)
\Longleftrightarrow
f(z) = \sqrt{\frac{dz}{dw}}f(w),
$$
which identifies the line bundle to which $f$ belongs: 
$
f(z)\in K_C^{-\half}.
$
We conclude that the coordinate change $w=w(z)$
should satisfy  the vanishing of the Schwarzian derivative 
$s_z(w) \equiv 0$, and the solution $f(z)$ should 
be considered as a
(multivalued) section of the inverse half-canonical 
$K_C^{-\half}$. The vanishing of the 
Schwarzian derivative dictates us to use 
a complex projective coordinate system of $C$.

A holomorphic connection in a vector bundle $E$
on $C$ is a $\bC$-linear map
$
\nabla:E\lrar  K_C \tensor E
$
satisfying the Leibniz condition
$
\nabla(fs) = f\nabla (s) + df\tensor s
$
for $f\in \cO_C$ and $s\in E$. 
Since $C$ is a complex curve, every connection on $C$
is automatically flat. Therefore, $\nabla$ gives rise to 
a holonomy representation 
\be\label{holonomy}
\rho:\pi_1(C)\lrar G
\ee
of the fundamental group $\pi_1(C)$ of the curve $C$
into the structure  group $G$ of the vector bundle $E$. 
A flat connection  $\nabla$ is \textbf{irreducible}
if the image of the holonomy representation 
\eqref{holonomy} is Zariski dense in the complex algebraic
group $G$. In our case, since $G = SL_2(\bC)$,
this requirement is equivalent that $Im(\rho)$ contains
two non-commuting elements of $G$. 
The moduli space 
$\cM_{\deR}$ of irreducible holomorphic
connections $(E,\nabla)$ in a $G$-bundle
$E$ has been  constructed (see \cite{S}).

\begin{Def}[$SL_2(\bC)$-opers]
\label{def:oper}
Consider a point $(E,\nabla)\in \cM_{\deR}$
consisting of 
an irreducible holomorphic $SL_2(\bC)$-connection
$
\nabla:E\lrar E\tensor K_C
$
acting on a vector bundle $E$.
It is an
$SL_2(\bC)$\textbf{-oper} if 
there is a line subbundle $F\subset E$ such that
the connection induces an $\cO_C$-module
isomorphism
\be\label{grading}
\overline{\nabla}:
F\overset{\sim}{\lrar}
\left(E/F\right)\tensor K_C.
\ee
\end{Def}

The notion of oper is a generalization of projective
structures on a Riemann surface. 
Every compact
Riemann surface $C$ amsits a projective structure
subordinating the given complex structure
(see \cite{Gun}).
For our purpose of quantization of Hitchin spectral curves
associated with \emph{holomorphic} Higgs bundles,
let us assume that $g(C)\ge 2$ in what follows.
When we allow singularities, we can relax this condition
and deal with genus $0$ and $1$ cases. 
A \emph{complex projective coordinate system}
is a coordinate neighborhood covering 
$$
C = \bigcup_\a U_\a
$$
with a local coordinate $z_\a$ of $U_\a$
such that for every 
$U_\a \cap U_\b$,
we have a fractional linear transformation
\be
\label{Mobius}
z_\a = \frac{a_{\a\b}z_\b + b_{\a\b}}
{c_{\a\b}z_\b + d_{\a\b}}, 
\qquad f_{\a\b}:=
\begin{bmatrix}
a_{\a\b}&b_{\a\b}\\
c_{\a\b}&d_{\a\b}
\end{bmatrix}
\in SL_2(\bC),
\ee
satisfying a cocycle condition $[f_{\a\b}][f_{\b\gam}]=
[f_{\a\gam}]$. 
Here, $[f_{\a\b}]$ is 
the class of $f_{\a\b}$ in the projection
\be\label{psl2}
0\lrar \bZ/2\bZ \lrar SL_2(\bC)\lrar PSL_2(\bC)\lrar 0,
\ee
which determines the fractional linear transformation. 
A choice of $\pm$ on each $U_\a\cap U_\b$ 
 is an element of
$
H^1(C,\bZ/2\bZ) = (\bZ/2\bZ)^{2g},
$
indicating that there are $2^{2g}$ choices of
a lift. We make this 
choice once and for all,
and consider $f_{\a\b}$ an $SL_2(\bC)$-valued
$1$-cocycle. 
Since 
$$
dz_\a = \frac{1}{(c_{\a\b}z_\b + d_{\a\b})^2}
\; dz_\b,
$$
 a transition function for 
$K_C$  is given by
the cocycle 
$
\left\{(c_{\a\b}z_\b + d_{\a\b})^2\right\}$
on each  $U_\a\cap U_\b$.
A \emph{theta characteristic} (or a \emph{spin structure})
 $K_C^\half$ is the line bundle defined by the 
 $1$-cocycle
\be
\label{xiab}
\xi_{\a\b} =  c_{\a\b}z_\b + d_{\a\b}.
\ee
Here again, we have $2^{2g}$ choices
$\pm \xi_{\a\b}$ for a transition function
of $K_C^\half$. Since we have
already made a choice of the sign for $f_{\a\b}$,
we have a consistent choice in \eqref{xiab}, 
as explained below.
Thus we see that the choice of the lift
\eqref{psl2} is determined by $K_C^\half$. 
From \eqref{xiab}, we obtain
\be\label{zeta''}
\partial^2_\b \xi_{\a\b} = 0.
\ee
This  property plays an essential role
  in our construction of global 
connections on $C$. First we show 
 that actually \eqref{zeta''}
implies \eqref{Mobius}. 

\begin{prop}[A condition for  projective coordinate]
\label{proj condition}
A coordinate system of $C$
 with which the second derivative
of the transition function of $K_C^\half$ vanishes
is a projective coordinate system.
\end{prop}

\begin{proof} The condition \eqref{zeta''}
means that the transition function
of  $K_C^\half$  is a linear polynomial
$c_{\a\b}z_\b + d_{\a\b}$
satisfying the cocycle condition. Therefore, 
$$
\frac{dz_\a}{dz_\b} = \frac{1}{(c_{\a\b}z_\b+d_{\a\b})^2}.
$$
Solving this differential equation, we obtain
\be\label{int}
z_a = \frac{m_{\a\b}(c_{\a\b}z_\b+d_{\a\b})-1/c_{\a\b}}
{c_{\a\b}z_\b+d_{\a\b}}
\ee
with a  constant of integration $m_{\a\b}$. In this way
we find an element of $SL_2(\bC)$ on each
$U_\a\cap U_\b$. The cocycle condition makes
\eqref{int} 
 exactly
\eqref{Mobius}. 
\end{proof}

The transition function $f_{\a\b}$ defines a
rank $2$ vector bundle $E$ on $C$ whose
structure group is $SL_2(\bC)$. 
Since $f_{\a\b}$
 is a \emph{constant} element of $SL_2(\bC)$, 
 the notion of locally constant sections of $E$
  makes sense independent of the coordinate chart.
  Thus defining $\nab_\a = d$ on $U_\a$ for each $U_\a$ 
  determines a
  global connection $\nabla$ in $E$. 
Suppose we have a projective coordinate system
on $C$.  Let  $E$ be the vector bundle we 
have just constructed,
and $\bP(E)$ its projectivization, 
i.e.,  the  $\bP^1$ bundle associated with $E$. 
Noticing that $\Aut(\bP^1) = PSL_2(\bC)$, the local
coordinate system $\{z_\a\}$ of \eqref{Mobius} is a
global section of $\bP(E)$. Indeed, this section
defines a map
\be\label{za}
z_\a:U_\a\lrar \bP^1,
\ee
which induces the projective structure of $\bP^1$ into 
$U_\a$ via pull back. 
The map \eqref{za} is not a constant, because
its derivative $dz_\a$  never vanishes on 
the intersection $U_\a\cap U_\b$.
A global section of $\bP(E)$ corresponds to a 
line subbundle $F$ of $E$, 
 such that $E$ is realized as an extension
$
0\lrar F\lrar E\lrar E/F\lrar 0.
$
The equality
$$
\begin{bmatrix}
z_a\\1
\end{bmatrix}
= \frac{1}{c_{\a\b}z_\b+d_{\a\b}}
\begin{bmatrix}
a_{\a\b}&b_{\a\b}\\
c_{\a\b}&d_{\a\b}
\end{bmatrix}
\begin{bmatrix}
z_\b\\1
\end{bmatrix}
$$
shows that $\begin{bmatrix}
z_a\\1
\end{bmatrix}$ defines a global section of the
vector bundle $K_C^{-\half} \tensor E$, and that
the projectivization image of this section is the 
 global section $\{z_\a\}$ of
$\bP(E)$ corresponding to $F$. Here, the choice
of the theta characteristic is consistently made so that the 
$\pm$ ambiguity of \eqref{xiab}
and the one in the  lift of the 
fractional linear transformation to $f_{\a\b}\in 
SL_2(\bC)$ cancel. 
 Since this section is nowhere vanishing, 
it generates a trivial subbundle
$$
\cO_C=  K_C^{-\half}\tensor F\subset 
 K_C^{-\half} \tensor E.
 $$
Therefore,  
$F= K_C^\half$. Note that $\det E = \cO_C$, 
hence $E/F = K_C^{-\half}$, 
and $E$ is  an extension
\be
\label{ext}
0\lrar K_C^\half\lrar E\lrar K_C^{-\half}\lrar 0,
\ee
 determining  an element of 
$
\Ext^1\big(K_C^{-\half}, K_C^{\half}\big) 
\isom H^1(C,K_C) \isom \bC.
$

There is a more straightforward way to 
obtain \eqref{ext}.

\begin{thm}[Projective coordinate systems and opers]
Every projective coordinate system 
\eqref{Mobius} determines   an oper
$(E,\nab)$ of Definition~\ref{def:oper}
with $F = K_C^{\half}$. 
\end{thm}

\begin{proof}
Let $\sig_{\a\b}=-(d/dz_\b) \xi_{\a\b} = -c_{\a\b}$, where
$\xi_{\a\b}$ is defined by \eqref{xiab}. 
Since
\be\label{ext matrix}
\begin{aligned}
g_{\a\b}:&=
\begin{bmatrix}
\xi_{\a\b}&\sig_{\a\b}
\\
&\xi_{\a\b}^{-1}
\end{bmatrix}
=
\begin{bmatrix}
c_{\a\b}z_\b+d_{\a\b}&-c_{\a\b}
\\
&(c_{\a\b}z_\b+d_{\a\b})^{-1}
\end{bmatrix}
\\
&
=\begin{bmatrix}
&1
\\
-1&z_\a
\end{bmatrix}
\begin{bmatrix}
a_{\a\b}&b_{\a\b}
\\
c_{\a\b}&d_{\a\b}
\end{bmatrix}
\begin{bmatrix}
z_\b&-1
\\
1&
\end{bmatrix},
\end{aligned}
\ee
which follows from 
$$
a_{\a\b} - c_{\a\b}z_\a = \frac{a_{\a\b}(c_{\a\b}z_\b
+d_{\a\b}) -c_{\a\b}(a_{\a\b}z_\b
+b_{\a\b})}{(c_{\a\b}z_\b
+d_{\a\b})}
=\frac{1}{(c_{\a\b}z_\b
+d_{\a\b})},
$$
we find that $f_{\a\b}$ and $g_{\a\b}$ define the 
same $SL_2(\bC)$-bundle $E$. The shape of the matrix
$g_{\a\b}$ immediately shows \eqref{ext}. 
Since the connection $\nab$ in $E$ is simply $d$ on each
$U_\a$ with respect to $f_{\a\b}$, the differential operator on $U_\a$ with respect
to the transition function $g_{\a\b}$ is given by 
\be\label{naba}
\nab_\a := 
\begin{bmatrix}
&1
\\
-1&z_\a
\end{bmatrix}
d
\begin{bmatrix}
z_\a&-1
\\
1&
\end{bmatrix}
=d-
\begin{bmatrix}
0&0\\
1&0
\end{bmatrix}
dz_\a .
\ee
Since the $(2,1)$-component of the connection 
matrix is $dz_\a$ which is nowhere vanishing, 
\be\label{O-linear}
 F= K_C^\half \overset{\nab}{\lrar}
 E\tensor K_C\lrar (E/F)\tensor K_C \isom K_C^\half
\ee
given by this component is an isomorphism. This
proves that $(E,\nab)$ is an $SL_2(\bC)$-oper. 
\end{proof}

In the final step of the proof to show $(E,\nab)$ is an oper,
we need that \eqref{O-linear} is an $\cO_C$-linear 
homomorphism. This is because we are considering 
the difference of connections $\nab$ and
$\nab|_F$ in $E$. More generally, 
suppose we have two connections $\nab_1$ and
$\nab_2$ in the same vector bundle $E$. Then 
the Leibniz condition tells us that 
$$
\nab_1(fs)-\nab_2(fs) = f\nab_1(s)-f\nab_2(s)
$$
for $f\in \cO_C$ and $s\in E$. Therefore, 
$
\nab_1-\nab_2:E\lrar E\tensor K_C
$
is  an $\cO_C$-module homomorphism. 
Although the extension class of \eqref{ext}
is parameterized by $H^1(C,K_C) = \bC$, 
the complex structure of $E$ depends
only if $\sig_{\a\b} = 0$ or not. This is because
\be\label{const change}
\begin{bmatrix}
\lam\\
&\lam^{-1}
\end{bmatrix}
\begin{bmatrix}
\xi_{\a\b}&\sig_{\a\b}
\\
&\xi_{\a\b}^{-1}
\end{bmatrix}
\begin{bmatrix}
\lam^{-1}\\
&\lam
\end{bmatrix}
=
\begin{bmatrix}
\xi_{\a\b}&\lam^2 \sig_{\a\b}
\\
&\xi_{\a\b}^{-1}
\end{bmatrix},
\ee
hence we can normalize $\sig_{\a\b} = 0$ or 
$\sig_{\a\b} = 1$. The former case gives 
the trivial extension of two line bundles
$E=K_C^\half\dsum K_C^{-\half}$. 
Since $\sig_{\a\b}=c_{\a\b} = 0$, the 
projective coordinate system \eqref{Mobius}
is actually an \emph{affine} coordinate system.
Since we are assuming $g(C)>1$, there is no
affine structure in $C$. Therefore, 
only the latter case can happen. And
the latter case gives a non-trivial extension, as we 
will show later. 

Suppose we have another projective structure in 
$C$ subordinating the same complex structure of $C$. 
Then we can adjust the $\pm$ signs of the lift of \eqref{psl2}
and the square root of \eqref{xiab} so that we obtain
the exact same holomorphic
vector bundle $E$ of \eqref{ext}. Since we are dealing with 
a different coordinate system, the only change we have
is reflected in the connection $\nab$. Thus two different
projective structures give rise to two connections in the
same vector bundle $E$. Hence this difference is
an $\cO_C$-linear homomorphism $E\lrar E\tensor K_C$
as noted above.
This consideration motivates the following.

A \textbf{Higgs bundle} of rank $r$
 \cite{H1,H2} defined on $C$ 
is a pair $(E,\phi)$ consisting of a holomorphic
vector bundle $E$ of rank $r$ 
on $C$ and an $\cO_C$-module
homomorphism
$$
\phi:E\lrar E\tensor K_C.
$$
An $SL_2(\bC)$-Higgs bundle is a pair 
$(E,\phi)$  of rank $2$
with a fixed isomorphism $\det E = \cO_C$ and
$\tr \phi = 0$. It  is \emph{stable}
if  every line subbundle 
$F\subset E$ that is invariant with respect to 
$\phi$, i.e., $\phi:F\lrar F\tensor K_C$, 
has a negative degree
$
\deg F < 0.
$
 The moduli spaces of 
stable Higgs bundles
are constructed in \cite{S}. 
We denote by $\cM_{\Dol}$
the moduli space of 
stable holomorphic $SL_2(\bC)$-Higgs bundles on $C$. 
 It is diffeomorphic to the moduli space $\cM_{\deR}$ of
 pairs $(E,\nab)$ consisting of an
 irreducible holomorphic connection in an 
 $SL_2(\bC)$-bundle
 (see \cite{Donaldson, H1, S}). 
A particular diffeomorphism 
\be\label{NAH}
\nu:\cM_{\Dol}\overset{\sim}{\lrar}\cM_{\deR}
\ee
is the  \textbf{non-Abelian Hodge correspondence},
which is explained in Section~\ref{sect:NAH}.

The total space of the line bundle $K_C$ is the 
cotangent bundle $\pi:T^*C\lrar C$ of the curve $C$. 
We denote by $\eta\in H^0(T^*C,\pi^*K_C)$
the tautological section
$$
\xymatrix{&\pi^*K_C\ar[d]&K_C\ar[d]
\\
T^*C\ar[ur]^{\eta}\ar[r]^{=}&T^*C\ar[r]^{\pi}&C,
}
$$
which is a holomorphic $1$-form on $T^*C$. 
Since $\phi$ is an $\End(E)$-valued holomorphic
$1$-form on $C$, its eigenvalues are 
$1$-forms. The set of eigenvalues
is    thus  a multivalued
 section of $K_C$,
and hence a multivalued section of $\pi:T^*C\lrar C$.
The image $\S\subset T^*C$ of this multivalued section is the 
\textbf{Hitchin spectral curve},
 which defines a ramified covering of $C$.
 The formal definition of Hitchin spectral curve $\Sigma$ 
 is that it is the divisor of zeros   in $T^*C$ of 
 the characteristic
 polynomial  
 $$
 \det(\eta - \pi^*\phi) \in \pi^*K_C^{\tensor 2}.
 $$

Hitchin fibration  \cite{H1} is a holomorphic 
fibration
\be\label{Hitchin map}
\mu_H:\cM_{\Dol} \owns (E,\phi)
\longmapsto \det(\eta - \pi^*\phi)\in  B, \qquad
B := H^0\left(C,K_C^{\tensor 2}
\right),
\ee
that defines 
an algebraically completely integrable
Hamiltonian system in $\cM_{\Dol}$. 
Hitchin notices in \cite{H1} that
the choice of a spin structure $K_C^\half$ that
 we have made allows us to construct  
 a natural section 
$\kappa: B\hookrightarrow \cM_{\Dol}$. 
Define
\be
\label{TDS}
X_- :
=\begin{bmatrix}
0&0\\
1&0
\end{bmatrix},
\qquad  X_+ := X_-^t=
\begin{bmatrix}
0&1\\
0&0
\end{bmatrix},
\qquad 
 H :=[X_+,X_-]=
\begin{bmatrix}
1&0\\
0&-1
\end{bmatrix}.
\ee
These elements generate 
the Lie algebra $\la X_+,X_-,H\ra\isom 
sl_2(\bC)$. 

\begin{lem}
\label{lem:Hitchin section}
Let
$
q
\in B=
H^0(C,K_C^{\tensor 2})
$
be an arbitrary point of the Hitchin base $B$, and
define a Higgs bundle 
$\left(E_0,\phi({q})\right)$
consisting of a vector bundle
\be
\label{E0}
E_0 := \left(K_C^\half\right)^{\tensor H}
=K_C^\half\oplus K_C^{-\half}
\ee
and a Higgs field 
\be\label{phi q}
\phi({q}): = 
X_- +{q} X_+ =
\begin{bmatrix}
0&q\\
1&0
\end{bmatrix}.
\ee
Then it  is
a stable $SL_2(\bC)$-Higgs bundle.
The Hitchin section is defined by
\be\label{Hitchin section}
\kappa: B\owns {q}
\longmapsto \left(E_0,\phi({q})\right)
\in  \cM_{\Dol},
\ee
which gives a biholomorphic map between 
$B$ and $\kappa(B)\subset \cM_{\Dol}$. 
\end{lem}

\begin{proof}
We first note that $X_-:E_0\lrar E_0\tensor K_C$ 
is a globally defined $\End_0(E_0)$-valued 
$1$-form, since it is essentially the constant  map
\be\label{si}
1: K_C^\half
\overset{=}{\lrar}
K_C^{-\half}
\tensor K_C.
\ee
Similarly,  multiplication by a quadratic differential gives
$$
{q}:K_C^{-\half}
{\lrar}
K_C^{\frac{3}{2}}
=
K_C^\half
\tensor K_C.
$$
Thus $\phi({q}):E_0\lrar E_0\tensor K_C$
is globally defined as a Higgs field in $E_0$.
The Higgs pair is  stable  because no line subbundle of 
$E_0$ is invariant under $\phi({q})$, unless
${q} = 0$. And when ${q}=0$, the invariant line subbundle
$K_C^\half$
has  degree $g-1$, which is positive since $g\ge 2$.
\end{proof}

\begin{rem}
Hitchin sections exist for the moduli space
of stable $G$-Higgs bundles for an arbitrary
simple complex algebraic group $G$. The construction
 utilizes Kostant's 
\emph{principal three-dimensional subgroup}
(TDS)
\cite{Kostant}. The use of TDS is crucial in 
our quantization, as noted in \cite{DFKMMN, OM5}.
\end{rem}

 The image $\kappa(B)$ is a holomorphic
 Lagrangian submanifold of a holomorphic
 symplectic space $\cM_{\Dol}$. The holomorphic 
 symplectic structure of $\cM_{\Dol}$
 is induced from its open dense subspace
 $T^*\cS\cU(2,C)$, where $\cS\cU(2,C)$
 is the moduli space of rank $2$ stable bundles of 
 degree $0$ on $C$. Since the codimension of
 the complement of $T^*\cS\cU(2,C)$ in $\cM_{\Dol}$
 is  $2$, the natural holomorphic symplectic form
 on the cotangent bundle automatically extends
  to $\cM_{\Dol}$.

Our first step of constructing the quantization 
of the Hitchin spectral curve $\S$ is 
to define  $\hbar$-connections on $C$ 
that are holomorphically
depending on $\hbar$.
We use a 
 one-parameter  
family $\cE$ 
of deformations of vector bundles
$$
\xymatrix{
E_\hbar \ar[d]\ar[r]&\cE\ar[d]
\\
C\times \{\hbar\}\ar[r] &C\times H^1(C,K_C),
}
$$
and a $\bC$-linear first-order differential operator
$
\hbar \nabla^\hbar:E_\hbar \lrar E_\hbar \tensor
K_C
$
depending holomorphically on 
$\hbar\in H^1(C,K_C)\isom \bC$
for $\hbar \ne 0$. Here, we identify 
the Planck constant $\hbar$ of the
quantization as a geometric parameter
$$
\hbar \in H^1(C,K_C) =
\Ext^1\left(K_C^{-\half},K_C^\half\right) 
\isom  \bC,
$$
which determines a unique extension
\be
\label{r=2 extension}
0\lrar K_C^\half \lrar E_\hbar \lrar  K_C^{-\half} \lrar 0.
\ee
This is exactly the same as \eqref{ext}. 
The extension $E_\hbar$ 
 is given by
a system of transition functions
\be
\label{r=2 trans}
E_\hbar \longleftrightarrow
\left\{
\begin{bmatrix}
\xi_{\a\b} &\hbar \sig_{\a\b}\\
0&\xi_{\a\b}^{-1}
\end{bmatrix}
\right\}
\ee
on each $U_\a\cap U_\b$. The cocycle condition 
for the transition functions translates into a condition
\be
\label{sigma relation}
\sig_{\a\gam} = \xi_{\a\b}\sig_{\b\gam}
+ \sig_{\a\b}\xi_{\b\gam}^{-1}.
\ee
The application of the exterior 
differentiation $d$ to the cocycle condition
$\xi_{\a\gam} = \xi_{\a\b}\xi_{\b\gam}$ 
for $K_C^\half$ yields
$$
\frac{d\xi_{\a\gam}}{dz_\gam}dz_\gam
= \frac{d\xi_{\a\b}}{dz_\b}dz_\b \xi_{\b\gam}
+\xi_{\a\b}
\frac{d\xi_{\b\gam}}{dz_\gam}dz_\gam.
$$
Noticing 
$
\xi_{\a\b}^2 = \frac{dz_\b}{dz_\a},
$
we find that 
\be
\label{sigma ab}
\sig_{\a\b} :=- \frac{d\xi_{\a\b}}{dz_\b}
=-\partial_\b \xi_{\a\b}
\ee
solves
\eqref{sigma relation}. The negative sign
is chosen to relate \eqref{r=2 trans} and \eqref{Mobius}.
By the same
reason as before, 
the complex structure of the vector bundle $E_\hbar$ is isomorphic to $E_1$ if $\hbar \ne 0$, and to $E_0$ of \eqref{E0} if $\hbar = 0$. 
The transition function can also be written as
\be\label{gab}
\begin{bmatrix}
\xi_{\a\b}&\hbar \sig_{\a\b}
\\
&\xi_{\a\b}^{-1}
\end{bmatrix}
=
\exp\left(
\log \xi_{\a\b} {\begin{bmatrix}
1&0\\
0&-1
\end{bmatrix}}
\right)
\exp\left(
-\hbar \partial_\b \log \xi_{\a\b}
\begin{bmatrix}
0&1
\\
0&0
\end{bmatrix}
\right).
\ee
Therefore, in the multiplicative sense, the
extension class is determined by
$\partial_\b \log \xi_{\a\b}$.

\begin{lem}
The extension class
 $ \sig_{\a\b}$ of  \eqref{sigma ab}
 defines a non-trivial extension
 \eqref{r=2 extension}.
\end{lem}

\begin{proof} 
The long exact sequences
of cohomologies
$$
\xymatrix{
H^1(C,\bC)\ar[r]\ar[d]& H^1(C,\cO_C)
\ar[r]\ar[d]& H^1(C,K_C) \ar[r]^{\sim}\ar[d]_{\parallel}
& H^2(C,\bC)
\\
H^1(C,\bC^*)\ar[r]\ar[d]_0 &H^1(C,\cO^*_C)
\ar[r]^{d\log}\ar[d]_{c_1} &H^1(C,K_C)
\\
H^2(C,\bZ)\ar[r]^{=}&H^2(C,\bZ)
}
$$
associated with 
 exact sequences of sheaves 
$$
\xymatrix{
&0\ar[d]&0\ar[d]
\\
0\ar[r]&\bZ\ar[d]\ar[r]^{=}&\bZ\ar[d]\ar[r]&0\ar[d]
\\
0\ar[r]&\bC\ar[d]\ar[r]&\cO_C\ar[d]\ar[r]^d&K_C\ar[d]\ar[r]
&0
\\
0\ar[r]&\bC^*\ar[d]\ar[r]&\cO_C^*\ar[d]\ar[r]^{d\log}
&K_C\ar[d]\ar[r]&0
\\
&0&0&0
}
$$
show that the class $\{\sig_{\a\b}\}$ corresponds
to the image
of $\{\xi_{ab}\}$ via the map
$$
\begin{CD}
H^1(C,\cO_C ^*)@>{d\log}>>
H^1(C,K_C).
\end{CD}
$$
If  $d\log \{\xi_{\a\b}\} = 0\in H^1(C,K_C)$, 
then it comes from a class in $H^1(C,\bC^*)$, 
which is the moduli space  of 
line bundles with holomorphic connections,
as explained in \cite{Bloch}. It leads to a 
contradiction
$$
0 = c_1\big(K_C^\half\big) = g-1> 0,
$$
because $g(C)\ge 2$. 
\end{proof}

\begin{rem}
Gukov and Su\l kowski \cite{GS}
defines an intriguing   quantizability condition for a spectral curve
in terms of the algebraic K-group  $K_2\big(\bC(\S)\big)$ of 
the function field of  spectral curve $\S$.
They relate the quantizability and 
 Bloch regulators  of \cite{Bloch}.
\end{rem}

The class $\{\sigma_{\a\b}\}$
of \eqref{sigma ab} gives a natural isomorphism
$H^1\left(C,K_C\right)\isom\bC$. We identify
the deformation parameter $\hbar\in \bC$
with the cohomology class $\{\hbar \sig_{\a\b}\}
\in H^1\left(C,K_C\right)=\bC$. 
We trivialize the line bundle $K_C^{\tensor 2}$
with respect to a coordinate chart 
$C = \bigcup_\a U_\a$, and write  ${q}\in H^0(C,K_C^{\tensor} 2)$
as $\{({q})_\a\}$ that satisfies the 
transition relation
\be
\label{q ell relation}
({q})_\a =  ({q})_\b \xi_{\a\b}^{4}.
\ee
The  transition function 
of the trivial extension $E_0$ is given by
\be
\label{xi H}
\xi_{\a\b}^H = \exp(H \log \xi_{\a\b}). 
\ee
Since $X_-:E_0\lrar E_0\tensor K_C$ is 
a globally defined Higgs filed, its local expressions
$\{X_- dz_\a\}$
with respect to a coordinate system
satisfies the transition relation
\be
\label{X- local}
X_-dz_\a = \exp(H \log \xi_{\a\b})
X_-dz_\b \exp(-H \log \xi_{\a\b})
\ee
on   every $U_\a\cap U_\b$.
The same relation holds for the Higgs field
$\phi({q})$ as well:
\be
\label{Higgs local}
\phi_\a({q})dz_\a = \exp(H \log \xi_{\a\b})
\phi_\b({q})dz_\b \exp(-H \log \xi_{\a\b}).
\ee

\begin{thm}[Construction of $SL_2(\bC)$-opers]
\label{thm:construction of opers}
On each $U_\a\cap U_\b$ define a 
transition function 
\be
\label{g hbar}
g_{\a\b}^\hbar := 
\exp(H \log \xi_{\a\b})
\exp\big(-\hbar \partial_\b \log\xi_{\a\b}
X_+\big)
=\begin{bmatrix}
\xi_{\a\b}\\
&\xi^{-1}_{\a\b}
\end{bmatrix}
\cdot
\begin{bmatrix}
1&-\hbar \partial_\b \log\xi_{\a\b}\\
&1
\end{bmatrix},
\ee
where $\partial_\b = \frac{d}{dz_\b}$, 
and $\hbar \partial_\b \log\xi_{\a\b}\in H^1(C,K_C)$.
Then
\begin{itemize}

\item The collection 
$\{g_{\a\b}^\hbar\}$ satisfies the cocycle condition
\be
\label{f cocycle}
g_{\a\b}^\hbar g_{\b\gam}^\hbar = g_{\a\gam}^\hbar,
\ee
which defines the holomorphic vector bundle bundle 
$E_\hbar$ of \eqref{r=2 extension}.

\item The local expression 
\be
\label{d+X}
\nabla_\a ^\hbar(0) :=d-\frac{1}{\hbar}X_- dz_a
\ee
on $U_\a$ for  $\hbar \ne 0$
defines a global holomorphic connection in 
$E_\hbar$, i.e., 
\be
\label{d+X gauge equation}
d-\frac{1}{\hbar}X_-dz_\a 
=g_{\a\b}^\hbar \left(d -\frac{1}{\hbar}
X_-dz_\b\right) \left(g_{\a\b}^\hbar\right)^{-1},
\ee
if and only if the coordinate  is a projective
coordinate system. We choose one.

\item With this particular projective
coordinate system,  every point $\left(E_0,\phi({q})\right)
\in\kappa (B)\subset \cM_{\Dol}$ of the Hitchin 
section \eqref{Hitchin section} 
 gives rise
to a one-parameter
family of  $SL_2(\bC)$-opers  $\left(E_\hbar,
\nabla^\hbar({q})\right)\in \cM_{\deR}$. 
In other words, the local
expression
\be
\label{nabla q}
\nabla_\a^\hbar({q}):= d -
\frac{1}{\hbar}\phi_\a({q})dz_\a
\ee
on every $U_\a$ for $\hbar \ne 0$
determines a global holomorphic connection 
\be
\label{nabla q gauge}
\nabla^\hbar_\a({q}) 
=
 g_{\a\b}^\hbar 
\nabla^\hbar_\b({q}) 
\left(g_{\a\b}^\hbar\right)^{-1}
\ee
in
$E_\hbar$ 
satisfying the oper condition.

\item  
Deligne's $\hbar$-connection 
\be
\label{Deligne}
\hbar \nabla^\hbar ({q}):E_\hbar\lrar E_\hbar\tensor K_C
\ee
interpolates the Higgs pair and the oper, i.e.,
at $\hbar = 0$, the family \eqref{Deligne}
gives the Higgs pair 
$\left(E,-\phi({q})\right)\in \cM_{\Dol}$,
and at $\hbar = 1$ it gives an 
$SL_2(\bC)$-oper  $\left(E_1,\nabla^1({q})\right)
\in \cM_{\deR}$.

\item After a suitable gauge transformation
depending on $\hbar$,
the $\hbar \rar \infty$ limit of 
the oper $\nabla^\hbar ({q})$ exists and is equal to 
$\nabla^{\hbar = 1}(0)$. 
This point corresponds to the $\bC^*$-fixed point
on the Hitchin section. 
\end{itemize}
\end{thm}

\begin{proof}
The cocycle condition of $g_{\a\b}$ 
has been established in \eqref{sigma ab} and
\eqref{gab}. 
Proof of  \eqref{d+X gauge equation} is a
straightforward calculation, using
 the power series expansion of the 
adjoint action
\be\label{adjoint formula}
e^{\hbar A}Be^{-\hbar A}
= 
\sum_{n=0}^\infty 
\frac{1}{n!} \hbar ^n (\rm{ad}_A)^n(B)
:=
\sum_{n=0}^\infty
\frac{1}{n!}\; \hbar ^n\;
\overset{n}
{\overbrace{[A,[A,[\cdots,[A}},B]\cdots]]].
\ee
It follows that
\begin{align*}
&g_{\a\b}^\hbar X_-\left(g_{\a\b}^\hbar \right)^{-1}
\\
&=
 \exp(H \log \xi_{\a\b})
\exp\left(-\hbar \partial_\b \log\xi_{\a\b}
X_+\right)X_-
\exp\left(\hbar \partial_\b \log\xi_{\a\b}
X_+\right)
\exp(-H \log \xi_{\a\b})
\\
&=
 \exp(H \log \xi_{\a\b})
X_-
\exp(-H \log \xi_{\a\b})
- \hbar \partial_\b \log\xi_{\a\b}H
\\
&\qquad
-  \hbar^2 (\partial_\b \log\xi_{\a\b})^2
\exp(H \log \xi_{\a\b})
X_+
\exp(-H \log \xi_{\a\b}).
\end{align*}
Note that \eqref{zeta''} is equivalent to 
$$
\partial_\b \partial_\b\log\xi_{\a\b}
=
\partial_\b \left(\xi_{\a\b}^{-1}\partial_\b\xi_{\a\b}\right)
=
-\xi_{\a\b}^{-2}(\partial_\b\xi_{\a\b})^2
= 
- (\partial_\b \log\xi_{\a\b})^2,
$$
hence to 
$$
\partial_\b g_{\a\b}^\hbar \left(g_{\a\b}^\hbar \right)^{-1}
=
\partial_\b \log\xi_{\a\b}H 
+\hbar (\partial_\b \log\xi_{\a\b})^2
\exp(H \log \xi_{\a\b})
X_+
\exp(-H \log \xi_{\a\b}).
$$
Therefore, noticing \eqref{X- local}, 
\eqref{zeta''} is equivalent to 
\begin{align*}
\left(
\frac{1}{\hbar}
g_{\a\b}^\hbar X_- \left(g_{\a\b}^\hbar \right)^{-1}
+
\partial_\b g_{\a\b}^\hbar \left(g_{\a\b}^\hbar \right)^{-1} 
\right) dz_\b
&=
\frac{1}{\hbar} \exp(H \log \xi_{\a\b})
X_-dz_\b
\exp(-H \log \xi_{\a\b})
\\
&=
\frac{1}{\hbar} X_-dz_\a.
\end{align*}
The statement follows from 
Proposition~\ref{proj condition}.

To prove \eqref{nabla q gauge},
we need, 
  in addition to 
 \eqref{d+X gauge equation},
 the following relation:
\be
\label{q gauge computation}
 ({q})_\a X_+
 dz_\a = g_{\a\b}^\hbar 
 ({q})_\b X_+ dz_\b\left(g_{\a\b}^\hbar\right)^{-1}.
\ee
But \eqref{q gauge computation} is obvious from 
 \eqref{Higgs local}
and \eqref{g hbar}.

The line bundle $F$ required in the definition 
of  $SL_2(\bC)$-oper is simply $K_C^\half$. 
The isomorphism \eqref{grading} is 
a consequence of \eqref{si}. 
Finally, the gauge transformation of 
$\nabla^\hbar({q})$ by a bundle automorphism 
\be\label{hbar gauge transf 1}
\hbar^{-\frac{H}{2}}
=
\begin{bmatrix}
\hbar^{-\frac{1}{2}}&
\\
&\hbar^{\frac{1}{2}}
\end{bmatrix}
\ee
on each coordinate neighborhood $U_\a$ 
gives
\be\label{hbar gauge transf 2}
d-\frac{1}{\hbar}\phi({q})
\longmapsto 
\hbar^{-\frac{H}{2}}
\left(
d-\frac{1}{\hbar}\phi({q})
\right)
\hbar^{\frac{H}{2}}
=
d-\left(X_-+  \frac{{q}}{\hbar^2} 
X_+\right).
\ee
This is because
$$
\hbar^{-\frac{H}{2}}X_-\hbar^{\frac{H}{2}}
= \hbar X_- \qquad \text{and}\qquad 
\hbar^{-\frac{H}{2}}X_+^2 \hbar^{\frac{H}{2}}
= \hbar^{-2} X_+^2,
$$
which follows from the adjoint formula 
\eqref{adjoint formula}.
Therefore,
$$
\lim_{\hbar\rar \infty}\nabla^\hbar({q}) \sim
d-X_- = \nabla^{\hbar=1}(0),
$$
where the symbol $\sim$ means gauge equivalence. 
This completes the proof of the theorem.
\end{proof}

The  construction theorem
yields the following.

\begin{thm}[Biholomorphic quantization
of Hitchin spectral curves]
\label{thm:quantization-holomorphic}
Let $C$ be a compact Riemann surface 
of genus $g\ge 2$
with a chosen projective coordinate system 
subordinating its complex structure. 
We denote by $\cM_{\Dol}$ the moduli
space of stable holomorphic $SL_2(\bC)$-Higgs
bundles over $C$, and by $\cM_{\deR}$ the
moduli space of irreducible 
holomorphic $SL_2(\bC)$-connections
on $C$. For a fixed theta characteristic
$K_C^\half$, we have a Hitchin section 
$\kappa(B)\subset \cM_{\Dol}$ of \eqref{Hitchin section}. 
We denote by $Op\subset \cM_{\deR}$ the moduli
space of $SL_2(\bC)$-opers with the
condition that the required line bundle 
is given by 
$F = K_C^{\half}$. 
Then the map
\be\label{biholomorphic quantization}
\cM_{\Dol}\supset 
\kappa(B)\owns \left(E_0,\phi({q})\right)
\overset{\gam}{\longmapsto} 
\left(E_\hbar,\nabla^\hbar({q})\right)
\in Op\subset \cM_{\deR}
\ee
evaluated at $\hbar = 1$
is a biholomorphic map with respect to the 
natural complex structures induced from the 
ambient spaces. 

The biholomorphic quantization 
\eqref{biholomorphic quantization} is also 
$\bC^*$-equivariant. The
$\lam\in \bC^*$ action on the Hitchin 
section is defined by
$\phi\longmapsto \lam\phi$. 
The  oper  corresponding to $\left(E_0,\lam\phi({q})\right)
\in \kappa(B)$ is 
$d-\frac{\lam}{\hbar}\phi({q})$. 
\end{thm}

\begin{proof}
The $\bC^*$-equivariance follows from the
same argument of the gauge transformation
\eqref{hbar gauge transf 1}, \eqref{hbar gauge transf 2}.
The action $\phi\longmapsto \lam\phi$
on the Hitchin section induces a weighted 
action 
$$
B\owns {q}
\longmapsto \lam^2 {q}\in B
$$
through $\mu_H$. 
Then we have the gauge equivalence via the gauge 
transformation 
$\left(\frac{\lam}{\hbar}\right)^{\frac{H}{2}}$:
$$
d - \frac{\lam}{\hbar}\phi({q})
\sim 
\left(\frac{\lam}{\hbar}\right)^{\frac{H}{2}}
\left(
d - \frac{\lam}{\hbar}\phi({q})
\right)
\left(\frac{\lam}{\hbar}\right)^{-\frac{H}{2}}
=
d -\left( X_- +\frac{\lam^2 {q}}
{\hbar^2} X_+\right).
$$
\end{proof}

\begin{rem}
In the construction theorem, our
use of a projective coordinate system is
essential, through 
\eqref{xiab}.  Only in such a coordinate,
our particular  definition 
\eqref{nabla q} makes sense. This is due
to the vanishing of the second derivative
of $\xi_{\a\b}$. And as we have seen above,
the projective coordinate system determines the \emph{origin}
$\nab^1(0)$ of the space $Op$ of opers. Other
opers are simply translation $\nab^1(q)$ from the
origin by $q\in H^0(C,K_C^{\tensor 2})$. 
\end{rem}


\section{Semi-classical limit of 
$SL_2(\bC)$-opers}
\label{sect:SCL}

A holomorphic connection on a 
compact Riemann surface $C$ is automatically flat.
Therefore, it defines a $\cD$-module over
$C$. Continuing the last section's conventions,
let us fix a projective coordinate system on $C$,
and let $\left(E_0,\phi({q})\right) = \kappa({q})$
be a point on the Hitchin section of 
\eqref{Hitchin section}. 
It uniquely defines an $\hbar$-family of opers
$\left(E_\hbar,\nabla^\hbar({q})\right)$.  

In this section, we establish that 
the $\hbar$-connection $\hbar \nabla^\hbar({q})$
defines a family of  $\cD$-modules on 
$C$ parametrized by $B$ such that 
the semi-classical limit of the family agrees with the 
family of spectral curves
 over $B$. 

To calculate the semi-classical limit,
let us trivialize the vector bundle $E_\hbar$
on each simply connected coordinate neighborhood
 $U_\a$ with 
  coordinate $z_\a$ of the chosen 
 projective coordinate system.
 A flat section $\Psi_\a$ of $E_\hbar$ over
 $U_\a$ is a solution of  
 \be\label{trivialization}
 \hbar \nabla^\hbar_\a({q})\Psi_\a:=
\left( \hbar d - \phi_\a ({q})\right) 
\begin{bmatrix}
\hbar \psi
\\
\psi
\end{bmatrix}_\a
=
0,
 \ee
 with an appropriate unknown function $\psi$. 
Since $\Psi_\a = g_{\a\b}^\hbar \Psi_\b$, 
the function $\psi$ on $U_\a$
satisfies the transition relation
$(\psi)_\a = \xi_{\a\b}^{-1}(\psi)_\b$. It means
that
$\psi$ is  actually
a  local section of the line bundle 
$K_C^{-\frac{1}{2}}$. 
There are two linearly independent solutions
of \eqref{trivialization}, because
${q}$ is a
holomorphic function on $U_\a$.
Since $\phi({q})$ is independent of $\hbar$ and takes
 the form
\be
\phi({q}) 
=
\begin{bmatrix}
0&{q}
\\
1&0
\end{bmatrix},
\ee
we see that
\eqref{trivialization}
is equivalent to 
  the second order equation
 \be\label{order 2}
\hbar^2 \psi'' -{q}\psi =0
\ee
for $\psi\in K_C^{-\frac{1}{2}}$. 
Since we are using a fixed projective coordinate
system, the connection $\nabla^\hbar({q})$
takes the same form on each coordinate
neighborhood $U_\a$. Therefore,
the shape of the differential equation
of \eqref{order 2} as an equation for 
$\psi$ is again the
same on every coordinate neighborhood,
as we wished to achieve in \eqref{2nd}.
This is exactly what
we  refer to as the
\textbf{quantum curve} of the 
spectral curve $\det(\eta-\phi({q}))=0$.
It is now obvious to calculate the semi-classical
limit of the  $\cD$-module
corresponding to $\hbar\nabla^\hbar({q})$.

\begin{thm}[Semi-classical limit of
an oper]
\label{thm:SCL}
Under the same setting of Theorem~\ref{thm:quantization-holomorphic}, let
$\cE({q})$ denote the  $\cD$-module
$\left(E_\hbar, \hbar\nabla^\hbar({q})\right)$
associated with the oper of \eqref{biholomorphic quantization}. Then the semi-classical limit 
of $\cE({q})$ is the spectral curve $\Sigma\subset T^*C$
of $\phi({q})$
defined by the characteristic equation
$\det(\eta-\phi({q}))=0$.
\end{thm}

The semi-classical limit of \eqref{order 2}
is the limit
\be
\label{char poly = qc}
\lim_{\hbar \rar 0}
e^{-\frac{1}{\hbar}S_0(z_\a)}
\left[\hbar^2\left(\frac{d}{dz_\a}
\right)^2 -{q}\right]
e^{\frac{1}{\hbar}S_0(z_a)}
=
y^2-{q},
\ee
where $S_0(z_\a)$ is a holomorphic 
function on $U_\a$ so that $dS_0 = ydz_\a$
gives a local trivialization of $T^*C$ over 
$U_\a$. 
The computation of semi-classical limit
is the same as the calculation of 
the determinant
of the  connection $\hbar\nabla^\hbar({q})$, after 
taking conjugation by the scalar  diagonal 
matrix  $e^{-\frac{1}{\hbar}S_0(z_\a)}I_{2\times 2}$, 
and then take the limit as $\hbar$ goes to $0$.

For every $\hbar\in H^1(C,K_C)$, 
the $\hbar$-connection 
$\left(E_\hbar,\hbar \nabla^\hbar({q})\right)$
of \eqref{Deligne} defines a global  
$\cD_C$-module structure in $E_\hbar$. 
Thus we have constructed a universal family 
$\cE_C$ of 
$\cD_C$-modules on a given $C$
with a fixed spin structure  and a projective structure:
\be\label{family D-modules}
\xymatrix{
\cE_C \ar[d]
& \left(E_\hbar,\nabla^\hbar({q})\right)\ar[d]
\ar[l]_{\supset}
\\
C\times B\times H^1(C,K_C)
&C\times \{{q}\}\times\{\hbar\}\ar[l] .
}
\ee
The universal family $\cS_C$ of spectral curves is defined 
over $C\times B$.
\be\label{family spectral}
\xymatrix{
\bP\left(K_C\dsum \cO_C\right)\times B\ar[d]
& \cS_C \ar[d]\ar[l]
& \left(\det\big(\eta-\phi({q})\big)\right)_0
\ar[d]\ar[l]
\\
C\times B & C\times B\ar[l]_=
& C\times \{{q}\}\ar[l]_\supset .
}
\ee
The semi-classical limit is thus a map of families
\be\label{family SCL}
\xymatrix{
\cE_C \ar[d]\ar[r]& \cS_C\ar[d]
\\
C\times B\times H^1(C,K_C)\ar[r]
&C\times B.
}
\ee


\section{Non-Abelian Hodge correspondence
between Hitchin moduli spaces
}
\label{sect:NAH}

The biholomorphic map \eqref{biholomorphic quantization}
is defined by fixing a projective structure of 
the base curve $C$. Gaiotto \cite{Gai} 
conjectured that such a correspondence would be
canonically constructed through a 
\emph{scaling limit} of non-Abelian Hodge 
correspondence.
The conjecture has been solved in 
\cite{DFKMMN} for \emph{holomorphic}
Hitchin moduli spaces
$\cM_{\Dol}$ and $\cM_{\deR}$ 
constructed over an arbitrary 
complex simple and simply connected 
Lie group $G$. 
In this section, we review the main result of
\cite{DFKMMN} 
for $G=SL_2(\bC)$ and
compare it with our quantization.

 We denote by $E^{\text{top}}$
the topologically trivial complex
vector bundle of 
rank $2$ on a compact Riemann surface $C$
of genus $g\ge 2$. 
The  correspondence
between stability conditions of
holomorphic vector bundles
on $C$ and PDEs on differential geometric data is used in 
 Narasimhan-Seshadri \cite{NS} to obtain 
 topological structures of the moduli space of stable bundles
(see also \cite{AB, MFK}). Extending
this classical case, 
the stability
condition for an $SL_2(\bC)$-Higgs
bundle  $(E,\phi)$ translates into 
a system of PDEs, known as
\emph{Hitchin's equations}, imposed on 
a set of geometric data  \cite{Donaldson, H1, S}.
The data we need are a Hermitian fiber metric $h$
on $E^{\rm{top}}$, a unitary connection 
$D$ in $E^{\rm{top}}$
with respect to $h$, and a differentiable 
$sl_2(\bC)$-valued
$1$-form $\phi$ on $C$. In this section we use $D$
for unitary connections to avoid confusion with 
holomorphic connections we have been using until
the last section. 
 \textbf{Hitchin's equations} are the following system of 
 nonlinear PDEs.
\be\label{Hitchin's}
\begin{cases}
F_D + [\phi, \phi^\dagger] = 0\\
D^{0.1}\phi = 0.
\end{cases}
\ee
Here, $F_D$ denotes  the curvature
$2$-form  of $D$, 
$\phi^\dagger$ is the Hermitian conjugate 
of $\phi$ with respect to 
the metric $h$, and $D^{0,1}$ is the Cauchy-Riemann
part of $D$ defined by
the complex structure of the base
curve  $C$. $D^{0,1}$ determines a complex structure in $E^{\rm{top}}$,
which we simply denote by $E$. Then
$\phi$ becomes a holomorphic Higgs field in $E$
because it satisfies the Cauchy-Riemann equation
\eqref{Hitchin's}. The pair $(E,\phi)$ 
 constructed in this way 
 from a solution of Hitchin's 
 equations is a \emph{stable} Higgs bundle.
 Conversely \cite{S}, a stable Higgs bundle
 $(E,\phi)$ gives rise to a unique 
 harmonic Hermitian metric $h$ and the Chern 
 connection $D$ with respect to $h$ so that 
 the data satisfy Hitchin's equations. 
The stability condition
for the holomorphic Higgs pair $(E,\phi)$ is  thus
translated into 
\eqref{Hitchin's}.

Define a one-parameter family
of connections
\be\label{twister}
D(\zeta) := \frac{1}{\zeta}\cdot \phi + D + \zeta \cdot \phi^\dagger,
\qquad
\zeta\in \bC^*.
\ee
Then the flatness of $D(\zeta)$ for all $\zeta$
is equivalent to \eqref{Hitchin's}. 
The \emph{non-Abelian Hodge correspondence}
\cite{Donaldson, H1, Mochizuki, S}
is a diffeomorphic correspondence
$$
\nu:\cM_{\Dol} \owns (E,\phi)
\longmapsto (\widetilde{E},\widetilde{\nabla})
\in \cM_{\deR}.
$$
Proving the diffeomorphism of these moduli spaces
is far beyond of the scope
of this article. Here, we only give
the definition of the map $\nu$. 
We start with the solution $(D,\phi,h)$ of 
Hitchin's equations corresponding to 
a stable Higgs bundle $(E,\phi)$. It induces a family of flat 
connections $D(\zeta)$.
Define a complex structure $\widetilde{E}$
in $E^{\rm{top}}$ by $D(\zeta=1)^{0,1}$. 
Since $D(\zeta)$ is flat, $\widetilde{\nabla}:=
D(\zeta=1)^{1,0}$ is automatically a holomorphic
connection in $\widetilde{E}$. 
 Stability
of $(E,\phi)$ implies 
that the resulting holomorphic connection  is irreducible, hence
$(\widetilde{E},\widetilde{\nabla})\in \cM_{\deR}$.
Since this correspondence goes through the
real unitary connection $D$, the change of the 
complex structure of $E$ to that of 
$\widetilde{E}$ is not a 
holomorphic deformation.

Extending the idea of  one-parameter
family \eqref{twister}, Gaiotto
conjectures:

\begin{conj}[Gaiotto \cite{Gai}]
\label{conj:Gaiotto}
Let $(D, \phi, h)$ be the solution
of \eqref{Hitchin's} corresponding
to a sable Higgs bundle 
$\left(E_0,\phi({q})\right)$ on the 
$SL_2(\bC)$-Hitchin 
section \eqref{Hitchin section}. 
Consider
the following two-parameter family of 
connections 
\be\label{two}
D(\zeta,R) := \frac{1}{\zeta}\cdot R\phi + D + \zeta \cdot R\phi^\dagger,
\qquad
\zeta\in \bC^*, R\in \bR_+.
\ee
Then the scaling limit
\be\label{scaling}
\lim_{\substack{R\rar 0,\zeta\rar 0\\
\zeta/R = \hbar}} D(\zeta,R)
\ee
exists for every $\hbar\in \bC^*$, and  forms
an $\hbar$-family of $SL_2(\bC)$-opers. 
\end{conj}

\begin{rem}
\begin{enumerate}
\item 
The existence of the limit is non-trivial, because
the Hermitian metric $h$ blows up as $R\rar 0$.

\item
Unlike the case of non-Abelian Hodge 
correspondence, the Gaiotto limit works
only for a point in the Hitchin section. 
\end{enumerate}
\end{rem}

\begin{thm}[\cite{DFKMMN}]
Gaiotto's conjecture holds for an arbitrary
simple and simply connected complex
algebraic group $G$. 
\end{thm}

Recall that the representation \eqref{rho}
gives a realization of $C$ from its universal covering
space $\bH$ as
 $$
 C \isom \bH\big/\rho\big(\pi_1(C)\big).
 $$
 The representation $\rho$ lifts to $SL_2(\bR)
 \subset SL_2(\bC)$, and defines a projective
 structure in $C$ subordinating its complex
 structure coming from $\bH$. 
 This projective structure is what we call the
 \emph{Fuchsian} projective structure. 

\begin{cor}[Gaiotto correspondence and
quantization \cite{DFKMMN}]
Under the same setting of 
Conjecture~\ref{conj:Gaiotto}, the limit
oper of \eqref{scaling} is given by
\be\label{Gaiotto limit}
\lim_{\substack{R\rar 0,\zeta\rar 0\\
\zeta/R = \hbar}} D(\zeta,R)
= d-\frac{1}{\hbar}\phi({q}) = \nabla^\hbar({q}),
\qquad \hbar\ne 0,
\ee
with respect to the Fuchsian projective coordinate
system. The correspondence
$$
\left(E_0,\phi({q})\right)
\overset{\gam}{\longmapsto}
\left(E_\hbar, \nabla^\hbar({q})\right)
$$
is  biholomorphic, unlike the non-Abelian Hodge
correspondence. 
\end{cor}

\begin{proof}
The key point is that since $E_0$ is made out
of $K_C$, the fiber metric $h$ naturally comes from
the metric of $C$ itself. Hitchin's equations
\eqref{Hitchin's} for ${q}=0$  then become a harmonic
equation for the metric of $C$, and its solution
is given by the constant curvature hyperbolic 
metric. This metric in turn defines the 
Fuchsian projective structure in $C$. For more
detail, we reefer to \cite{O-Paris, DFKMMN}.
\end{proof}

\begin{rem}
The case of an arbitrary simple algebraic group $G$ of the
conjecture, and the whole story of Part 2 for $G$,
have been worked out  in \cite{DFKMMN,OM5}. 
The key point is to use Kostant's TDS of \cite{Kostant}
and transcribe the $SL_2(\bC)$ situation  into 
$G$. 
\end{rem}



\providecommand{\bysame}{\leavevmode\hbox to3em{\hrulefill}\thinspace}

\bibliographystyle{amsplain}

\end{document}